\documentclass[12pt]{amsart}
\usepackage{amsfonts,amsmath,amssymb,color,amscd,amsthm}

\usepackage{amsfonts,amssymb,color,amscd}
\usepackage[T1]{fontenc} 
\usepackage[francais,english]{babel} 
\usepackage{typearea}
\usepackage[all]{xy}
\usepackage{tikz-cd}
\usetikzlibrary{decorations.pathreplacing}

\usepackage{ulem} 
\renewcommand{\emph}[1]{{\it #1}}

\usepackage{enumerate}

\usepackage[backref, colorlinks, linktocpage, citecolor = blue, linkcolor = blue]{hyperref}

\newcommand\iso{\stackrel{\simeq}{\longrightarrow}}
\newcommand\mapsfrom{\mathrel{\reflectbox{\ensuremath{\mapsto}}}}
\newcommand{\dis}{\displaystyle}

\newcommand\N{{\mathbb N}}

\newcommand\R{{\mathbb R}}
\newcommand\A{{\mathbb A}}
\newcommand\Z{{\mathbb Z}}
\newcommand\C{{\mathbb C}}

\newcommand\Pic{\mathrm{ Pic}}

\newcommand\tr{\hbox to 1mm  {${}^t \!  $} }

\newcommand\p{{\mathbb P}}
\newcommand\F{{\mathbb F}}

\renewcommand\k{\mathrm{k}}
\newcommand\kk{\overline{\k}}

\DeclareMathOperator{\Aut}{Aut}

\DeclareMathOperator{\GL}{GL}
\DeclareMathOperator{\PGL}{PGL}
\DeclareMathOperator{\SL}{SL}

\DeclareMathOperator{\id}{id}

\DeclareMathOperator{\Spec}{Spec}

\newtheorem{theorem}{Theorem}
\newtheorem{lemma}{Lemma}[section]
\newtheorem{corollary}[lemma]{Corollary}
\newtheorem{proposition}[lemma]{Proposition}

\newtheorem*{corollary*}{Corollary}
\newtheorem*{corollary**}{Corollary}

\theoremstyle{definition}
\newtheorem{definition}[lemma]{Definition}

\theoremstyle{remark} 
\newtheorem{remark}[lemma]{Remark}
\newtheorem{example}[lemma]{Example}


\title[Exceptional isomorphisms between complements of  affine plane curves]
{Exceptional isomorphisms between\\ complements of affine plane curves} 
\thanks{The authors gratefully acknowledge support by the Swiss National Science Foundation Grants ``Birational Geometry'' PP00P2\_128422 /1  and  ``Curves in the spaces'' 200021\_169508 and by the French National Research Agency Grant ``BirPol'', ANR-11-JS01-004-01. The article was written mainly during the second author's stay in Basel, for one year.}

\author{J\'er\'emy Blanc}
\address{J\'er\'emy Blanc, Universit\"{a}t Basel, Departement Mathematik und Informatik, Spiegelgasse $1$, CH-$4051$ Basel, Switzerland}
\email{jeremy.blanc@unibas.ch}

\author{Jean-Philippe Furter}
\address{Jean-Philippe Furter, Dpt. of Math., Univ. of La Rochelle, av. Cr\'epeau, 17000 La Rochelle, France}
\email{jpfurter@univ-lr.fr}

\author{Mattias Hemmig}
\address{Mattias Hemmig, Universit\"{a}t Basel, Departement Mathematik und Informatik, Spiegelgasse $1$, CH-$4051$ Basel, Switzerland}
\email{mattias.hemmig@gmail.com}

\begin{document}

\maketitle

\vspace{-4mm}

\begin{abstract}
This article describes the geometry of isomorphisms between complements of geometrically irreducible closed curves in the affine plane $\mathbb{A}^2$, over an arbitrary field,
which do not extend to an automorphism of $\mathbb{A}^2$.

We show that such isomorphisms are quite exceptional. In particular, they occur only when both curves are isomorphic to open subsets of the affine line $\mathbb{A}^1$, with the same number of complement points, over any field extension of the ground field.
Moreover, the isomorphism is uniquely determined by one of the curves, up to left composition with an automorphism of $\mathbb{A}^2$, except in the case where the curve is isomorphic to the affine line $\mathbb{A}^1$ or to the punctured line $\mathbb{A}^1 \setminus \{0\}$. If one curve is isomorphic to $\mathbb{A}^1$, then both curves are equivalent to lines. In addition, for any positive integer $n$, we construct a sequence of $n$ pairwise non-equivalent closed embeddings of $\mathbb{A}^1 \setminus \{0\}$ with isomorphic complements. In characteristic~$0$ we even construct infinite sequences with this property.

Finally, we give a geometric construction that produces a large family of examples of non-isomorphic geometrically irreducible closed curves in $\mathbb{A}^2$ that have isomorphic complements, answering negatively the Complement Problem posed by Hanspeter Kraft \cite{Kraft96}. This also gives a negative answer to the holomorphic version of this problem in any dimension $n \geq 2$. The question had been raised by Pierre-Marie Poloni in \cite{Pol16}.
\end{abstract}

\tableofcontents

\section{Introduction}
In the Bourbaki Seminar {\it Challenging problems on affine $n$-space} \cite{Kraft96}, Hanspeter Kraft gives a list of eight basic problems related to the affine $n$-spaces. The sixth one is the following:
\begin{quote}{
Complement Problem. Given two irreducible hypersurfaces $E,F\subset \A^n$ and an
isomorphism of their complements, does it follow that $E$ and $F$ are isomorphic?}\end{quote}

Algebraically, the formulation of this problem is the following: given some base-field $\k$, two irreducible polynomials $P,Q\in \k[x_1,\ldots,x_n]$, and an isomorphism of $\k$-algebras \linebreak
$\varphi\colon\k[x_1,\ldots,x_n,\frac{1}{P}]\iso\k[x_1,\ldots,x_n,\frac{1}{Q}]$, is it true that the $\k$-algebras $\k[x_1,\ldots,x_n]/(P)$ and $\k[x_1,\ldots,x_n]/(Q)$ are isomorphic?

We may restrict ourselves to the case where the isomorphism between the complements does not extend to an automorphism of $\A^n$, or  equivalently when the  isomorphism $\varphi$ does not restrict to an automorphism of $\k[x_1,\ldots,x_n]$.
Indeed, otherwise, the answer to the complement problem is trivially positive.

Recently, Pierre-Marie Poloni gave a negative answer to the problem for any $n\ge 3$ \cite{Pol16}. The construction is given by explicit formulas. There are examples where both $E$ and $F$ are smooth, and examples where $E$ is singular, but $F$ is smooth. This article deals with the case of dimension $n=2$. 
The situation is much more rigid than in dimension $n\geq 3$, as we discuss in Theorem~\ref{RigidityThm}.

We will work over a fixed arbitrary field $\k$ and we will only consider curves, surfaces, morphisms, and rational maps defined over $\k$, unless we explicitly state so (and will then talk about $\kk$-curves, $\kk$-surfaces, $\kk$-morphisms, and $\kk$-rational maps, where $\kk$ denotes the algebraic closure of $\k$.) We recall that two  closed curves $C,D\subset \A^2$ are \emph{equivalent} if there is an automorphism of~$\A^2$ that sends one curve onto the other. Note that equivalent curves are isomorphic. A variety (defined over $\k$) is called \emph{geometrically irreducible} if it is irreducible over $\kk$.
A \emph{line} in $\A^2$ is a closed curve of degree~$1$.

\begin{theorem}\label{RigidityThm}
Let $C \subset \A^2$ be a geometrically irreducible closed curve and let $\varphi \colon \A^2 \setminus C \hookrightarrow \A^2$ be an open embedding. Then, the complement $D \subset \A^2$ of the image of $\varphi$
is also a geometrically irreducible closed curve. Assuming that $\varphi$ does not extend to an automorphism of $\A^2$, the following holds:
\begin{enumerate}[$(1)$]
\item \label{RigidityThmPQ} Both $C$ and $D$ are isomorphic to open subsets of $\A^1$, with the same number of complement points. This means that there exist square-free polynomials $P,Q \in \k[t]$ with the same number of roots in $\k$ and such that 
\[C\simeq \Spec(\k[t,\frac{1}{P}]) \quad \text{and} \quad D\simeq \Spec(\k[t,\frac{1}{Q}]).\]
Moreover, the same result holds for every field extension $\k' / \k$.
\item If $C$ is isomorphic to $\A^1$, then both $C$ and $D$ are equivalent to lines.
\item\label{RigidityThmUniqueUp}
 If $C$ is not isomorphic to $\A^1$ or $\A^1 \setminus \{0\}$, then $\varphi$ is uniquely determined up to a left composition with an automorphism of $\A^2$.
\end{enumerate}
\end{theorem}

\begin{corollary}  \label{Coro:AtMostTwoEq}
If $C \subset \A^2$ is a geometrically irreducible closed curve not isomorphic to $\A^1 \setminus \{0\}$, then there are at most two equivalence classes of closed curves whose complements are isomorphic to $\A^2 \setminus C$. 
\end{corollary}

\begin{corollary}\label{Coro:AtMostOnecounterEx}
Let $C \subset \A^2$ be a geometrically irreducible closed curve. Then there exists at most one closed curve $D \subset \A^2$, up to equivalence, such that $C$ and $D$ are non-isomorphic, but have isomorphic complements.
\end{corollary}

\begin{corollary}\label{Coro:Index12}
Let $C \subset \A^2$ be a geometrically irreducible closed curve, not isomorphic to $\A^1$ or $\A^1 \setminus \{0\}$.
Then, the group $\Aut(\A^2,C)=\{g\in \Aut(\A^2)\mid g(C)=C\}$, which can be naturally identified with a subgroup of $\Aut(\A^2 \setminus C)$, has index $1$ or $2$ in this group.
\end{corollary}

\begin{corollary}\label{Coro:Sing}
If $C \subset \A^2$ is a singular, geometrically irreducible closed curve and $\varphi \colon \A^2 \setminus C \iso \A^2 \setminus D$ is an isomorphism, for some closed curve $D$, then $\varphi$ extends to an automorphism of $\A^2$.
\end{corollary}

Corollary~\ref{Coro:Sing} shows in particular that the Complement Problem for $n=2$ has a positive answer if one of the curves is singular, contrary to the case where $n\ge 3$, as pointed out before. This is also different from the case of $\p^2$, where there exist non-isomorphic geometrically irreducible closed curves with isomorphic complements \cite[Theorem~1]{Bla09}, but where all these curves are necessarily singular (see Proposition~\ref{Prop:InP2Counterexamplessingular} below).

Theorem~\ref{RigidityThm} moreover shows that the Complement Problem for $n=2$ has a positive answer if one of the curves is not rational (this was already stated in \cite[Proposition~$3$]{Kraft96} and does not need all tools of Theorem~\ref{RigidityThm} to be proven, see for instance Corollary~\ref{Cor:NotRat} below). More generally, the answer is positive when one of the curves 
is not isomorphic to an open subset of~$\A^1$. The circle of equation $x^2+y^2=1$ over $\R$ is an example of a smooth rational affine  curve which is not isomorphic to an open subset of~$\A^1$. Note that \cite[Proposition~$3$]{Kraft96} says in addition that
the Complement Problem for $n=2$ and $\k=\C$ has a positive answer if one of the curves has Euler characteristic one; this is also provided by Theorem~\ref{RigidityThm}.

Corollary~\ref{Coro:AtMostTwoEq} describes a situation quite different from the case of dimension
$n \ge 3$, where there are infinitely many hypersurfaces $E\subset \A^n$, up to equivalence, that have isomorphic complements \cite[Lemma 3.1]{Pol16}. It is also in contrast with the case of $\p^2$, where we can find algebraic families of closed curves in $\p^2$, non-equivalent under automorphisms of $\p^2$, that have isomorphic complements (and thus infinitely many if $\k$ is infinite). This follows from a construction in \cite{Cos12}, see Proposition~\ref{Prop:CostaFamilies} below.

All tools necessary to obtain the rigidity result (Theorem~\ref{RigidityThm}) are developped in Section~\ref{Sec:GeomRig}, using some basic results given in Section~\ref{Sec:Preliminaries}. The proof is carried out at the end of Section~\ref{Sec:GeomRig}. It uses embeddings into various smooth projective surfaces and a detailed study of the configuration of the curves at infinity. We study in particular embeddings into  Hirzebruch surfaces that have mild singularities on the boundary and then study blow-ups of these, and completions by unions of trees.

\bigskip

Our second theorem is an existence result which demonstrates the optimality of Theorem~\ref{RigidityThm}.

\begin{theorem} \label{NegativityThm}\item
\begin{enumerate}[$(1)$]
\item\label{Negativity1}
There exists a closed curve  $C\subset \A^2$, isomorphic to $\A^1\setminus \{0\}$, whose complement $\A^2 \setminus C$ admits infinitely many equivalence classes of open embeddings $\A^2\setminus C\hookrightarrow \A^2$ into the affine plane. Moreover, the set of equivalence classes of curves with this property is infinite.

\item\label{Negativity2}
For every integer $n\ge 1$, there exist pairwise non-equivalent closed curves $C_1,\dots,C_n\subset \A^2$, all isomorphic to $\A^1\setminus \{0\}$,  such that the surfaces $\A^2\setminus C_1$, $\dots$, $\A^2\setminus C_n$ are all isomorphic. Moreover, if $\mathrm{char}(\k)=0$, we can find an infinite sequence of pairwise non-equivalent closed curves $C_i\subset \A^2$, $i\in \N$, such that the surfaces $\A^2\setminus C_i$, $i\in \N$, are all isomorphic.

\item\label{Negativity3}
For each polynomial $f\in \k[t]$ of degree $\ge 1$, there exist two non-equivalent closed curves $C,D\subset \A^2$, both isomorphic to $\Spec(\k[t,\frac{1}{f}])$, such that the surfaces $\A^2\setminus C$ and  $\A^2\setminus D$ are isomorphic. Moreover, the set of equivalence classes of the curves $C$ in such pairs $(C,D)$ is infinite.
\end{enumerate}
\end{theorem}

A constructive proof of Theorem~\ref{NegativityThm} is given in Section~\ref{SecFamilies}. We use explicit equations and work with birational maps which either preserve one projection $\A^2\to \A^1$ or are compositions of a small number of them.

Note that parts~\eqref{Negativity1} and \eqref{Negativity2} of Theorem~\ref{NegativityThm} concern closed curves $C\subset \A^2$ isomorphic to $\A^1\setminus \{0\}$. This is the only case where there can be more than two curves $D\subset \A^2$, up to equivalence,
such that $\A^2\setminus C$ is isomorphic to $\A^2\setminus D$ (Corollary~\ref{Coro:AtMostTwoEq}).
When $\k=\C$, the classification of curves $C\subset \A^2$  isomorphic to $\A^1\setminus \{0\}$ has a long history. A line $L\subset \A^2$ intersecting $C$ with multiplicity $0$, resp.~$1$, is called a ``very good asymptote'', resp.~``good asymptote'', and the curves of $\C^2$ isomorphic to $\C^*$ admitting such asymptotes have been classified in seven families in \cite[Theorem 8.2]{CNKR09} (the first two families corresponding to the ``very good'' asymptote were already classified in \cite{Kal}).
These curves also belong to the longer list of \cite{BoZo10}, which covers curves homeomorphic to $\C^*$ with some additional regularity condition. The curves isomorphic to $\C^*$ in $\C^2$ without any ``good'' or ``very good'' asymptote are called ``sporadic'' and studied in \cite{KR,KPR}, where some restrictions on the possibilities are given, in order to possibly give a future classification.

The examples of curves isomorphic to $\A^1\setminus \{0\}$ that we give to prove Theorem~\ref{NegativityThm}\eqref{Negativity1}-\eqref{Negativity2} are all given by $xy^d+b(y)=0$ for some $d\ge 1$ and some polynomial $b(y)\in \k[y]$ with $b(0)\not=0$ (see \S\ref{SubSec:A1minus0}). These
form a subfamily of the very first case of the classification of \cite[Theorem 8.2]{CNKR09}, which are the curves given by $y^m+(xy^d+b(y))^n=0$, where $m,n\ge 1$ are coprime, and $b(y)\in \k[y]$ satisfies  $b(0)\not=0$. This family also appears in the proposition on page 143 of  \cite{Kal}.

It would be interesting to study the more complicated cases of the partial classification  and in general to determine for a curve $C\subset \A^2$ isomorphic to $\A^1\setminus \{0\}$ all curves $D\subset \A^2$ such that $\A^2\setminus C$ and $\A^2\setminus D$ are isomorphic (note that $D$ is again isomorphic to $\A^1\setminus \{0\}$ by Theorem~\ref{RigidityThm}). The structure of the group  $\Aut(\A^2\setminus C)$ 
also seems to be a nice subject
for investigation. Let us at least mention that according to \cite[Theorem 2]{BS15}, the natural subgroup $\Aut(\A^2,C)$ is always algebraic and even finite in most cases.

\bigskip

We then give counterexamples to the Complement Problem in dimension~$2$:

\begin{theorem}\label{Theorem:ComplementProblemAllFields}
 There exist two geometrically irreducible closed curves $C,D\subset \A^2$ which are not isomorphic, but whose complements $\A^2\setminus C$ and $\A^2\setminus D$ are isomorphic. Furthermore, these two curves can be chosen of degree $7$ if the field admits more than $2$ elements and of degree $13$ if the field has $2$ elements.
\end{theorem}

The proof  is given in Section~\ref{Sec:NonIsoCurves}.
We first establish Proposition~\ref{proposition: existence of two curves with isomorphic complements and prescribed rings of functions} (mainly via blow-ups of points on singular curves in $\p^2$) which asserts that, for each polynomial $P\in \k[t]$ of degree $d \geq 1$ and each $\lambda\in \k$ with $P( \lambda ) \not=0$, there exist two closed curves $C,D\subset \A^2$ of degree $d^2-d+1$ such that $\A^2\setminus C$ and $\A^2\setminus D$ are isomorphic and such that the following isomorphisms hold:
\[ C\simeq \Spec \Big( \k[t,\frac{1}{P}] \Big)  \text{ and } D \simeq \Spec \Big( \k[t,\frac{1}{Q}] \Big), \text{ where } Q(t)=P \Big( \lambda +\frac{1}{t} \Big) \cdot t^{\deg(P)}.  \]
Then, the proof of Theorem~\ref{Theorem:ComplementProblemAllFields} follows by providing an appropriate pair $(P, \lambda)$ for every field. The case of infinite fields is quite easy. Indeed, if $\k$ is infinite and $P\in \k[t]$ is a polynomial with at least $3$ roots in $\overline{\k}$, then $\Spec(\k[t,\frac{1}{P}])$ and $\Spec(\k[t,\frac{1}{Q}])$ are not isomorphic, for a general element $\lambda \in \k$ (Lemma~\ref{lemma: P and Q are not isomorphic for a general lambda}). This shows that the isomorphism type of counterexamples to the Complement Problem is as large as possible (indeed, by Theorem~\ref{RigidityThm}\eqref{RigidityThmPQ}, any curves $C,D \subset \A^2$ providing a counterexample to the Complement Problem are necessarily isomorphic to open subsets of $\A^1$ with at least three complement $\kk$-points).

\bigskip

We finish this introduction by presenting some easy consequences of Theorem~\ref{Theorem:ComplementProblemAllFields} that are further elaborated in Section~\ref{Related}:

$(i)$ The negative answer to the Complement Problem for $n=2$ directly gives a negative answer for any $n\ge 3$ (Proposition~\ref{Prop:Products}): Our construction produces, for each $n\ge 3$, two geometrically irreducible smooth closed hypersurfaces $E,F \subset \A^n$ which are not isomorphic, but whose complements $\A^n\setminus E$ and $\A^n\setminus F$ are isomorphic (Corollary~\ref{Coro:HighDimExp}). All the hypersurfaces constructed this way are isomorphic to $\A^{n-2}\times C$ for some open subset $C\subset \A^1$. This does not allow us to give singular examples like those of \cite{Pol16}, but provides a different type of example.

$(ii)$ Choosing $\k=\C$, our construction gives families of closed complex curves $C,D\subset \C^2$ whose complements are biholomorphic (because they are isomorphic as algebraic varieties), but which are not themselves biholomorphic (Proposition~\ref{Prop:Curvesholomorphic}).
From this there directly follows the existence of algebraic hypersurfaces $E,F\subset \C^n$ which are complex manifolds that are not biholomorphic, but have biholomorphic complements, for every $n \ge 2$ (Corollary~\ref{Coro:Holhighdimension}). This answers a question asked in \cite{Pol16}. Note that in the counterexamples of \cite{Pol16}, if both hypersurfaces are smooth, then they are always biholomorphic (even if they are not isomorphic as algebraic varieties).

\bigskip

The authors thank Hanspeter Kraft, Lucy Moser-Jauslin and Pierre-Marie Poloni for interesting discussions during the preparation of this article.

\section{Preliminaries}\label{Sec:Preliminaries}

In the sequel, $\k$ is an arbitrary field and $\kk$ its algebraic closure. Unless otherwise specified, all varieties of dimension at least one are  $\k$-varieties, i.e.~algebraic varieties defined over $\k$, or equivalently $\overline{\k}$-varieties with a $\k$-structure. When we say for example \emph{rational}, resp.~\emph{isomorphic}, we mean $\k$-rational, resp.~$\k$-isomorphic (which means that the maps are defined over $\k$). Nevertheless, we will often have to consider $\kk$-varieties, but we will then always state so explicitly. A variety is called \emph{geometrically rational,} resp. \emph{geometrically irreducible,} if it is rational, resp. irreducible, after the extension to $\kk$. When dealing with ``points'' (but also with ``base-points'' or ``complement points'') we will always specify $\k$-points or $\kk$-points. Finally, let us recall that a $\kk$-base-point of a $\kk$-birational map $f \colon X \dasharrow Y$, where $X$ and $Y$ are smooth projective $\kk$-surfaces, is either \emph{proper}, when it belongs to $X$, or \emph{infinitely near}, when it does not belong to $X$, but to a surface obtained from $X$ via a finite number of blow-ups. If we assume furthermore that $f,X,Y$ are defined over $\k$, then a $\k$-base-point of $f$ is defined in the following obvious way: it is either a proper $\kk$-base-point defined over $\k$, or it is an infinitely near $\kk$-base-point of $f$ which is a $\k$-point of a surface obtained from $X$ via a finite number of blow-ups of $\k$-points. Of course, there is no reason for a birational map $f \colon X \dasharrow Y$ to admit a $\k$-base-point. For example, when $\k=\F_2$ the birational involution of $\p^2$ given by $[x:y:z]\mapsto [x^2+y^2+yz: xz+y^2+z^2: x^2+xy+z^2]$ admits no $\k$-base-point (but has three base-points over $\F_8=\F_3[u]/(u^3+u+1)$, namely $[1:u:u^2+u+1]$, $[u:u^2+u+1:1]$ and $[u^2+u+1:1:u]$). Similar examples of degree $5$ for $\k=\R$ are classical and can be found in \cite[Example 3.1]{BM}. Also, a closed curve in $\A^2$ does not necessarily admit a $\k$-point. For example, the geometrically irreducible closed curve of equation $x^2+y^2+1= 0$ admits no $\R$-point.

Working over an algebraically closed field, every birational map $\varphi\colon X\dasharrow Y$ between two smooth projective irreducible surfaces $X$ and $Y$ admits a  resolution,
which consists of two birational morphisms $\eta \colon Z\to X$ and $\pi\colon Z\to Y$, where $Z$ is a smooth projective irreducible surface, such that the following diagram is commutative.
\[\xymatrix@R=4mm@C=2cm{& Z \ar[ld]_{\eta}\ar[rd]^{\pi} \\
X \ar@{-->}[rr]^{{\varphi}} && Y} \]
Let us also recall that a birational morphism between two smooth projective irreducible surfaces is a composition of finitely many blow-downs.
We can moreover choose this resolution to be \emph{minimal}, which corresponds to asking that no irreducible curve of $Z$ of self-intersection $(-1)$ be contracted by both $\eta$ and $\pi$. The morphism $\eta$ is obtained by blowing up all base-points in $X$ of $\varphi$. Analogously $\pi$ is obtained by blowing up all base-points in $Y$ of $\varphi^{-1}$. In Lemma~\ref{Lem:IsoComplRat}\eqref{Resolution}, we will prove that under some additional hypotheses (satisfied by all birational maps that we will consider), such a miminal resolution also exists over an arbitrary field $\k$, and that moreover the morphisms $\eta$ and $\pi$ are obtained by sequences of blow-ups of $\k$-points (which may be proper or infinitely near).

\subsection{Basic properties} \label{subsection: basic properties}
In order to study isomorphisms between affine surfaces, it is often interesting to see the affine surfaces as open subsets of projective surfaces and then to see the isomorphisms as birational maps between the projective surfaces.
Recall that a rational map $\varphi \colon X \dasharrow Y$ between smooth projective irreducible surfaces is defined on an open subset $U\subset X$ such that $F= X \setminus U$ is finite. If $C$ is an irreducible curve of the surface $X$, its image is defined by $\varphi (C) :=\overline{\varphi ( C \setminus F) }$. We then say that $C$ is \emph{contracted by $\varphi$} if $\varphi(C)$  is a point. The aim of this section is to establish Proposition~\ref{Prop:Threecasescontraction}, that we often use in the sequel. Its proof relies on some easy results that we begin by recalling:
Proposition~\ref{Prop:Jac},
Corollary~\ref{corollary: open embedding of p^2 minus a finite set of curves into p^2} and Lemma~\ref{Lem:IsoComplRat}.

We begin with the following definition, that we will frequently use, in particular to extend birational maps of $\A^2$ to birational maps of $\p^2$:

\begin{definition} \label{Def:Plane}
The  morphism
\[ \begin{array}{rcl} \A^2&\hookrightarrow & \p^2 \\ (x,y)&\mapsto &[x:y:1] \end{array} \]
is called the \emph{standard embedding}.
It induces an isomorphism $\A^2\iso \p^2\setminus L_{\infty}$, where $L_{\infty}\subset \p^2$ 
denotes the \emph{line at infinity} given by $z=0$.
\end{definition}

With this embedding every line in~$\A^2$, given by an equation $ax+by=c$ where $a,b,c$ are elements of $\k$ and  $a,b$ are not both zero, is the restriction of a line of~$\p^2$, given by the equation 
$ax+by=cz$ and distinct from $L_{\infty}$.

\begin{definition}
For each birational map $\varphi\colon \p^2\dasharrow \p^2$, we define $J_\varphi\subset \p^2$ to be the reduced curve given by the union of all irreducible $\kk$-curves contracted by $\varphi$.
\end{definition}

\begin{proposition}  \label{Prop:Jac}
Let $\varphi\colon \p^2\dasharrow \p^2$ be a birational map. 

\begin{enumerate}[$(1)$]
\item\label{Jdefinedk}
The curve $J_\varphi$ is defined over $\k$, i.e.~is the zero locus of a homogeneous polynomial $f\in \k[x,y,z]$.
\item\label{IsoJvarphi}
The restriction of $\varphi$ induces an isomorphism $\p^2\setminus J_\varphi \to \p^2\setminus J_{\varphi^{-1}}$.  Moreover, the number of irreducible components of $J_\varphi$ and $J_{\varphi^{-1}}$ over $\kk$ are equal.
\end{enumerate}
\end{proposition}

\begin{proof}
$(\ref{Jdefinedk})$. The maps $\varphi$ and $\varphi^{-1}$ may be written in the form
\[ \begin{array}{rllll}
\varphi\colon& [x:y:z]\mapsto& [s_0(x,y,z):s_1(x,y,z):s_2(x,y,z)]& \text{ and}\\\varphi^{-1}\colon& [x:y:z]\mapsto& [q_0(x,y,z):q_1(x,y,z):q_2(x,y,z)],\end{array}\]
where $s_0,s_1,s_2\in \k[x,y,z]$  (as well as $q_0,q_1,q_2$) are homogeneous polynomials of the same degree and with no common factor. Since $\varphi^{-1}\circ \varphi = \id$, there exists a homogeneous polynomial $f \in \k [x,y,z]$ such that $q_0(s_0,s_1,s_2)=xf$, $q_1(s_0,s_1,s_2)=yf$, $q_2(s_0,s_1,s_2)=zf$. We now observe that $J_\varphi$ is the zero locus of $f$. Indeed, the polynomial $f$ is zero along an irreducible $\kk$-curve if and only if this curve is sent by $\varphi$ to a base-point of $\varphi^{-1}$. In characteristic zero, note that $J_\varphi$ is also the zero locus of the Jacobian determinant associated to $\varphi$.

$(\ref{IsoJvarphi})$ By extending the scalars, we may assume that $\k=\kk$ is algebraically closed. We take a minimal resolution of $\varphi$, with the commutative diagram
\[\xymatrix@R=4mm@C=2cm{& X\ar[ld]_{\eta}\ar[rd]^{\pi} \\
\p^2\ar@{-->}[rr]^{{\varphi}} && \p^2} \]
where $\eta$ and $\pi$ are birational morphisms. The morphism $\eta$, resp.~$\pi$, is the sequence of blow-ups of the base-points of~${\varphi}$, resp.~$\varphi^{-1}$.

By computing the Picard rank of~$X$, we see that $\eta$ and $\pi$ contract the same number of irreducible curves of~$X$. Let $n$ be this number. We then denote by $E \subset X$, resp.~$F\subset X$, the union of the $n$ irreducible curves contracted by $\eta$, resp.~$\pi$. The map $\varphi$ then restricts to an isomorphism
\[\p^2\setminus \eta ( E \cup F)\iso  \p^2 \setminus \pi ( E \cup F).\]

We now show that $\eta ( E \cup F) = \eta  (F)$. Since $\eta(E)$ consists of finitely many points, it suffices to see that these are contained in the curves of $\eta(F)$. Each point $p$ of $\eta(E)$ corresponds to a connected component of $E$, which contains at least one $(-1)$-curve $\mathcal{E}\subset E$. The curve $\mathcal{E}$ is not contracted by $\pi$, by minimality, and hence is sent by $\pi$ onto a curve $\pi(\mathcal{E})\subset\p^2$ of self-intersection $\ge 1$. This implies that $\mathcal{E}$ intersects $F$ and thus $p\in \eta(F)$. 
We similarly get that $\pi (E \cup F) = \pi (E)$, and obtain that $\varphi$ restricts to an isomorphism
\[\p^2\setminus \eta(F)\iso \p^2\setminus \pi(E).\]
Since $\eta(F)$ is a closed curve in $\p^2$ whose irreducible components are contracted by $\varphi$, we have $\eta(F)=J_{\varphi}$. Similarly, we get $\pi(E)=J_{\varphi^{-1}}$. Moreover, the number of $\kk$-irreducible components of $\eta(F)$ is equal to the number of irreducible components of $\overline{F\setminus E}$, which is equal to the number of irreducible components of $\overline{E\setminus F}$. This completes the proof.
\end{proof}

\begin{corollary} \label{corollary: open embedding of p^2 minus a finite set of curves into p^2}
Let $\Gamma \subset \p^2$ be a closed curve and $\varphi \colon \p^2 \setminus  \Gamma \hookrightarrow \p^2$ an open embedding. Then the complement of $\varphi(\p^2 \setminus  \Gamma )$ is a closed curve $\Delta\subset \p^2$ with the same number of irreducible components over $\kk$ as $\Gamma$.
\end{corollary}

\begin{proof}
Let $\hat{\varphi} \colon \p^2 \dasharrow \p^2$ be the birational map induced by $\varphi$. Proposition~\ref{Prop:Jac} implies that $J_{\hat{\varphi}} \subset \Gamma$, that $J_{\hat{\varphi}}$ and $J_{\hat\varphi^{-1}}$ have the same number of irreducible components over $\kk$, and that $\hat{\varphi}$ induces an isomorphism $\p^2\setminus J_{\hat\varphi}\iso \p^2\setminus J_{\hat\varphi^{-1}}$.

If $J_{\hat{\varphi}} = \Gamma$, the proof is finished. Otherwise, $\Gamma'=\Gamma\setminus J_{\hat\varphi}$ is a closed curve of $\p^2\setminus J_{\hat\varphi}$, which has the same number of irreducible components over $\kk$ as the closed curve $\Delta'=\hat\varphi(\Gamma')$ of $\p^2\setminus J_{\hat\varphi^{-1}}$. The result follows with $\Delta=\Delta' \cup J_{\hat\varphi^{-1}}$.
\end{proof}

\begin{lemma}  \label{Lem:IsoComplRat} 
Let $\varphi \colon X \dasharrow Y$ be a birational map between two smooth projective surfaces that restricts to an isomorphism $U=X \setminus C \iso Y \setminus D=V$, where $C$, resp. $D$, is the union of geometrically irreducible closed curves $C_1, \ldots, C_r$ in $X$, resp. $D_1, \ldots, D_s$ in $Y$. Then, the following holds.

\begin{enumerate}[$(1)$]
\item\label{BasePtsKRat}
All $\kk$-base-points of $\varphi$, resp. $\varphi^{-1}$, are $\k$-rational and belong to $C$, resp. $D$.

\item \label{Resolution}
The map $\varphi$ admits a minimal resolution which is given by birational morphisms $\eta \colon Z \to X$ and $\pi \colon Z \to Y$, which are blow-ups of the base-points of $\varphi$ and $\varphi^{-1}$ respectively,  as shown in the following diagram:
\[\xymatrix@R=4mm@C=2cm{& Z\ar[ld]_{\eta}\ar[rd]^{\pi} \\
X\ar@{-->}[rr]^{\varphi} && Y\\
 U \ar@{^{(}->}[u]\ar^{\simeq}[rr]&&   V.\ar@{^{(}->}[u] }\]
 
\item  \label{inverse images of U and V are equal}
In the above resolution, we have $\eta^{-1}(U)=\pi^{-1}(V)$.

\item \label{ComplementsOutside}
For each $i\in\{1,\dots,r\}$, there exists
$j\in \{1,\dots,s\}$ such that either $\varphi$ restricts to a birational map $C_i\dasharrow D_j$ or $\varphi(C_i)$ is a $\k$-point of $D_j$. In this latter case, the curve $C_i$ is rational $($over $\k)$.
\end{enumerate} 
\end{lemma}

\begin{proof}We argue by induction on the total number of $\kk$-base-points of $\varphi$ and $\varphi^{-1}$. If there is no such base-point, then $\varphi$ is an isomorphism and everything follows.

Suppose now that $q\in Y$ is a proper $\kk$-base-point of $\varphi^{-1}$. As $\varphi$ induces an isomorphism $U\iso V$, we have $q \in D_j (\kk)$  for some $j\in \{1,\dots,s\}$. There is moreover an irreducible $\kk$-curve of $Y$ contracted by $\varphi$ onto $q$,  which is then equal to $C_i$ for some $i\in \{1,\dots,r\}$. Since $C_i$ is defined over $\k$, so is its image (the generic point of $C_i$ is defined over $\k$ and is sent onto the $\k$-point $q$), i.e.~$q$ is $\k$-rational. Let $\varepsilon \colon \hat{Y} \to Y$ be the blow-up of $q$ and let $E\subset \hat{Y}$ be the exceptional divisor (which is isomorphic to $\p^1$). The birational map $\hat\varphi =  \varepsilon^{-1} \circ \varphi \colon X \dasharrow \hat Y$ induces an isomorphism $U\iso \hat{V}$, where $\hat{V}=\varepsilon^{-1}(V)= \hat Y\setminus (\tilde{D}_1\cup \dots \cup \tilde{D}_s\cup E)$, and where $\tilde{D}_i \subset \hat{Y}$ is the strict transform of $D_i$ for $i=1,\dots,s$.
The $\kk$-base-points of $\hat{\varphi}^{-1}$ correspond to the $\kk$-base-points of $\varphi^{-1}$ from which the point $q$ is removed and the $\kk$-base-points of $\hat{\varphi}$ coincide with the $\kk$-base-points of $\varphi$.

We may thus apply the induction hypothesis and obtain assertions \eqref{BasePtsKRat}--\eqref{ComplementsOutside} for $\hat\varphi$. Denoting by $\hat\eta\colon Z\to X$ and $\hat \pi \colon Z\to \hat{Y}$ the blow-ups of the base-points of $\hat\varphi$ and $\hat\varphi^{-1}$ respectively (which give the resolution of $\hat\varphi$ as in \eqref{Resolution}), we obtain \eqref{BasePtsKRat}--\eqref{Resolution} for $\varphi$
with $\eta=\hat\eta$, $\pi=\varepsilon\hat\pi$. Assertion~\eqref{inverse images of U and V are equal} is given by  $\eta^{-1}(U)=\hat\eta^{-1}(U)\stackrel{\eqref{inverse images of U and V are equal}\text{ for }\hat\varphi}{=}\hat\pi^{-1}(\hat{V})=\hat\pi^{-1}(\epsilon^{-1}(V))=\pi^{-1}(V)$. Assertion~\eqref{ComplementsOutside} follows from the assertion
for $\hat\varphi$ and from the fact that $\varepsilon$ restricts to a birational morphism $\tilde{D}_i \to D_i$ for each $i$, and sends  $E \simeq \p^1$ onto a $\k$-point of $D_j$.

In the case where $\varphi^{-1}$ admits no $\kk$-base-point, a
symmetric argument can be applied to $\varphi^{-1}$ by starting with a proper $\kk$-base-point of $\varphi$.
\end{proof}

In the sequel, we will frequently use the following
result.

\begin{proposition}  \label{Prop:Threecasescontraction}
Let $C\subset \A^2$ be a geometrically irreducible closed curve and let $\varphi\colon \A^2\setminus C \hookrightarrow \A^2$ be an open embedding. Then, there exists a geometrically irreducible closed curve $D \subset \A^2$ such that $\varphi(\A^2\setminus C)=\A^2\setminus D$.
Denote by $\overline{C}$ and $\overline{D}$ the closures of $C$ and $D$ in $\p^2$, using the standard  embedding of Definition~$\ref{Def:Plane}$. Denote also  by $L_{\infty} = \p^2 \setminus \A^2$ the line at infinity and by  $\hat{\varphi} \colon \p^2 \dasharrow \p^2$ the birational map induced by $\varphi$.
Then, one of the following three possibilities holds:

\begin{enumerate}[$(1)$]
\item\label{Case1:AutA2}
We have $\hat{\varphi} ( \overline{C} ) = \overline{D}$. Then, the map $\varphi$ extends to an automorphism of~$\A^2= \p^2\setminus L_{\infty}$ that sends $C$ onto $D$.

\item\label{Case2:Dline}
We have $\hat{\varphi} ( \overline{C} ) =L _{\p^2}$. Then, the curve $D$ is a line in $\A^2$, i.e.~$\overline{D}$ is a line in~$\p^2$ and $\varphi$ extends to an isomorphism $\A^2=\p^2\setminus L_{\infty}\iso\p^2\setminus \overline{D}$ that sends $C$ onto $L_{\infty} \setminus \overline{D}$.
In particular, $C$ is equivalent to a line.

\item\label{Case3:Ccontracted}
The map $\hat\varphi$ contracts the curve $\overline{C}$ to a $\k$-point of $\p^2$. Then, the curve $\overline{C}$ $($and therefore, also the curve $C)$  is a rational curve $($i.e.~is $\k$-birational to $\p^1$$)$. 
\end{enumerate}
\end{proposition}

\begin{proof}
The restriction of $\hat\varphi$ to $\p^2\setminus (L_{\infty} \cup \overline{C} )=\A^2\setminus C$ gives the open embedding $\varphi\colon \A^2\setminus C\hookrightarrow \A^2\hookrightarrow \p^2$. By Corollary~\ref{corollary: open embedding of p^2 minus a finite set of curves into p^2}, we obtain an isomorphism  $ \p^2 \setminus (L_{\infty} \cup \overline{C} )  \iso \p^2 \setminus \Delta$, for some curve $\Delta \subset \p^2$, which is the union of two $\kk$-irreducible closed curves of $\p^2$. Since $L_{\infty}$ is included in $\Delta$, there exists an irreducible closed $\kk$-curve $D$ of $\A^2$ such that $\Delta= L_{\infty} \cup \overline{D}$. As a conclusion, the restriction of $\hat\varphi$ at the source and the target induces an isomorphism
\[\p^2 \setminus ( L_{\infty} \cup \overline{C} )  \iso \p^2 \setminus  ( L_{\infty} \cup \overline{D} ).  \]
It follows that $\varphi ( \A^2 \setminus C ) = \A^2 \setminus D$. The equality $D = \A^2 \setminus \varphi ( \A^2 \setminus C )$ proves that the curve $D$ is defined over $\k$ and is therefore geometrically irreducible. By Lemma~\ref{Lem:IsoComplRat}\eqref{ComplementsOutside}, one of the following three possibilities holds:
\begin{enumerate}
\item[$(\ref{Case1:AutA2})$] We have $\hat \varphi ( \overline{C} ) = \overline{D}$. Hence, the restriction of $\hat \varphi$ at the source and the target provides an automorphism of $\A^2=\p^2\setminus L_{\infty}$ (Proposition~\ref{Prop:Jac}).

\item[$(\ref{Case2:Dline})$] We have $\hat \varphi ( \overline{C} ) = L_{\infty}$. Then, the restriction of $\hat \varphi$ at the source and the target provides an isomorphism $\p^2 \setminus L_{\infty}\iso\p^2 \setminus \overline{D}$ (again by Proposition~\ref{Prop:Jac}). Since the Picard group of~$\p^2\setminus \Gamma$ is isomorphic to $\Z/\deg(\Gamma)\Z$, for each irreducible curve $\Gamma$, the curve  $\overline{D}$ must be a line in $\p^2$.

\item[$(\ref{Case3:Ccontracted})$] The map $\hat\varphi$ contracts the curve $\overline{C}$ to a $\kk$-point of $\p^2$. Then, by Lemma~\ref{Lem:IsoComplRat}\eqref{ComplementsOutside} this point is necessarily a $\k$-point and the curve $\overline{C}$ is $\k$-rational.
\qedhere
\end{enumerate}
\end{proof}

\begin{corollary}\label{Cor:NotRat}
Let $C\subset \A^2$ be a geometrically irreducible closed curve. If $C$ is not  rational 
 $($i.e.~not $\k$-birational to $\p^1)$, then every open embedding $\A^2\setminus C\hookrightarrow \A^2$ extends to an automorphism of~$\A^2$.
\end{corollary}
\begin{proof}
This follows from Proposition~\ref{Prop:Threecasescontraction} and the fact that cases~$(\ref{Case2:Dline})$-$(\ref{Case3:Ccontracted})$ occur only when $C$ is rational.
\end{proof}

\begin{remark}
It follows from Corollary~\ref{Cor:NotRat} that the automorphism group $\Aut(\A^2 \setminus C)$, where $C$ is a non-rational geometrically irreducible closed curve, may be identified with the group $\Aut (\A^2 ,C)$ of automorphisms of $\A^2$ preserving $C$. By \cite[Theorem~2]{BS15}, this group is finite (and in particular conjugate to a subgroup of $\GL_2 (\k)$ if $\mathrm{char}(\k)=0$, as one can deduce from \cite[Theorem 5]{DaniGiz}, \cite[\S 6.2, Proposition 21]{Serre} or from \cite[Theorem 4.3]{Kambayashi}). 
For a general discussion on the group $\Aut (\A^2 \setminus C)$, where $C$ is a geometrically irreducible closed curve, see Section~\ref{Automorphisms-of-complements-of-curves.subsec} below.

\end{remark}
We find it interesting to prove that case $(\ref{Case3:Ccontracted})$ of Proposition~\ref{Prop:Threecasescontraction}  occurs only when $\overline{C}$ intersects $L_{\infty}$ in at most two $\kk$-points, even if this will not be used in the sequel.
\begin{corollary}
If $C\subset \A^2$ is a
geometrically irreducible closed
curve such that $\overline{C}$ intersects $L_{\infty}=\p^2\setminus \A^2$ in at least three $\kk$-points, then every open embedding $\A^2\setminus C\hookrightarrow \A^2$ extends to an automorphism of $\A^2$.
\end{corollary}
\begin{proof}

We may assume that $\k=\kk$. Assume by contradiction that the extension $\hat\varphi\colon \p^2\dasharrow \p^2$ does not restrict to an automorphism of $\A^2$. By Proposition~\ref{Prop:Threecasescontraction}, the curve $\overline{C}$ is contracted by $\hat\varphi$ (because $C$ is not equivalent to a line, so $(\ref{Case2:Dline})$ is impossible). We recall that $\hat\varphi$ restricts to an isomorphism $\A^2\setminus C=\p^2\setminus (L_{\infty}\cup \overline{C})\iso \A^2\setminus D=\p^2\setminus (L_{\infty}\cup \overline{D})$ (Proposition~\ref{Prop:Threecasescontraction}) and that $\overline{C}\subset J_{\hat{\varphi}}\subset L_{\infty}\cup \overline{C}$, $J_{\hat{\varphi}^{-1}}\subset L_{\infty}\cup \overline{D}$, where $J_{\hat\varphi}$, $J_{\hat{\varphi}^{-1}}$ have the same number of irreducible components (Proposition~\ref{Prop:Jac}). We take a minimal resolution of $\hat\varphi$ which yields a commutative diagram
\[\xymatrix@R=4mm@C=2cm{& X\ar[ld]_{\eta}\ar[rd]^{\pi} \\
\p^2\ar@{-->}[rr]^{{\hat\varphi}} && \p^2.} \]
We first observe that the strict transforms $\tilde{L}_{\p^2},\tilde{C}\subset X$ of $L_{\infty},\overline{C}$ by $\eta$ intersect in at most one point.
Indeed, otherwise the curve $\tilde{L}_{\p^2}$ would not be contracted by $\pi$, because $\pi$ contracts $\tilde{C}$, and is sent onto a singular curve, which then has to be $\overline{D}$. We get $J_{\hat\varphi}=\overline{C}$, $J_{\hat\varphi^{-1}}=L_{\infty}$ and get an isomorphism $\p^2\setminus \overline{C}\to \p^2\setminus L_{\infty}$, which is
impossible, because $\overline{C}$ has degree at least $3$.

Secondly, the fact that $\tilde{L}_{\p^2},\tilde{C}\subset X$ intersect in at most one point implies that $\eta$ blows up all points of $\overline{C}\cap L_{\infty}$, except at most one. Since $J_{\hat\varphi^{-1}}\subset D\cup L_{\infty}$, there are at most two $(-1)$-curves contracted by $\eta$. But $L_{\infty}$ and $\overline{C}$ intersect in at least three points, so we obtain exactly two proper base-points of $\hat\varphi$, corresponding to exactly two $(-1)$-curves $E_1, E_2\subset X$ contracted to two points $p_1,p_2\in \overline{C}\cap L_{\infty}$ by $\eta$. Moreover, the identity $J_{\hat\varphi^{-1}}=D\cup L_{\infty}$ implies that $J_{\hat\varphi}=C \cup L_{\infty}$ (Proposition~\ref{Prop:Jac}). We write $E_i'=\overline{\eta^{-1}(p_i)\setminus E_i}$ and find that $\pi$ contracts $F=E_1'\cup E_2'\cup \tilde{C}\cup \tilde{L}_{\p^2}$.

We now show that $E_i\cdot F\ge 2$, for $i=1,2$, which will imply that $\pi(E_i)$ is a singular curve for $i=1,2$, and lead to a contradiction since $E_1,E_2$ are sent onto $L_{\infty}$ and $\overline{D}$ by $\pi$. As $E_i\cup E_i'=\eta^{-1}(p_i)$, it is a tree of rational curves, which intersects both $\tilde{C}$ and $\tilde{L}_{\p^2}$ since $p_i\in \overline{C}\cap L_{\infty}$. If $E_i'$ is empty, then $E_i\cdot \tilde{C}\ge 1$ and $E_i\cdot \tilde{L}_{\p^2}\ge 1$, whence $E_i\cdot F\ge 2$ as we claimed. If $E_i'$ is not empty, then $E_i\cdot E_i'\ge 1$. The only possibility to get $E_i\cdot F\le 1$ would thus be that $E_i\cdot E_i'=1$, $E_i\cdot \tilde{C}=E_i\cdot \tilde{L}_{\p^2}=0$. The equality $E_i\cdot E_i'=1$ implies that $E_i'$ is connected, and $E_i\cdot \tilde{C}=E_i\cdot \tilde{L}_{\p^2}=0$ implies that $\tilde{C}\cdot E_i'\ge 1$ and $\tilde{L}_{\p^2}\cdot E_i'\ge 1$. Since $\tilde{L}_{\p^2}$ and $\tilde{C}$ intersect in a point not contained in $E_i'$, it follows that $F$ contains a loop and thus cannot be contracted.\end{proof}

\begin{remark}
In case~$(\ref{Case3:Ccontracted})$ of Proposition~\ref{Prop:Threecasescontraction}, it is possible that $\overline{C}$ intersects the line $L_{\infty}$ in two $\kk$-points. This is the case in most of our examples (see for example Lemma~\ref{lemma:SL2} or Lemma~\ref{lemma: infinitely many non-equivalent curves having the same complement in characteristic 0}). The case of one point is of course also possible (see for instance Lemma~\ref{Lemm:Lines}$(\ref{AutA2minusline})$).
\end{remark}

We will also need the following basic algebraic result.

\begin{lemma}  \label{Lemm:RingInvertible}
Let $f\in \k[x,y]$ be a polynomial, irreducible over $\kk$, and let $C\subset \A^2$ be the curve given by $f=0$. Then, the ring of functions on $\A^2\setminus C$ and its subset of invertible elements are equal to
\[\mathcal{O}(\A^2\setminus C)=\k[x,y,f^{-1}]\subset \k(x,y),\ \mathcal{O}(\A^2\setminus C)^{*}=\{\lambda f^n\mid \lambda\in \k^*, n\in\mathbb{Z}\}.\]
In particular, every automorphism of~$\A^2\setminus C$ permutes the fibres of the morphism \[\A^2\setminus C\to \A^1\setminus \{0\}\] given by $f$.
\end{lemma}

\begin{proof}
The field of rational functions of~$\A^2\setminus C$ is equal to $\k(x,y)$. We may write any element of this field as
$u/v$, where $u,v \in \k[x,y]$ are coprime polynomials, $v \not = 0$. The rational function is regular on $\A^2\setminus C$ if and only if $v$ does not vanish on any $\kk$-point of~$\A^2\setminus C$. This means that $v =\lambda f^n$, for some $\lambda\in \k^*$, $n\ge 0$. This provides the description of~$\mathcal{O}(\A^2\setminus C)$ and $\mathcal{O}(\A^2\setminus C)^{*}$. The last remark follows from the fact that the group $\mathcal{O}(\A^2\setminus C)^{*}$ is generated by $\k^{*}$ and one single element $g$, if and only if this element $g$ is equal to $\lambda f^{\pm 1}$ for some $\lambda\in \k^*$: Therefore, every automorphism of $\A^2 \setminus C$ induces an automorphism of $\mathcal{O}(\A^2\setminus C)$ which sends $f$ onto $\lambda f^{\pm 1}$.
\end{proof}

\subsection{The case of lines}
Proposition~\ref{Prop:Threecasescontraction} shows that we need to study isomorphisms $\A^2\setminus C\iso \A^2\setminus D$ which extend to birational maps of~$\p^2$ that contract the curve $C$ to a point. One can ask whether this point might be a point of~$\A^2$ (and would thus be contained in $D$) or belongs to the boundary line $L_{\infty}=\p^2\setminus \A^2$. As we will show (Corollary~\ref{Coro:CurveContractedPtA2}), the first possibility only occurs in a very special case, namely when $C$ is equivalent to a line. The case of lines is special for this reason, and is treated separately here.

\begin{lemma}   \label{Lemm:Lines}
Let $C\subset \A^2$ be the line given by $x=0$.
\begin{enumerate}[$(1)$]
\item \label{AutA2minusline}
The group of automorphisms of $\A^2\setminus C$ is given by:
\[\Aut(\A^2\setminus C)=\{(x,y)\mapsto (\lambda x^{\pm 1},\mu x^{n}y+s(x,x^{-1})) \mid \lambda,\mu\in \k^*, n\in \Z,s\in \k[x,x^{-1}]\}.\]
\item\label{OpenA2minusline}
Every open embedding $\A^2 \setminus C \hookrightarrow \A^2$ is equal to $\psi  \alpha$, where $\alpha\in \Aut(\A^2\setminus C)$ and
$\psi \colon \A^2 \setminus C \hookrightarrow \A^2$ extends to an automorphism of $\A^2$.
In particular, the complement of its image,
i.e.~the complement of $\psi  \alpha(\A^2 \setminus C) = \psi (\A^2 \setminus C)$,
is a curve equivalent to a line.
\end{enumerate}
\end{lemma}

\begin{proof} To prove $(\ref{AutA2minusline})$, we first observe that each
transformation $(x,y) \mapsto (\lambda x^{\pm 1},\mu x^{n}y+s(x,x^{-1}))$ actually yields an automorphism of~$\A^2\setminus C$. Then we only need to show that all automorphisms of $\A^2\setminus C$ are of this form.
An automorphism of~$\A^2\setminus C$ corresponds to an automorphism of~$\k[x,y,x^{-1}]$ which sends $x$ to $\lambda x^{\pm1 }$, where $\lambda\in \k^*$ (Lemma~\ref{Lemm:RingInvertible}). Applying the inverse of $(x,y)\mapsto (\lambda x^{\pm1 },y)$, we may assume that $x$ is fixed. We are left with an $R$-automorphism of~$R[y]$, where $R$ is the ring $\k[x,x^{-1}]$.
Such an automorphism is of the form $y\mapsto ay+b$, where $a\in R^*$, $b\in R$. Indeed, if the maps $y \mapsto p(y)$ and $y \mapsto q(y)$ are inverses of each other, the equality $y=p(q(y))$ implies that $\deg p = \deg q =1$. This actually proves that $p$ has the desired form, i.e.~$p= ay+b$, where $a\in R^*$, $b\in R$.

To prove $(\ref{OpenA2minusline})$, we use Proposition~\ref{Prop:Threecasescontraction} and write $\varphi$ as an isomorphism $\A^2\setminus C\iso \A^2\setminus D$ where $D$ is a geometrically irreducible closed curve, and only need to see that $D$ is equivalent to a line. We write $\psi=\varphi^{-1}$, choose an equation $f=0$ for $D$ (where $f\in \k[x,y]$ is an irreducible polynomial over $\kk$), and get an isomorphism $\psi^{*}\colon \mathcal{O}(\A^2\setminus C)=\k[x,y,x^{-1}]\to \mathcal{O}(\A^2\setminus D)=\k[x,y,f^{-1}]$ that sends $x$ to $\lambda f^{\pm 1}$ for some $\lambda \in \k^*$ (since the group $\mathcal{O}(\A^2\setminus D)^*$ is generated by $\k^*$ and the single element $\psi^{*}(x)$, this forces  $\psi^{*}(x)= \lambda f^{\pm 1}$). We can thus write $\psi$ as $(x,y)\mapsto (\lambda f(x,y)^{\pm 1},g(x,y)f(x,y)^n)$, where $n\in \Z$ and $g\in\k[x,y]$. Replacing $\psi$ by its composition with the automorphism $(x,y)\mapsto ((\lambda^{-1} x)^{\pm 1},y((\lambda^{-1} x)^{\pm 1})^{-n})$ of~$\A^2\setminus C$, we may assume that $\psi$ is of the form $(x,y)\mapsto (f(x,y),g(x,y))$. If $g$ is equal to a constant $\nu\in \k$ modulo $f$, we apply the automorphism $(x,y)\mapsto (x,(y-\nu)x^{-1})$ and decrease the degree of~$g$. After finitely many steps we obtain an isomorphism $\A^2\setminus D\iso \A^2\setminus C$ of the form $\psi_0\colon (x,y)\mapsto (f(x,y),g(x,y))$ where $g$ is not a constant modulo $f$. The image of~$D$ by $\psi_0$ is then dense in $C$, which implies that $\psi_0$ extends to an automorphism of~$\A^2$ that sends $D$ onto $C$ (Proposition~\ref{Prop:Threecasescontraction}).
\end{proof}

\section{Geometric description of open embeddings $\A^2\setminus C\hookrightarrow \A^2$}\label{Sec:GeomRig}

\subsection{Embeddings into Hirzebruch surfaces}
We will need not only embeddings of $\A^2$ into $\p^2$, but also embeddings of~$\A^2$ into other smooth projective surfaces, and in particular into Hirzebruch surfaces.  These surfaces play a natural role in the study of automorphisms of $\A^2$ (and of images of curves by these automorphisms), as we can decompose every automorphism of $\A^2$ into  elementary links between such surfaces and then study how the singularities at infinity of the curves behave under these elementary links (see for instance \cite{BS15}).

\begin{example} \label{Exa:Hirzebruch}
For $n \ge 1$, the $n$-th Hirzebruch surface $\F_n$ is
\[ \F_n=\{([a:b:c],[u:v])\in \p^2 \times \p^1\ \;|\;\ bv^n=cu^n\} \]
and the projection $\pi_n\colon \F_n\to \p^1$ yields a $\p^1$-bundle structure on $\F_n$.
	
Let $S_n, F_n \subset \F_n$ be the curves given by $[1:0:0] \times \p^1$ and $v=0$, respectively. The  morphism
\[ \begin{array}{rcl}
\A^2&\hookrightarrow & \F_n \\
(x,y)&\mapsto &([x:y^n:1],[y:1])
\end{array} \]
gives an isomorphism $\A^2\stackrel{\sim}{\to} \F_n \backslash (S_n\cup F_n)$.
\end{example}

We recall the following easy classical result:
\begin{lemma}
For each $n\ge 1$, the projection $\pi_n\colon \F_n\to \p^1$ is the unique $\p^1$-bundle structure on $\F_n$, up to automorphisms of the target $\p^1$. The curve $S_n$ is the unique irreducible $\kk$-curve in $\F_n$ of self-intersection
$-n$, and we have $(F_n)^2=0$.
\end{lemma}

\begin{proof}
Since $\F_n \backslash (S_n\cup F_n)$ is isomorphic to $\A^2$, whose Picard group is trivial, we have
$\Pic(\F_n)=\Z F_n+ \Z S_n$ (where the class of a divisor $D$ is again denoted by $D$).
Moreover, $F_n$ is a fibre of $\pi_n$ and $S_n$ is a section, so $(F_n)^2=0$ and $F_n\cdot S_n=1$. We denote by $S_n'\subset \F_n$ the section given by $a=0$, and find that $S_n'$ is equivalent to $S_n+nF_n$, by computing the divisor of  $\frac{a}{c}$. 

Since $S_n$ and $S_n'$ are disjoint, this yields $0=S_n\cdot (S_n+nF_n)=(S_n)^2+n$, so $(S_n)^2=-n$.

To get the result, it suffices to show that an irreducible $\kk$-curve $C\subset \F_n$ not equal to $S_n$ or to a fibre of $\pi_n$ has self-intersection at least equal to $n$. This will show in particular that a general fibre $F$ of any morphism $\F_n\to \p^1$ is equal to a fibre of $\pi_n$, since $F$ has self-intersection $0$. We write $C=kS_n+lF_n$ for some $k,l\in \Z$.
Since $C\not=S_n$ we have $0\le C\cdot S_n=l-nk$. Since $C$ is not a fibre, it intersects every fibre, so $0<F_n\cdot C=k$. This yields
$l  \geq nk>0$ and $C^2 = -nk^2 + 2kl = kl +k(l-nk) \geq kl \geq nk^2 \geq n$.
\end{proof}

\begin{lemma}\label{Lemm:EmbFnminimaldegree}
Let $C\subset \A^2$ be a geometrically irreducible closed curve. Then, there exists an integer $n\ge 1$ and an isomorphism $\iota \colon \A^2\iso \F_n \backslash (S_n\cup F_n)$ such that the closure of~$\iota(C)$ in $\F_n$ is a curve $\Gamma$ which satisfies one of the following two possibilities: 
\begin{enumerate}[$(1)$]
\item  \label{SectionCase}
$\Gamma\cdot F_n=1$ and $\Gamma\cap F_n \cap S_n = \emptyset$.

\item
\label{ng}
$\Gamma\cdot F_n\ge 2$ and the following assertions hold:
\begin{enumerate}[$(a)$]
\item
If $n=1$, then $2m_p(\Gamma)\le \Gamma\cdot F_1$ for $\{ p \} =S_1\cap F_1$, and $m_{r}(\Gamma)\le \Gamma\cdot S_1$ for each $r\in F_1(\k)$.
\item
If $n\ge 2$, then $2m_{r}(\Gamma)\le \Gamma\cdot F_n$ for each $r\in F_n(\k)$.
\end{enumerate}
\end{enumerate}
Furthermore, in case~$(\ref{SectionCase})$, the curve $C$ is equivalent to a curve given by an equation of the form
\[ a(y)x+b(y)=0,\]
where $a,b\in \k[y]$ are coprime polynomials such that $a \not = 0$ and $\deg b < \deg a$. Moreover, the following assertions are equivalent:
\begin{enumerate}[$(i)$]
\item \label{LL1} The polynomial $a$ is constant;
\item\label{LL2} The curve $C$ is equivalent to a line;
\item \label{LL3}  The curve $C$ is isomorphic to $\A^1$;
\item\label{LL4} $\Gamma \cdot S_n=0$.
\end{enumerate}
\end{lemma}

\begin{proof}

Let us take any fixed isomorphism  $\iota \colon \A^2 \iso \F_n \backslash (S_n\cup F_n)$ for some $n\ge 1$, and denote by $\Gamma$ the closure of~$\iota(C)$.

We first assume that $\Gamma\cdot F_n=1$. This is equivalent to saying that $\Gamma$ is a section of~$\pi_n$. We may furthermore assume that the $\k$-point
$q_n$ defined by $\{ q_n \} =F_n \cap S_n$ does not belong to $\Gamma$, as otherwise we could blow up the point $q_n$, contract the curve $F_n$, change the embedding to $\F_{n+1}$ and decrease by one unit the intersection number of~$\Gamma$ with $S_n$ at the point $q_n$. After finitely many steps we get $q_n\not\in \Gamma$, i.e.~we are in case~$(\ref{SectionCase})$.

If $\Gamma \cdot F_n=0$, then $\Gamma$ is a fibre of $\pi_n \colon \F_n \to \p^1$. Let $\psi$ be the unique automorphism of $\A^2$ such that $\iota \circ \psi$ is the standard embedding of~$\A^2$ into $\F_n$ of Example~$\ref{Exa:Hirzebruch}$. Then, the curve $C$ is equivalent to the curve $\psi^{-1}(C)$, which has equation $y=\lambda$, for some $\lambda\in \k$. This proves that $C$ is equivalent to the line $y=\lambda$, and thus to the line $x=\lambda$, sent by the standard embedding onto a curve  satisfying conditions~$(\ref{SectionCase})$.

It remains to consider the case where $\Gamma\cdot F_n\ge 2$.  If $\Gamma$ satisfies~$(\ref{ng})$, we are done. Otherwise, we have a $\k$-point $p\in F_n$ satisfying one of the following two possibilities:
\begin{center}\begin{enumerate}[$(a)$]
\item\label{F1P2F1}
$n=1$, $m_p(\Gamma)> \Gamma\cdot S_1$, and $p\in F_1$.
\item
\label{FnFnp}
$2m_p(\Gamma)> \Gamma\cdot F_n$ and either $n\ge 2$ or $n=1$ and $p\in S_1\cap F_1$.
\end{enumerate}\end{center}
We will replace the isomorphism $\A^2\iso \F_n \backslash (S_n\cup F_n)$ by another, where the singularities of the curve $\Gamma$ either decrease (all multiplicities are unchanged, except one which has decreased) or stay the same (as usual, the multiplicities taken into account concern not only the proper points of
$\F_n$, but also the infinitely near points). Moreover, the case where the multiplicities stay the same is only in $(\ref{F1P2F1})$, which cannot appear two consecutive times. Note that in all that process the intersection $\Gamma \cdot F_n$ remains unchanged. Then, after finitely many steps, the new curve $\Gamma$ satisfies the conditions~$(\ref{ng})$.

In case $(\ref{F1P2F1})$, we observe that the inequality $m_p(\Gamma)>\Gamma\cdot S_1$
combined
with the inequality $\Gamma\cdot S_1 \geq ( \Gamma\cdot S_1)_p \geq m_p (\Gamma)\cdot m_p(S_1)$ implies that $p\notin S_1$. We may then choose $p$ to be a $\k$-point of~$F_1\setminus S_1$ of maximal multiplicity and denote by $\tau\colon \F_1\to \p^2 $ the birational morphism contracting $S_1$ to a $\k$-point $q\in \p^2$, observe that $\tau(F_1)$ is a line through $q$, that $\tau(\Gamma)$ is a curve of multiplicity $\Gamma\cdot S_1$ at $q$ and of multiplicity $m_p(\Gamma)>\Gamma\cdot S_1$ at $p'=\tau(p)\in \tau(F_1)$. Moreover, $p'$ is a $\k$-point of~$\tau(F_1)$  of maximal multiplicity on that line. Denote by $\tau' \colon \F'_1 \to \p^2$ the birational morphism which is the blow-up at $p'$. Let $S'_1$ be the exceptional fibre of~$\tau'$, $F'_1$ the strict transform of~$\tau (F_1)$ and $\Gamma'$ the strict transform of~$\tau (\Gamma)$. We then replace the isomorphism $\A^2\iso \F_1\setminus (S_1\cup F_1)$ with the analogous isomorphism $\A^2\iso \F'_1\setminus (S'_1\cup F'_1)$ and get
\[ \forall \, r \in F'_1,\; m_r (\Gamma') \leq \Gamma' \cdot S'_1 = m_p(\Gamma).\]
Hence, $(\ref{F1P2F1})$ is no longer possible. Moreover, the singularities of the new curve $\Gamma'$ have either decreased or stayed the same: Indeed, the multiplicities of the singular points of $\tau(\Gamma)$ are the same as those of $\Gamma$, plus one point of multiplicity $\Gamma \cdot S_1$. Similarly, the multiplicities of the singular points of $\tau(\Gamma)$ are the same as those of $\Gamma'$, plus one point of multiplicity $m_p(\Gamma)$. Of course, we do not really get a singular point if the multiplicity is~$1$. Therefore, the singularities of the new curve remain the same if and only if $m_p(\Gamma)=1$ and $\Gamma\cdot S_1=0$. The situation is illustrated below in a simple example (which satisfies $m_p(\Gamma)=3>\Gamma\cdot S_1=2$).
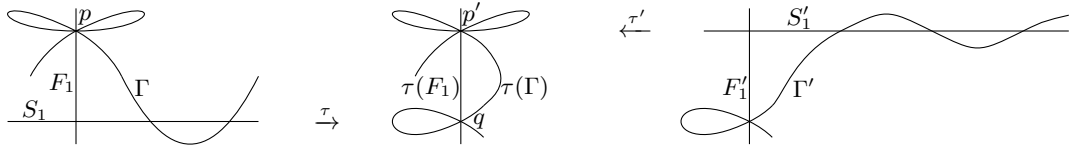
\[\begin{tikzpicture}[xscale=0.6,yscale=0.3]
\draw (0,-1) -- (0, 5);
\draw (-1.5,0) -- (4, 0);
\draw (-1,2) ..controls +(0.25,1) and +(-0.25,-0.25).. (0,4);
\draw (0,4) ..controls +(2,0) and +(2, 2).. (0,4);
\draw (0,4) ..controls +(-2,0) and +(-2,2).. (0,4);
\draw (1,2) ..controls +(-0.25,1) and +(0.25,-0.25).. (0,4);
\draw (1,2) ..controls +(1,-4) and +(-1,-4).. (4,2);
\draw (-0.3, 1.7) node{\scriptsize$F_1$};
\draw (-0.9, 0.5) node{\scriptsize$S_1$};
\draw (1.45, 1.5) node{\scriptsize$\Gamma$};
\draw (0.2, 4.6) node{\scriptsize$p$};
\draw (5.5, 0.2) node{\scriptsize $\xrightarrow{\tau}$};
\begin{scope}[xshift=240] 
\draw (0,-1) -- (0, 5);
\draw (-1,2) ..controls +(0.25,1) and +(-0.25,-0.25).. (0,4);
\draw (0,4) ..controls +(2,0) and +(2, 2).. (0,4);
\draw (0,4) ..controls +(-2,0) and +(-2,2).. (0,4);
\draw (0.8,2.5) ..controls +(-0.25,1) and +(0.25,-0.25).. (0,4);
\draw (0.5, 0.7) ..controls +(0.125, 0.25) and +(0.25,-1).. (0.8,2.5);
\draw (0.5, -0.7) ..controls +(-0.125, 0.25) and +(0.2, -0.2).. (0,0);
\draw (0,0) ..controls +(-2,2) and +(-2,-2).. (0,0);
\draw (0.5, 0.7) ..controls +(-0.125, -0.25) and +(0.2, 0.2).. (0,0);
\draw (-0.65, 1.5) node{\scriptsize $\tau(F_1)$};
\draw (0.25, 4.75) node{\scriptsize $p'$};
\draw (1.4, 1.5) node{\scriptsize $\tau(\Gamma)$};
\draw (0.4, 0) node{\scriptsize $q$};
\end{scope}
\begin{scope}[xshift=420] 
\draw (-2.5, 4.4) node{\scriptsize $\xleftarrow{\tau'}$};
\draw (0,-1) -- (0, 5);
\draw (-1,4) -- (7, 4);
\draw (0.5, -0.7) ..controls +(-0.125, 0.25) and +(0.2, -0.2).. (0,0);
\draw (0,0) ..controls +(-2,2) and +(-2,-2).. (0,0);
\draw (0.5, 0.7) ..controls +(-0.125, -0.25) and +(0.2, 0.2).. (0,0);
\draw (0.5, 0.7) ..controls +(0.25, 0.5) and +(-1, -1).. (2,4);
\draw (2, 4) ..controls +(1, 1) and +(-1, 1).. (4,4);
\draw (4, 4) ..controls +(1, -1) and +(-1, -1).. (6,4);
\draw (6, 4) ..controls +(0.5, 0.5) and +(-0.5, -0.25).. (7, 4.7);
\draw (-0.3, 1.5) node{\scriptsize $F'_1$};
\draw (1.1, 4.6) node{\scriptsize $S'_1$};
\draw (1.2, 1.5) node{\scriptsize $\Gamma'$};
\end{scope}
\end{tikzpicture}\]
In case $(\ref{FnFnp})$, we denote by $\kappa\colon \F_n\dasharrow \F_{n'}$ the birational map that blows up the point $p$ and contracts the strict transform of~$F_n$. Call $q$ the point to which the strict transform of~$F_n$ is contracted. We have $\kappa = \pi_q \circ (\pi_p)^{-1}$, where $\pi_p$, resp.~$\pi_q$, are blow-ups of the point $p$ of~$\F_n$, resp.~the point $q$ of~$\F_{n'}$. The drawing below illustrates the situation in a case where $n' = n-1$. The composition of~$\iota$ with $\kappa$ provides a new isomorphism $\A^2\to \F_{n'}\setminus (S_{n'} \cup F_{n'} )$, where $S_{n'}$ is the image of~$S_n$ and $F_{n'}$ is the curve corresponding to the exceptional divisor of~$p$. Note that $F_{n'}$ is a fibre of the $\p^1$-bundle $\pi'\colon \F_{n'}\to \p^1$ corresponding to $\pi'=\pi_n\circ \kappa^{-1}$, and that $S_{n'}$ is a section, of self-intersection $-n'$, where $n'= n+1$ if $p\in S_n$ and $n'= n-1$ if $p\notin S_n$.
Hence, since $n\ge 2$ or $n=1$ and $\{ p \} = S_n \cap F_n$,
we get that $(S_{n'})^2=-n'<0$, and obtain a new isomorphism $\iota'\colon \A^2 \iso \F_{n'} \backslash (S_{n'}\cup F_{n'})$. The singularity of the new curve $\Gamma'$ at the point $q$ is equal to $\Gamma \cdot F_n-m_p(\Gamma)$, which is strictly smaller than $m_p(\Gamma)$ by assumption. Moreover $2m_p(\Gamma)> \Gamma\cdot F_n\ge 2$, which implies that $p$ was indeed a singular point of~$\Gamma$.
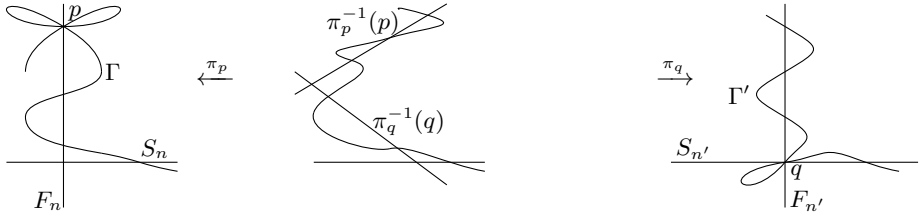
\[\begin{tikzpicture}[xscale=0.5,yscale=0.3]
\draw (0,-2) -- (0, 7);
\draw (-1.5,0) -- (3, 0);
\draw (-1,4) ..controls +(0, 1) and +(-0.25,-0.25).. (0,6);
\draw (0,6) ..controls +(2,0) and +(2, 2).. (0,6);
\draw (0,6) ..controls +(-2,0) and +(-2,2).. (0,6);
\draw (1,4) ..controls +(0, 1) and +(0.25,-0.25).. (0,6);
\draw (1,4) ..controls +(0, -1) and +(0,1).. (-1,2);
\draw (-1,2) ..controls +(0, -1) and +(-1, 0.3).. (1, 0.4);
\draw (1,0.4) ..controls +(1, -0.3) and +(-1, 0.3).. (3, -0.4);
\draw (-0.4, -1.7) node{\scriptsize $F_n$};
\draw (2.4, 0.6) node{\scriptsize $S_n$};
\draw (1.3, 4) node{\scriptsize $\Gamma$};
\draw (0.3, 6.6) node{\scriptsize $p$};
\draw (4, 4) node{\scriptsize $\xleftarrow{\pi_p}$};
\begin{scope}[xshift=230] 
\draw (-1.5,0) -- (3, 0);
\draw (2,-1) -- (-2, 4);
\draw (-2,3) -- (2, 7);
\draw (0.5,0.6) ..controls +(0.5, 0.2) and +(-1, 0.3).. (3, -0.4);
\draw (0.5,0.6) ..controls +(-0.5, -0.2) and +(-0.2, -1).. (-1.5,2.2);
\draw (-1.5,2.2) ..controls +(0.2, 1) and +(1, -0.5).. (-0.5,4.5);
\draw (-0.5,4.5) ..controls +(-1, 0.5) and +(-1, -0.5).. (0.5,5.5);
\draw (0.5,5.5) ..controls +(1, 0.5) and +(1, -0.5).. (1.5,6.5);
\draw (1.5,6.5) ..controls +(-1,0.5) and +(0, 0).. (0.75, 6.75);
\draw (-0.2, 6.2) node{\scriptsize $\pi_p^{-1}(p)$};
\draw (1, 1.8) node{\scriptsize $\pi_q^{-1}(q)$};
\draw (8, 4) node{\scriptsize $\xrightarrow{\pi_q}$};
\end{scope}
\begin{scope}[xshift=540] 
\draw (0,-2) -- (0, 7);
\draw (3.5,0) -- (-3, 0);
\draw (1,0.4) ..controls +(0.5, 0.2) and +(-1, 0.3).. (3, -0.4);
\draw (1, 0.4) ..controls +(-0.5, -0.2) and +(0.5, 0.125).. (0,0);
\draw (0,0) ..controls +(-2, -0.5) and +(-1,-2).. (0,0);
\draw (0,0) ..controls +(0.5, 1) and +(1, -1).. (0,2);
\draw (0,2) ..controls +(-1, 1) and +(-1, -1).. (0,4);
\draw (0,4) ..controls +(1, 1) and +(1, -1).. (0,6);
\draw (0,6) ..controls +(-1, 1) and +(0.5, -0.5).. (-0.5, 6.5);
\draw (0.6, -1.7) node{\scriptsize $F_{n'}$};
\draw (-2.4, 0.6) node{\scriptsize $S_{n'}$};
\draw (-1.2, 3) node{\scriptsize $\Gamma'$};
\draw (0.3, -0.45) node{\scriptsize $q$};
\end{scope}
\end{tikzpicture}\]

Finally, we must now prove the last statement of our lemma, which concerns case~$(\ref{SectionCase})$. Let $\psi$ be the unique automorphism of $\A^2$ such that $\iota \circ \psi$ is the standard embedding of~$\A^2$ into $\F_n$ of Example~$\ref{Exa:Hirzebruch}$. Then, by replacing $\iota$ by $\iota \circ \psi$ and $C$ by the equivalent curve $\psi^{-1} (C)$, we may assume that $\iota \colon \A^2\iso \F_n \backslash (S_n\cup F_n)$ is the standard embedding.  This being done, the restriction of $\pi_n \colon \F_n \to \p^1$ to $\A^2$ is $(x,y)\to [y:1]$.  The fibres of $\pi_n$, equivalent to $F_n$ being given by $y=\mathrm{cst}$, the degree in $x$ of the equation of $C$ is equal to $\Gamma\cdot F_n$ (this can be done for instance by extending the scalars to $\kk$ and taking a general fibre). Since $\Gamma\cdot F_n=1$, the equation is of the form $xa(y)+b(y)$ for some polynomials $a,b\in \k[y]$, $a\not=0$. Since $C$ is geometrically irreducible, the polynomials $a$ and $b$ are coprime. 
There exist (unique) polynomials $q,\tilde{b} \in \k [x]$ such that $b =aq +\tilde{b}$ with $\deg \tilde{b} < \deg a$. Then, changing the coordinates by applying $(x,y) \mapsto(x+q(y),y)$, we may furthermore assume that $\deg b<\deg a$. 

Let us prove that points $(\ref{LL1})$-$(\ref{LL4})$ are equivalent. The implications $(\ref{LL1})$ $\Rightarrow$ $(\ref{LL2}) \Rightarrow (\ref{LL3})$ are obvious. We then prove $(\ref{LL3}) \Rightarrow (\ref{LL4})\Rightarrow (\ref{LL1})$.

$(\ref{LL3}) \Rightarrow (\ref{LL4})$: We recall that $\Gamma$ is a section of $\pi_n\colon \F_n\to \p^1$, so that
we have isomorphisms
$\Gamma\simeq \p^1$ and $\Gamma\setminus F_n\simeq \A^1$. The fact that $C=\Gamma \setminus (F_n\cup S_n)\simeq \A^1$ implies that $C\cap (S_n\setminus F_n)$ is empty. Since $\Gamma\cap F_n \cap S_n=\emptyset$ by assumption, we get $\Gamma\cdot S_n=0$.

$(\ref{LL4}) \Rightarrow  (\ref{LL1})$: We use the open embedding
\[ \begin{array}{rcl}
\A^2&\hookrightarrow & \F_n \\
(u,v)&\mapsto &([1:uv^n:u],[v:1]).
\end{array} \]
The preimages of $\Gamma$ and $S_n$ by this embedding are the curves of equations $a(v) + b(v) u =0$ and $u=0$. Hence $\Gamma \cdot S_n=0$ implies that $a$ has no $\kk$-root and thus is a constant.
\end{proof}

\subsection{Extension to regular morphisms on $\A^2$}
The following proposition is the principal tool in the proof of Proposition~\ref{Prop:XBCD}, Corollary~\ref{Cor:Th1a} and Proposition~\ref{Prp:Th1b}, which themselves give the main part of Theorem~\ref{RigidityThm}.

\begin{proposition}   \label{Prop:NoBasePtA2}
Let $C\subset \A^2$ be a geometrically irreducible closed curve, not equivalent to a line, and let $\varphi\colon \A^2\setminus C\hookrightarrow \A^2$ be an open embedding. Then, there exists an open embedding $\iota \colon \A^2 \hookrightarrow \F_n$, for some $n\ge 1$, such that the rational map $\iota \circ \varphi$ extends to a regular morphism $\A^2\to \F_n$, and such that $\iota(\A^2)=\F_n\setminus (S_n\cup F_n)$ $($where $S_n$ and $F_n$ are as in Example~$\ref{Exa:Hirzebruch})$.
\end{proposition}
\begin{proof}
By Proposition~\ref{Prop:Threecasescontraction}, $\varphi(\A^2\setminus C)=\A^2\setminus D$ for some geometrically irreducible closed curve $D$. 
If $\varphi$ extends to an automorphism of~$\A^2$ sending $C$ onto $D$, the result is obvious, by taking any isomorphism $\iota\colon \A^2\iso \F_n\setminus (F_n\cup S_n)$, so we may assume that $\varphi$ does not extend to an automorphism of~$\A^2$. 
Lemma~\ref{Lemm:Lines} implies, since $C$ is not equivalent to a line,  that the same holds for $D$. Moreover, Proposition~\ref{Prop:Threecasescontraction} implies that the extension of~$\varphi^{-1}$ to a birational map $\p^2\dasharrow \p^2$, via the standard embedding $\A^2\hookrightarrow \p^2$, contracts the curve $\overline{D}$  to a $\k$-point of~$\p^2$. In particular, it does not send $\overline{D}$ birationally onto $\overline{C}$ or onto $L_{\infty}.$

We choose an open embedding $\iota\colon \A^2\hookrightarrow\F_n$ given by Lemma~\ref{Lemm:EmbFnminimaldegree}, which comes from an isomorphism $\iota\colon \A^2\iso \F_n \backslash (S_n\cup F_n)$, such that the closure
of $\iota(D)$ in $\F_n$ is a curve $\Gamma$ which satisfies one of the two possibilities $(\ref{SectionCase})$-$(\ref{ng})$ of Lemma~\ref{Lemm:EmbFnminimaldegree}.

We want to show that the open embedding $\iota\circ \varphi\colon \A^2\setminus C\hookrightarrow \F_n$ extends to a regular morphism on $\A^2$. Using the standard embedding of  $\A^2$ into $\p^2$ (Definition~\ref{Def:Plane}), we get a birational map $\psi\colon \p^2\dasharrow \F_n$ and need to show that all $\kk$-base-points of this map are contained in $L_{\infty}$. Note that $\psi$ restricts to an isomorphism  $\p^2 \setminus (L_{\infty} \cup \overline{C})\iso\F_n \setminus ( F_n \cup S_n \cup \Gamma)$. This implies that all $\kk$-base-points of $\psi,\psi^{-1}$ are defined over $\k$ (Lemma~\ref{Lem:IsoComplRat}\eqref{BasePtsKRat}) and gives the following commutative diagram
\[\xymatrix@R=2mm@C=2cm{&& X\ar[ld]_{\eta}\ar[rd]^{\pi} \\
\A^2\ar@{^{(}->}^{\mathrm{std}}[r] & \p^2\ar@{-->}[rr]^{\psi} && \F_n & \A^2\ar@{_{(}->}_{\iota}[l]\\
& \A^2 \setminus C\ar@{_{(}->}[lu] \ar^{\varphi}_{\simeq}[rr]&&  \A^2\setminus D,\ar@{^{(}->}[ur]} \]
where $\eta,\pi$ are blow-ups of the base-points of $\psi$ and $\psi^{-1}$ respectively, and where 
$\eta^{-1} (L_{\infty} \cup \overline{C}) =\pi^{-1}( F_n \cup S_n \cup \Gamma)$ (Lemma~\ref{Lem:IsoComplRat}\eqref{Resolution}-\eqref{inverse images of U and V are equal}).

We assume by contradiction that $\psi$ has a base-point $q$ in $\A^2=\p^2\setminus L_{\infty}$, which means that one $(-1)$-curve $E_q\subset X$ is contracted by $\eta$ to $q$. This curve is the exceptional divisor of a base-point infinitely near to
$q$, but
not necessarily of~$q$. The minimality of the resolution implies that $\pi$ does not contract $E_q$, so $\pi(E_q)$ is a curve of~$\F_n$ contracted by $\psi^{-1}$ to $q$, which belongs to $\{\Gamma,F_n,S_n\}$.
 
We first study the case where $\psi$ has no base-point in $L_{\infty}$. The strict transform of~$L_{\infty}$ has then self-intersection $1$ on $X$. Hence, it is not contracted by $\pi$, and thus sent onto a curve of self-intersection $\ge 1$, which belongs to $\{\Gamma,F_n,S_n\}$ by Lemma \ref{Lem:IsoComplRat}\eqref{ComplementsOutside}. As $(F_n)^2=0$ and $(S_n)^2=-n\le -1$, $L_{\infty}$ is sent onto $\Gamma$ by $\psi$. This contradicts the fact that $\Gamma$ is not sent birationally onto $L_{\infty}$ by $\psi^{-1}$.

We can now reduce to the case where $\psi$ also has a base-point $p$ in $L_{\infty}$. There is thus a $(-1)$-curve $E_p\subset X$  contracted by $\eta$ to $p$ and not contracted by $\pi$. As above, this curve is the exceptional divisor of a base-point infinitely near to $p$, but not necessarily of~$p$. Again, $\pi(E_p)$ belongs to $\{\Gamma,F_n,S_n\}$.
 
 We thus have at least two of the curves $\Gamma$, $F_n$, $S_n$ that correspond to $(-1)$-curves of~$X$ contracted by $\eta$. 
 
 We suppose first that $S_n$ corresponds to a $(-1)$-curve of~$X$ contracted by $\eta$. The fact that $(S_n)^2=-n\le -1$ implies that $n=1$ and that $\pi$ does not blow up any point of~$S_n$. As there is another $(-1)$-curve of~$X$ contracted by $\eta$, the two curves are disjoint on $X$, and thus also disjoint on $\F_1$, since $\pi$ does not blow up any point of~$S_1$. The other curve is then $\Gamma$ (since $F_1\cdot S_1=1$), and $\Gamma\cdot S_1=0$. If moreover $\Gamma\cdot F_1=1$ (condition $(\ref{SectionCase})$ of Lemma~\ref{Lemm:EmbFnminimaldegree}), then the contraction $\F_1\to \p^2$ of~$S_1$ sends $\Gamma$ onto a line of~$\p^2$, which contradicts the fact that $D\subset \A^2$ is not equivalent to a line.  If $\Gamma\cdot F_1\ge 2$, then condition $(\ref{ng})$ of Lemma~\ref{Lemm:EmbFnminimaldegree} implies that $m_r(\Gamma)\le \Gamma\cdot S_1=0$ for each $r\in F_1(\k)$. Hence, the intersection of~$\Gamma$ with $F_1$ (which is not empty since $\Gamma\cdot F_1\ge 2$) consists only of points not defined over $\k$, which are therefore not blown up by $\pi$. The strict transforms $\tilde{\Gamma}$ and $\tilde{F_1}$ on $X$ then satisfy  $\tilde\Gamma\cdot \tilde{F_1}=\Gamma\cdot F_1\ge 2$. As $\tilde\Gamma$ is contracted by $\eta$, the image $\eta(\tilde{F_1})$ is a singular curve and is then equal to $\overline{C}$. This contradicts the fact that $\psi$ contracts $\overline{C}$ to a point. 
 
 There remains the case  where $S_n$ does not correspond to a $(-1)$-curve of~$X$ contracted by $\eta$, which implies that $\{\pi(E_p),\pi(E_q)\}=\{F_n,\Gamma\}$, or equivalently that $\{E_p,E_q\}=\{\tilde{F_n},\tilde{\Gamma}\}$, where $\tilde{F_n}$ and $\tilde{\Gamma}$ denote the strict transforms of~$F_n$ and $\Gamma$ on $X$. Since $(F_n)^2=0$ and $(\tilde{F_n})^2=-1$, there exists exactly one $\kk$-point $r\in F_n$ (and no infinitely near points) blown up by $\pi$, which is then a $\k$-point (as all base-points of~$\pi$ are defined over $\k$). We obtain
\begin{center}$m_r(\Gamma)=\Gamma\cdot F_n\ge 1$ and $\Gamma\cap F_n=\{r\}$,\end{center} since $\tilde{F_n}$ and $\tilde{\Gamma}$ are disjoint on $X$ (and because $\Gamma\cdot F_n\ge 1$, as $\Gamma$ satisfies one of the two conditions $(\ref{SectionCase})$-$(\ref{ng})$ of Lemma~\ref{Lemm:EmbFnminimaldegree}).

We now prove that $\pi^{-1}(r)$ and $\pi^{-1}(S_n)$ are two disjoint connected sets of rational curves which intersect the two curves $\tilde{F_n}$ and $\tilde{\Gamma}$, i.e.~the two curves $E_p$ and $E_q$. For this, it suffices to prove that $r\notin S_n$ and that $S_n\cdot \Gamma \ge 1$. Suppose first that $\Gamma\cdot F_n=1$ (condition $(\ref{SectionCase})$ of Lemma~\ref{Lemm:EmbFnminimaldegree}). Since
$\Gamma \cap F_n\cap S_n=\emptyset$,
we get $r\in F_n\setminus S_n$. The inequality $\Gamma\cdot S_n>0$ is provided by the fact that $D$ is not equivalent to a line (see again condition  $(\ref{SectionCase})$ of Lemma~\ref{Lemm:EmbFnminimaldegree} and the equivalence between $(\ref{LL2})$ and $(\ref{LL4})$ given in that case). Suppose now
that $\Gamma\cdot F_n\ge 2$.
As $m_r(\Gamma)=\Gamma\cdot F_n\ge 2$, we have $2m_r(\Gamma)> \Gamma\cdot F_n$, which implies that $n=1$, $r\in F_n\setminus S_n$ and $2\le m_r(\Gamma)\le \Gamma\cdot S_n$ (see again possibility $(\ref{ng})$ of Lemma~\ref{Lemm:EmbFnminimaldegree}).

We conclude by observing that, since $\eta(E_q)=q\in \p^2\setminus L_{\infty}$ and $\eta(E_p)=p\in L_{\infty}$, any connected set of curves
of~$\eta^{-1}(L_{\infty}\cup \overline{C})$
which intersects the two curves $E_q$ and $E_p$ must contain the strict transform $\tilde{C}$ of~$\overline{C}$.
Since $\pi^{-1}(r)$ and $\pi^{-1}(S_n)$ are included in $\pi^{-1}( F_n \cup S_n \cup \Gamma)=\eta^{-1} (L_{\infty} \cup \overline{C})$, this contradicts the fact that $\pi^{-1}(r)$ and $\pi^{-1}(S_n)$ are two disjoint connected sets of rational curves which intersect the two curves $\tilde{F_n}$ and~$\tilde{\Gamma}$.
\end{proof}

A direct consequence of Proposition~\ref{Prop:NoBasePtA2} is the following corollary, which shows that only smooth curves $C\subset \A^2$ are interesting to study. This also follows from Proposition~\ref{Prop:XBCD} below. Since the proof of Proposition~\ref{Prop:XBCD} is more involved, we prefer first to explain the simpler argument that shows how the smoothness follows from Proposition~\ref{Prop:NoBasePtA2}.

\begin{corollary} \label{corollary: in the singular case, the embedding extends to an automorphism of the plane}
Let $C\subset \A^2$ be a geometrically irreducible closed curve. If $C$ is not smooth, then every open embedding $\varphi \colon \A^2\setminus C\hookrightarrow \A^2$ extends to an automorphism of~$\A^2$.
\end{corollary}

\begin{proof}
By Proposition~\ref{Prop:Threecasescontraction}, $\varphi(\A^2\setminus C)=\A^2\setminus D$ for some geometrically irreducible closed curve $D$. We apply Proposition~$\ref{Prop:NoBasePtA2}$ and obtain an open embedding $\iota \colon \A^2 \hookrightarrow \F_n$, for some $n\ge 1$, such that the rational map $\iota \circ \varphi$ extends to a regular morphism $\A^2\to \F_n$. Embedding $\A^2$ into $\p^2$, we get a birational map
$\psi \colon \p^2\dasharrow \F_n$ which is regular on $\A^2$. In particular, the singular $\kk$-points of~$C$ are not blown up in the minimal resolution of $\psi$. Hence, the curve $\overline{C}$ is not contracted by $\psi$ and is thus sent onto a singular curve $\psi(\overline{C})\subset \F_n$. Since $\psi$ restricts to an isomorphism  $\p^2 \setminus (L_{\infty} \cup \overline{C})\iso\F_n \setminus ( F_n \cup S_n \cup \overline{D} )$, Lemma~\ref{Lem:IsoComplRat}\eqref{ComplementsOutside} shows that the singular curve $\psi(\overline{C})$ must be $F_n$, $S_n$ or $\overline{D}$. As $F_n$ and $S_n$ are smooth, we find that $\psi(\overline{C})=\overline{D}$. Proposition~\ref{Prop:Threecasescontraction} then shows that $\varphi$ extends to an automorphism of~$\A^2$. 
\end{proof}

Another direct consequence of Proposition~\ref{Prop:NoBasePtA2} is the following result, which shows that in case~$(\ref{Case3:Ccontracted})$ of Proposition~\ref{Prop:Threecasescontraction}, the point to which $\overline{C}$ is contracted  lies in $\A^2$ only in a very special situation: 

\begin{corollary} \label{Coro:CurveContractedPtA2}
Let $C\subset \A^2$ be a geometrically irreducible closed curve and let $\varphi\colon \A^2\setminus C \hookrightarrow \A^2$ be an open embedding. If the extension of $\varphi$ to $\p^2$ contracts the curve $C$ $($or equivalently its closure$)$ to a point of $\A^2$, then there exist automorphisms $\alpha,\beta$ of~$\A^2$ and an endomorphism  $\psi\colon \A^2\to \A^2$ of the form  $(x,y)\mapsto (x,x^ny)$, where $n \geq 1$ is an integer, such that $\varphi = \alpha \psi \beta$. In particular, $C\subset \A^2$ is equivalent to a line, via $\beta$.
\end{corollary}

\begin{proof}
By Proposition~\ref{Prop:Threecasescontraction}, $\varphi(\A^2\setminus C)=\A^2\setminus D$ for some geometrically irreducible closed curve $D$. Denote by $\varphi^{-1} \colon \A^2 \dasharrow \A^2$ the birational transformation which is the inverse of $\varphi$.
Since $C$ is contracted by $\varphi$ to a point of~$\A^2$, it is not possible to find an open embedding $\iota \colon \A^2 \hookrightarrow \F_n$, for some $n\ge 1$, such that the birational map $\iota \circ \varphi ^{-1}$
actually defines a regular morphism $\A^2\to \F_n$.
 By Proposition~\ref{Prop:NoBasePtA2}, this implies that $D$ is equivalent to a line. Hence, the same holds for $C$, by Lemma~\ref{Lemm:Lines}. Applying automorphisms of~$\A^2$ at the source and the target, we may then assume that $C$ and $D$ are equal to the line $x=0$.
By Lemma~\ref{Lemm:Lines}$(\ref{AutA2minusline})$, the map $\varphi$ is of the form $(x,y)\mapsto (\lambda x,\mu x^ny+s(x))$, where $\lambda,\mu\in\k^*$, $n\ge 1$ and $s\in \k[x]$ is a polynomial.
We then observe that
$\varphi =  \alpha \psi $, where $\alpha$ is the automorphism of~$\A^2$ given by  $(x,y) \mapsto (\lambda x, \mu y+s(x) )$ and $\psi$ is the endomorphism of $\A^2$ given by   $(x,y)\mapsto (x,x^ny)$.
\end{proof}

Corollary~\ref{Coro:CurveContractedPtA2} also gives a simple proof of the following characterisation of birational endomorphisms of~$\A^2$ that contract only one geometrically irreducible closed curve. This result has already been obtained by Daniel Daigle in \cite[Theorem~4.11]{Daigle1991b}.

\begin{corollary}  \label{Cor:BirMor}
Let $C\subset \A^2$ be a geometrically irreducible closed curve and let $\varphi$ be a birational endomorphism of $\A^2$ which restricts to an open embedding $ \A^2\setminus C\hookrightarrow \A^2$. Then, the following assertions are equivalent:

\begin{enumerate}[$(i)$]
\item \label{ContractsinA2}
The endomorphism $\varphi$ contracts the curve $C$.
\item \label{EndoA2}
The endomorphism $\varphi$ is not an automorphism.
\item \label{EndoA2special}
There exist automorphisms $\alpha,\beta$ of~$\A^2$ and an endomorphism  $\psi\colon \A^2\to \A^2$ of the form  $(x,y)\mapsto (x,x^ny)$, where $n \geq 1$ is an integer, such that $\varphi = \alpha \psi \beta$.
\end{enumerate}
\end{corollary}

\begin{proof}
 $(\ref{EndoA2special})\Rightarrow(\ref{EndoA2})$: This follows from the fact that, for each $n\ge 1$, the map $\psi\colon (x,y)\mapsto (x,x^ny)$ is a birational endomorphism of~$\A^2$ which is not an automorphism, as its inverse $\psi^{-1}\colon (x,y)\mapsto (x,x^{-n}y)$ is not regular.

$(\ref{EndoA2})\Rightarrow(\ref{ContractsinA2})$:
Denote by  $\hat{\varphi} \colon \p^2 \dasharrow \p^2$ the birational map induced by $\varphi$.
Since $\varphi$ is an endomorphism of $\A^2$ which is not an automorphism, cases~$(\ref{Case1:AutA2})$-$(\ref{Case2:Dline})$ of Proposition~\ref{Prop:Threecasescontraction} are not possible.
Hence, we are in case~$(\ref{Case3:Ccontracted})$: $C$ is contracted by $\hat{\varphi}$ to a point of~$\p^2$, which is necessarily in $\A^2$ since $\varphi(\A^2)\subset \A^2$. 

$(\ref{ContractsinA2})\Rightarrow(\ref{EndoA2special})$: This follows from Corollary~\ref{Coro:CurveContractedPtA2}.
\end{proof}

\subsection{Completion with two curves and a boundary}
The following technical Proposition~\ref{Prop:XBCD} is used to prove Corollary~\ref{Cor:Th1a} and Proposition~\ref{Prp:Th1b}, which yield almost all statements of Theorem~\ref{RigidityThm}.

\begin{definition}\label{Def:kForest}
Let $X$ be a smooth projective surface. A reduced closed curve $C\subset X$ is a \emph{$\k$-forest of $X$} if $C$ is a finite union of closed curves $C_1,\dots,C_n$, all isomorphic (over $\k$) to $\p^1$ and if each singular $\kk$-point of $C$ is a $\k$-point lying on exactly two components $C_i,C_j$ intersecting transversally. We moreover ask that $C$ does not contain any loop. If $C$ is connected, we say that $C$ is a \emph{$\k$-tree}.
\end{definition}

\begin{remark}\label{Rem:Preimage}
If $\eta\colon X\to Y$ is a birational morphism between smooth projective surfaces such that all $\kk$-base-points of $\eta^{-1}$ are defined over $\k$, then the exceptional curve of $\eta$ (the union of the contracted curves) is a $\k$-forest $E\subset X$. Moreover, the strict transform and the preimage of any $\k$-forest of $Y$ is a $\k$-forest of $X$. The preimage of a $\k$-tree is a $\k$-tree.
\end{remark}

\begin{proposition}   \label{Prop:XBCD}
Let $C,D\subset \A^2$ be geometrically irreducible closed curves, not equivalent to lines, and let $\varphi\colon \A^2\setminus C\iso\A^2\setminus D$ be an isomorphism which does not extend to an automorphism of~$\A^2$. Then there is a smooth projective surface $X$ and two open embeddings $\rho_1,\rho_2\colon \A^2\hookrightarrow X$
which make the following diagram commutative
\[\xymatrix{& X  \\
\A^2  \ar@{^{(}->}[ru]^{\rho_1}   &     & \A^2  \ar@{_{(}->}[lu]_{\rho_2}   \\
\A^2 \setminus C  \ar@{^{(}->}[u]  \ar^{\varphi}_{\simeq}[rr] & &   \A^2 \setminus D  \ar@{_{(}->}[u]  }  \]
and such that the following holds:
\begin{enumerate}[$(i)$]
\item\label{It1}
The curves $\Gamma=\overline{\rho_1(C)}\subset X$, $\Delta=\overline{\rho_2(D)}\subset X$ are isomorphic to $\p^1$.
\item\label{It2}
For $i=1,2$, we have $\rho_i(\A^2)=X\setminus B_i$ for some $\k$-tree $B_i$.
\item\label{It3}
Writing $B=B_1\cap B_2$, we have $B_1=B\cup \Delta$ and $B_2=B\cup \Gamma$.
\item\label{It4}
There is no birational morphism $X\to Y$, where $Y$ is a smooth projective surface, which contracts one connected component of~$B$, and no other $\kk$-curve.
\item\label{It5}
The number of connected components of~$B$ is equal to the number of $\kk$-points of~$B\cap \Gamma$ and to the number of $\kk$-points of  $B\cap \Delta$, and is at most $2$.
\end{enumerate}
\end{proposition}
\begin{proof}By Proposition~\ref{Prop:NoBasePtA2}, there exist integers $m,n \geq 1$, and isomorphisms
\[\iota_1 \colon \A^2 \iso \F_m\setminus (S_m\cup F_m), \; \iota_2 \colon \A^2 \iso \F_n\setminus (S_n\cup F_n)\]
 such that both open embeddings $\iota_1\varphi^{-1}\colon \A^2\setminus D\to \F_m$ and $\iota_2\varphi\colon \A^2\setminus C\to \F_n$ extend to regular morphisms $u_1\colon \A^2\to \F_m$ and $u_2\colon \A^2\to \F_n$. Denoting by $\psi\colon \F_m\dasharrow \F_n$ the corresponding birational map, equal to $\iota_2(u_1)^{-1}=u_2(\iota_1)^{-1}$, 
the restriction of $\psi$ gives an isomorphism $\F_m \setminus ( S_m \cup F_m \cup \iota_1(C) )\iso\F_n \setminus ( S_n \cup F_n \cup \iota_2(D) )$ (which corresponds to $\varphi$). We then  have the following commutative diagram
\[\xymatrix@R=4mm@C=2cm{& X\ar[ld]_{\eta}\ar[rd]^{\pi} \\
\F_m \ar@{-->}[rr]^{\psi} && \F_n \\
\A^2  \ar@{^{(}->}[u]^{\iota_1}   \ar[rru]_(0.8){u_2}|\hole &   & \A^2 \ar@{^{(}->}[u]_{\iota_2}    \ar[llu]^(0.8){u_1}  \\
\A^2 \setminus C  \ar@{^{(}->}[u]  \ar^{\varphi}_{\simeq}[rr]& &   \A^2 \setminus D  \ar@{^{(}->}[u]  }  \]
where $\eta$ and $\pi$ are birational morphisms, which are sequences of blow-ups of~$\k$-points, being the base-points of~$\psi$ and $\psi^{-1}$ respectively (Lemma~$\ref{Lem:IsoComplRat}$).

Since $u_1,u_2$ are regular on $\A^2$, the $\kk$-base-points of~$\psi$ (which are $\k$-points), resp.~$\psi^{-1}$, are infinitely near to $\k$-points of $F_m \cup S_m \subset \F_m$, resp.~$F_n \cup S_n \subset \F_n$. In particular, we get two open embeddings 
\[\rho_1=\eta^{-1}\iota_1\colon \A^2\hookrightarrow X, \; \rho_2=\pi^{-1}\iota_2\colon \A^2\hookrightarrow X\]
such that $\rho_2\varphi=\rho_1$ (or more precisely $\rho_2\varphi=\rho_1|_{\A^2\setminus C}$).
We have $\rho_1(\A^2)=X\setminus B_1$ and $\rho_2(\A^2)=X\setminus B_2$, where $B_1\, := \, \eta^{-1}(S_m\cup F_m)$ and
$B_2  \, := \,  \pi^{-1}(S_n\cup F_n)$ are $\k$-trees (see Remark~\ref{Rem:Preimage}).

By Lemma~\ref{Lem:IsoComplRat}, the following equality holds:
\[ \eta^{-1} ( S_m \cup F_m \cup \iota_1(C) ) = \pi^{-1}  ( S_n \cup F_n \cup \iota_2(D) ).\]
The left-hand side is equal to $B_1 \cup \Gamma$, where $ \Gamma=\overline{\rho_1(C)} \subset X$ is the strict transform of $\overline{ \iota_1(C) } \subset \F_m$ by $\eta$ and the right-hand side is equal to  $B_2 \cup \Delta$, where $ \Delta =\overline{\rho_2(D)}\subset X$ is the strict transform of $\overline{ \iota_2(D) } \subset \F_n$ by $\pi$. The fact that $\varphi$ does not extend to an automorphism of~$\A^2$ implies that $B_1 \not=B_2$, whence $\Delta \not=\Gamma$. Writing $B:=B_1\cap B_2$, the equality $B_1 \cup \Gamma = B_2 \cup \Delta$ yields:
\[B_2=B \cup \Gamma \text{ and } B_1=B \cup \Delta \text{ (with } \Gamma=\overline{\rho_1(C)}, \Delta =\overline{\rho_2(D)} \subset X).\]
In particular, since $B_1,B_2$ are two $\k$-trees, $\Gamma$ and $\Delta$ are isomorphic to $\p^1$ (over $\k$) and intersect transversally $B$ in a finite number of~$\k$-points.
We have now found the surface $X$ together with the embeddings $\rho_1,\rho_2$, satisfying conditions $(\ref{It1})$--$(\ref{It2})$--$(\ref{It3})$. We will then modify $X$ if needed, in order to get also $(\ref{It4})$--$(\ref{It5})$.

The number of connected components of~$B$ is equal to the number of $\kk$-points of~$B\cap \Gamma$, and of~$B\cap \Delta$:
This follows from the fact that $B \cup \Gamma$ and $B \cup \Delta$ are $\k$-trees.
Remember also that each $\kk$-point of~$B\cap \Gamma$, or of~$B\cap \Delta$, is a $\k$-point, as mentioned earlier.

Suppose that the number of connected components of~$B$ is $r\ge 3$, and let us show that at least $r-2$ connected components of~$B$ are contractible (in the sense that there is a birational morphism $X\to Y$, where $Y$ is a smooth projective rational surface, which contracts one component of~$B$ and no other $\kk$-curve). To show this, we first observe that $\Gamma$ intersects $r$ distinct curves of~$B$. Since $\Gamma$ is one of the irreducible components of~$B_2=\pi^{-1}(S_n\cup F_n)$, we can decompose $\pi$ as $\pi_2\circ \pi_1$ where $\pi_1(\Gamma)$ is an irreducible component of~$(\pi_2)^{-1}(S_n\cup F_n)$ intersecting exactly two other irreducible components $R_1,R_2$, and such that all $\kk$-points blown up by $\pi_1$ are infinitely near points of~$\pi_1(\Gamma)\setminus (R_1\cup R_2)$. This proves that we can contract at least $r-2$ connected components of~$B$.

If one connected component of~$B$ is contractible, there exists a morphism $X\to Y$, where $Y$ is a smooth projective rational surface, which contracts this component of~$B$, and no other curve. Since the component intersects $\Delta$ transversally in one point, and also $\Gamma$ in one point, we can replace $X$ by $Y$, $\rho_1,\rho_2$ by their compositions with the morphism $X \to Y$ and still fulfill conditions $(\ref{It1})$--$(\ref{It2})$--$(\ref{It3})$. After finitely many steps, condition $(\ref{It4})$ is satisfied. By the observation made earlier, the number of connected components of~$B$, after this is done, is at most $2$, giving then $(\ref{It5})$.
\end{proof}

\begin{corollary}\label{Cor:Th1a}
Let $C,D\subset \A^2$ be geometrically irreducible closed curves and let $\varphi\colon \A^2\setminus C\iso\A^2\setminus D$ be an isomorphism which does not extend to an automorphism of~$\A^2$. 

Then, the curves $C,D$ are isomorphic to open subsets of~$\A^1$: there exist polynomials $P,Q\in \k[t]$ without square factors, such that $C\simeq \Spec(\k[t,\frac{1}{P}])$ and $D\simeq \Spec(\k[t,\frac{1}{Q}])$. Moreover, the numbers of~$\kk$-roots of~$P$ and $Q$ are the same $($i.e.~extending the scalars to $\kk$, the curves $C$ and $D$ become isomorphic to $\A^1$ minus some finite number of points, the same number for both curves$)$. The numbers of~$\k$-roots of~$P$ and $Q$ are also the same.

\end{corollary}

\begin{remark}
When $\k=\C$, this follows from the fact that $C$ and $D$ are isomorphic to open subsets of $\A^1$, since the curves are rational (Corollary~\ref{Cor:NotRat}) and smooth (Corollary~\ref{corollary: in the singular case, the embedding extends to an automorphism of the plane}). Indeed, since $\A^2 \setminus C$ and $\A^2 \setminus D$ are isomorphic, they have the same Euler characteristic, so $C$ and $D$ also have the same Euler characteristic. 
\end{remark}

\begin{proof}
If $C$ or $D$ is equivalent to a line, so are both curves (Lemma~\ref{Lemm:Lines}), and the result holds. Otherwise, we apply Proposition~\ref{Prop:XBCD} and get a smooth projective surface $X$ and two open embeddings $\rho_1,\rho_2\colon \A^2\hookrightarrow X$ such that $\rho_2\varphi=\rho_1$ and satisfying the conditions $(\ref{It1})$-$(\ref{It2})$-$(\ref{It3})$-$(\ref{It4})$-$(\ref{It5})$. In particular, $C$ is isomorphic to $\Gamma\setminus B_1=\Gamma\setminus ((\Gamma \cap B)\cup (\Gamma\cup \Delta))$. Since $\Gamma$ is isomorphic to $\p^1$ and $\Gamma\cap B$ consists of one or two $\k$-points, this shows that $\Gamma$ is isomorphic to an open subset of~$\A^1$. Proceeding similarly for $D$, we get isomorphisms $C\simeq \Spec(\k[t,\frac{1}{P}])$ and $D\simeq \Spec(\k[t,\frac{1}{Q}])$ where $P,Q\in \k[t]$ are polynomials, which we may assume without square factors.

The number of~$\kk$-roots of~$P$ is equal to the number of  $\kk$-points of~$\Gamma\cap B_1$ minus $1$. Similarly, the number of~$\kk$-roots of~$Q$ is equal to the number of  $\kk$-points of~$\Delta \cap B_2$ minus~$1$. To see that these numbers are equal, we observe that $\Gamma\cap B_1=(\Gamma\cap B)\cup (\Gamma\cap \Delta)$, that $\Delta \cap B_2=(\Delta \cap B)\cup (\Delta\cap \Gamma)$, and that the number of $\kk$-points of~$\Gamma \cap B$ is the same as the number of $\kk$-points of~$\Delta\cap B$ (this follows from $(\ref{It5})$). As each point of~$\Gamma\cap B$ that is contained in $\Gamma\cap \Delta$ is also contained in $\Delta \cap B$, this shows that $P$ and $Q$ have the same number of~$\kk$-roots. As each $\kk$-point of~$\Gamma\cap B_1$ or $\Delta\cap B_2$ which is not a $\k$-point is contained in $\Gamma\cap \Delta$, the polynomials $P$ and $Q$ have the same number of $\k$-roots.
\end{proof}

\begin{proposition}\label{Prp:Th1b}
Let $C,D,D'\subset \A^2$ be geometrically irreducible closed curves, not equivalent to lines, and let $\varphi\colon \A^2\setminus C\iso\A^2\setminus D$, $\varphi'\colon \A^2\setminus C\iso\A^2\setminus D'$ be isomorphisms which do not extend to automorphisms of~$\A^2$. Then, one of the following holds:
\begin{enumerate}[$(a)$]
\item
The map $\varphi' (\varphi)^{-1}$ extends to an automorphism of~$\A^2$ $($sending $D$ to $D')$;
\item \label{a1remov}
The  curves $C,D,D'$ are isomorphic to $\A^1$;
\item 
The  curves $C,D,D'$ are isomorphic to $\A^1\setminus \{0\}$.\end{enumerate}
\end{proposition}

\begin{remark}
Case~$(\ref{a1remov})$ never occurs, as we will show later. Indeed, since $C$ is not equivalent to a line, the existence of~$\varphi,\varphi'$ is excluded (Proposition~\ref{Prop:CurvesA1} below).
\end{remark}

\begin{proof}
If $C\simeq \A^1$ or $C\simeq \A^1\setminus \{0\}$, then $D\simeq C\simeq D'$ by Corollary~\ref{Cor:Th1a}. We may thus assume that $C$ is not isomorphic to $\A^1$ or $\A^1\setminus \{0\}$.
We apply Proposition~\ref{Prop:XBCD} with $\varphi$ and $\varphi'$ and get smooth projective surfaces $X,X'$ and open embeddings $\rho_1,\rho_2,\rho_1',\rho_2'\colon \A^2\hookrightarrow X$ such that $\rho_2\varphi=\rho_1$, $\rho_2'\varphi'=\rho_1'$ and satisfying the conditions $(\ref{It1})$-$(\ref{It2})$-$(\ref{It3})$-$(\ref{It4})$-$(\ref{It5})$. In particular, we obtain an isomorphism $\kappa \colon X\setminus (B\cup \Gamma\cup \Delta)\iso X'\setminus (B'\cup \Gamma'\cup \Delta')$ (where 
 $\Gamma=\overline{\rho_1(C)}\subset X$, $\Delta=\overline{\rho_2(D)}\subset X$,  $\Gamma'=\overline{\rho_1'(C)}\subset X'$, $\Delta'=\overline{\rho_2'(D')}\subset X'$) and a commutative diagram
\[\xymatrix@R=4mm@C=2cm{
&X\ar@{-->}[rr]^{\kappa}& &X'\\
\A^2  \ar@{^{(}->}[ur]^{\rho_2}    &   & \A^2 \ar@{^{(}->}[ul]_{\rho_1}\ar@{^{(}->}[ur]^{\rho_1'}    &   & \A^2 \ar@{^{(}->}[ul]_{\rho_2'}  \\
\A^2 \setminus D  \ar@{^{(}->}[u]  & &   \A^2 \setminus C\ar_{\varphi}^{\simeq}[ll]\ar^{\varphi'}_{\simeq}[rr]   \ar@{^{(}->}[u]  & &   \A^2 \setminus D' \ar@{^{(}->}[u]  }\]
By construction, $\kappa$ sends birationally $\Gamma=\overline{\rho_1(C)}$ onto $\Gamma'=\overline{\rho_1'(C)}$. If $\kappa$ also sends $\Delta$ birationally onto $\Delta'$, then $\varphi'\varphi^{-1}$ extends to a birational map that sends birationally $D$ onto $D'$ and then extends to an automorphism of~$\A^2$ (Proposition~\ref{Prop:Threecasescontraction}). It remains then to show that this is the case.

Using Lemma~\ref{Lem:IsoComplRat}, we take a minimal resolution of the indeterminacies of~$\kappa$:
\[\xymatrix@R=2mm@C=2cm{ & Z\ar[ld]_{\eta}\ar[rd]^{\pi} \\
X\ar@{-->}[rr]^{\kappa}& &X'}\]
where $\eta$ and $\pi$ are the blow-ups of the $\kk$-base-points of $\kappa$ and $\kappa^{-1}$, all being $\k$-rational. We want to show that the strict transforms $\tilde\Delta$ and $\tilde\Delta'$ of $\Delta\subset X$, $\Delta'\subset X'$ are equal. We will do this by studying the strict transform $\tilde{\Gamma}=\tilde{\Gamma}'$ of $\Gamma$ and $\Gamma'$ and its intersection with $\tilde\Delta$ and $\tilde\Delta'$ and with the other components of $B_Z=\eta^{-1}(B\cup \Gamma \cup \Delta)=\pi^{-1}(B'\cup \Gamma'\cup \Delta')$.

Recall that $B_1=B\cup \Delta$, $B_2=B\cup \Gamma$, $B_1'=B'\cup \Delta'$, $B_2'=B'\cup \Gamma'$ are $\k$-trees  and that $C$ is isomorphic to $\Gamma\setminus B_1$ and $\Gamma'\setminus B_1'$ (Proposition~\ref{Prop:XBCD}).

$(i)$ Suppose first that $\Gamma\cap B_1$ contains some $\kk$-points which are not defined over $\k$. None of these points is thus a base-point of~$\kappa$ and each of these points belongs to $\Gamma \cap \Delta$, so $\tilde\Gamma\cap \tilde\Delta$ contains $\kk$-points not defined over $\k$. Since $B_2'$ is a $\k$-tree, $\pi^{-1}(B_2')$ is a $\k$-tree, so $\tilde\Gamma=\tilde{\Gamma}'$ intersects all irreducible components of $B_Z$ into $\k$-points, except maybe $\tilde\Delta'$. This yields $\tilde\Delta=\tilde\Delta'$ as we wanted.

$(ii)$ We can now assume that all $\kk$-points of $\Gamma\cap B_1$ are defined over $\k$, which implies that all intersections of irreducible components of $B_Z$ are defined over $\k$. We will say that an irreducible component of $B_Z$ is \emph{separating} if the union of all other irreducible components is a $\k$-forest (see Definition~\ref{Def:kForest}).

Since $B_1=B\cup \Delta$ is a $\k$-tree, its preimage on $B_Z$ is a $\k$-tree. The union of all components of $B_Z$ distinct from $\tilde\Gamma$ being equal to the disjoint union of $\eta^{-1}(B_1)$ with some $\k$-forest contracted to points of $\Gamma\setminus B_1$, we find that $\tilde\Gamma$ is separating. The same argument shows that $\tilde\Delta$ and $\tilde\Delta'$ are also separating.

It remains then to show that any irreducible component $E\subset B_Z$ which is not equal to $\tilde\Delta$ or $\tilde\Gamma$ is not separating. We use for this the fact that $C\simeq \Gamma\setminus B_1$ is not isomorphic to $\A^1$ or $\A^1\setminus \{0\}$, so the set $\Gamma\cap B_1$ contains at least $3$ points. If $\eta(E)$ is a point $q$, then the complement of $\eta^{-1}(q)$ in $B_Z$ contains a loop, since $\Gamma$ intersects the $\k$-tree $B_1$ into at least two points distinct from $q$. If $\eta(E)$ is not a point, it is one of the components of $B$. We denote by $F$ the union of all irreducible components of $B\cup \Gamma\cup \Delta$ not equal to $\eta(E)$, and prove that $F$ is not a $\k$-forest, since it contains a loop. This is true if $\Delta\cap \Gamma$ contains at least $2$ points. If $\Delta\cap \Gamma$ contains one or less points, then $\Delta\cap B$ contains at least two points, so contains exactly two points, on the two connected components of $B$  which both intersect $\Gamma$ and $\Delta$ (see Proposition~\ref{Prop:XBCD}$(\ref{It5})$). We again get a loop on the union of $\Gamma$, $\Delta$ and of the connected component of $B$ not containing $\eta(E)$. The fact that $F$ contains a loop implies that $\eta^{-1}(F)$ contains a loop, and achieves to prove that $E$ is not separating.
\end{proof}

\subsection{The case of curves isomorphic to $\A^1$ and the proof of Theorem~\ref{RigidityThm}}

To finish the proof of Theorem~\ref{RigidityThm}, we still need to handle the case of curves isomorphic to $\A^1$. The case of lines has
already been
treated in Lemma~\ref{Lemm:Lines}. In characteristic zero, this finishes the study by the Abyhankar-Moh-Suzuki theorem, but in positive characteristic, there are many closed curves of $\A^2$ which are isomorphic to
 $\A^1$, but are not equivalent to lines (these curves are sometimes called ``bad lines'' in the literature). We will show that an open embedding $\A^2\setminus C\hookrightarrow \A^2$ always extends to $\A^2$ if $C$ is isomorphic to 
$\A^1$, but not equivalent to a line.

\begin{lemma}\label{Lemm:Multiplicitesatleast}
Let $n\ge 1$ and let $\Gamma\subset \F_n$ be a  geometrically irreducible closed curve such that $\Gamma\cdot F_n\ge 2$. If there exists a birational map $\F_n\dasharrow \p^2$ that contracts $\Gamma$ to a point $($and perhaps contracts some other curves$)$, then $\Gamma$ is geometrically rational and singular. Moreover, one of the following occurs:
\begin{enumerate}[$(a)$]
\item \label{pt2mFn}
There exists a point $p\in \F_n(\kk)$ such that $2m_p(\Gamma)>\Gamma \cdot F_n$.
\item\label{mtS1F1}
We have $n=1$ and there exists a point $p\in \F_1(\kk)\setminus S_1$ such that $m_p(\Gamma)>\Gamma\cdot S_1$.
\end{enumerate}
\end{lemma}

\begin{proof}
We may assume that $\k=\kk$. Denote by $\psi\colon \F_n\dasharrow \p^2$ the birational map that contracts $C$ to a point (and maybe some other curves). 
The minimal resolution of this map yields a commutative diagram
\[\xymatrix@R=2mm@C=2cm{& X\ar[ld]_{\eta}\ar[rd]^{\pi} \\
\F_n\ar@{-->}[rr]^{\varphi} && \p^2} \]

In $\Pic(\F_n)=\Z F_n\bigoplus \Z S_n$ we write
\[\begin{array}{rcl} \Gamma&=& aS_n+bF_n\\
 -K_{\F_n}&=&2S_n+(2+n)F_n\end{array}\]
for some integers $a,b$. Note that $a=\Gamma\cdot F_n\ge 2$ and that $b-an=\Gamma\cdot S_n\ge 0$. By hypothesis, the strict transform $\tilde{\Gamma}$ of~$\Gamma$ on $X$ is a smooth curve contracted by $\pi$. In particular, $\Gamma$ is rational and the divisor $2\tilde{\Gamma}+aK_X$ is not effective, since \[(2\tilde{\Gamma}+aK_X)\cdot \pi^{*}(L)=aK_{X}\cdot \pi^{*}(L)=a\pi^*(K_{\p^2})\cdot \pi^{*}(L)=aK_{\p^2}\cdot L=-3a<0\] for a general line $L\subset \p^2$.

Denoting by $E_1,\dots,E_r\in \Pic(X)$ the pull-backs of the exceptional divisors blown up by $\eta$ (which satisfy $(E_i)^2=-1$ for each $i$ and $E_i\cdot E_j=0$ for $i\not=j$) we have 
\[\begin{array}{rcllllll} \tilde\Gamma&=& a\eta^{*}(S_n)&+&b\eta^{*}(F_n)&-&\sum_{i=1}^r m_i E_i\\
 -K_{X}&=&2\eta^{*}(S_n)&+&(2+n)\eta^{*}(F_n)&-&\sum_{i=1}^r E_i\\
 2\tilde\Gamma +aK_{X}&=&&&(2b-a(2+n))\eta^{*}(F_n)&+&\sum_{i=1}^r (a-2m_i)E_i\end{array}\]
which implies, since $2\tilde\Gamma +aK_{X}$ is not effective, that either $2b<a(2+n)$ or $2m_i>a$ for some $i$. If $2m_i>a$ for some $i$, we get $(\ref{pt2mFn})$, since the $m_i$ are the multiplicities of~$\tilde{\Gamma}$ at the points blown up by $\eta$.

It remains to study the case where $2m_i\le a$ for each $i$, and where $2b<a(2+n)$. Remembering that $b-an=\Gamma\cdot S_n\ge 0$,
we find
$n\le \frac{b}{a}<\frac{2+n}{2}$, whence $n=1$ and thus $2b<3a$. We then compute
\[\begin{array}{rcllllll} 
 3\tilde\Gamma +bK_{X}&=&(3a-2b)\eta^{*}(S_n)&+&\sum_{i=1}^r (b-3m_i)E_i\end{array}\]
 which is again not effective, since $(3\tilde{\Gamma}+bK_X)\cdot \pi^{*}(L) =b K_X \cdot \pi^{*}(L) = - 3b  <0$ for a general line $L\subset \p^2$, because $b\ge an=a\ge 2$. This implies that there exists an integer $i$ such that $3m_i>b$. Since $2m_i\le a$,
we find
$m_i>b-a=\Gamma\cdot S_1$, which implies $(\ref{mtS1F1})$.
\end{proof}

\begin{proposition}\label{Prop:CurvesA1}
Let $C\subset \A^2$ be a closed curve, isomorphic to $\A^1$ $($over $\k)$. The following are equivalent:
\begin{enumerate}[$(a)$]
\item  \label{1Ceqline}
The curve $C$ is equivalent to a line.
\item\label{2contractA2}
There exists an open embedding $\A^2\setminus C \hookrightarrow \A^2$ which does not extend to an automorphism of~$\A^2$.
\item\label{3contractP2}
There exists a birational map $\p^2 \dasharrow \p^2$ that contracts the curve $C$ $($or its closure$)$ to a $\kk$-point $($and perhaps contracts some other curves$)$. In this statement $\A^2$ is identified with an open subset of $\p^2$ via the standard embedding $\A^2 \hookrightarrow \p^2$.
\end{enumerate}
\end{proposition}

\begin{proof}
The implications  $(\ref{1Ceqline})\Rightarrow(\ref{2contractA2})$ and $(\ref{1Ceqline})\Rightarrow(\ref{3contractP2})$ can be observed, for example, by taking the map $(x,y)\mapsto (x,xy)$, which is an open embedding of~$\A^2\setminus \{x=0\}$ into $\A^2$, which does not extend to an automorphism of~$\A^2$, and whose extension to $\p^2$ contracts the line $x=0$ to a point.

To prove $(\ref{2contractA2})\Rightarrow(\ref{3contractP2})$, we take an open embedding $\varphi\colon\A^2\setminus C\hookrightarrow \A^2$ which does not extend to an automorphism of~$\A^2$ and look at the extension to $\p^2$. By Proposition~\ref{Prop:Threecasescontraction}, either this  contracts $C$, or $C$ is equivalent to a line, in which case $(\ref{3contractP2})$ is true as was shown earlier. 

It remains to prove $(\ref{3contractP2})\Rightarrow(\ref{1Ceqline})$. We apply Lemma~\ref{Lemm:EmbFnminimaldegree}, and obtain an isomorphism $\iota \colon \A^2\iso \F_n \backslash (S_n\cup F_n)$ such that the closure of~$\iota(C)$ in $\F_n$ is a curve $\Gamma$ which satisfies one of the two cases 
$(\ref{SectionCase})$-$(\ref{ng})$ of Lemma~\ref{Lemm:EmbFnminimaldegree}. In case $(\ref{SectionCase})$, the curve is equivalent to a line as it is isomorphic to~$\A^1$ (equivalence $(\ref{LL2})-(\ref{LL3}$) of Lemma~\ref{Lemm:EmbFnminimaldegree}). It remains to study the case where $\Gamma$ satisfies conditions $(\ref{ng})$ of Lemma~\ref{Lemm:EmbFnminimaldegree} (in particular $\Gamma\cdot F_n\ge 2$), and to show that these, together with $(\ref{3contractP2})$, yield a contradiction. We prove that there is no point $p\in \F_n(\kk)$ such that $2m_p(\Gamma)>\Gamma\cdot F_n$. Indeed, since $\Gamma\cdot F_n\ge 2$, such a point would be a singular point of~$\Gamma$, and since
$\Gamma \setminus (S_n\cup F_n)= \iota(C) \simeq C$
is  isomorphic to $\A^1$, $p$ would be a $\k$-point and the unique $\kk$-point of~$\Gamma \cap  (S_n\cup F_n)$.
Moreover, as $\Gamma\cdot F_n\ge 2$, we would find that $p\in F_n$. Since $2m_p(\Gamma)>\Gamma\cdot F_n$ and because $\Gamma$ satisfies conditions $(\ref{ng})$ of Lemma~\ref{Lemm:EmbFnminimaldegree}, the only possibility would be that $n=1$, $p\in F_1\setminus S_1$ and $0<m_p(\Gamma)\le \Gamma\cdot S_1$. This contradicts the fact that $\Gamma \cap (S_1 \cup F_1)$ contains only one $\kk$-point.

Denote by $\psi_0\colon \p^2\dasharrow \p^2$ the birational map that contracts $C$ (and maybe some other curves) to a $\kk$-point. Observe that $\psi_0\circ \iota^{-1}$ yields a birational map $\psi\colon \F_n\dasharrow \p^2$ which contracts $\Gamma$ to a $\kk$-point. As there is no point $p\in \F_n(\kk)$ such that $2m_p(\Gamma)>\Gamma\cdot F_n$, Lemma~\ref{Lemm:Multiplicitesatleast} implies that $n=1$ and that there exists a point $p\in \F_1(\kk)\setminus S_1$ such that $m_p(\Gamma)>\Gamma\cdot S_1$. Again, this point is a $\k$-point, since $C$ is isomorphic to $\A^1$. This contradicts the conditions $(\ref{ng})$ of Lemma~\ref{Lemm:EmbFnminimaldegree}.
\end{proof}

\begin{remark}
If $\k$ is algebraically closed, the equivalence between conditions $(\ref{1Ceqline})$ and $(\ref{3contractP2})$ of Proposition~\ref{Prop:CurvesA1} can also be proved using Kodaira dimension. We introduce the following conditions:
\begin{enumerate}
\item[$(\ref{1Ceqline})'$]  The Kodaira dimension $\kappa (C, \A^2)$ of $C$ is equal to $- \infty$. 
\item[$(\ref{3contractP2})'$]  There exists a birational transformation of $\p^2$ that sends $C$ onto a line.\end{enumerate}
The equivalence between $(\ref{1Ceqline})$ and $(\ref{1Ceqline})'$ follows from \cite[Theorem~2.4.(1)]{Ganong} and the equivalence between $(\ref{1Ceqline})' $ and $ (\ref{3contractP2})'$ is Coolidge's theorem (see e.g. \cite[Theorem~2.6]{KumarMurthy}).
We now recall how the classical equivalence between $(\ref{3contractP2})$ and $(\ref{3contractP2})'$ can be proven. Every simple quadratic birational transformation of $\p^2$ contracts three lines. This proves $(\ref{3contractP2})' \Rightarrow(\ref{3contractP2})$. To get $(\ref{3contractP2}) \Rightarrow(\ref{3contractP2})'$, we take a birational transformation $\varphi$ of $\p^2$ that contracts $C$ to a point and decompose $\varphi$ as $\varphi = \varphi_r \circ \cdots \circ \varphi_1$, where each $\varphi_i$ is a simple quadratic transformation (using the Castelnuovo-Noether factorisation theorem). If $i \geq 1$ is the smallest integer such that $(\varphi_i \circ \cdots \circ \varphi_1) ( C)$ is a $\kk$-point, the curve $(\varphi_{i-1} \circ \cdots \circ \varphi_1) ( C)$ is contracted by $\varphi_i$ and is thus a line.
\end{remark}

\begin{remark}
If the field $\k$ is perfect, then every curve that is geometrically isomorphic to $\A^1$ (i.e.~over $\kk$) is also isomorphic to $\A^1$. This can be seen by embedding the curve in $\p^1$ and considering the complement point, necessarily defined over $\k$. For non-perfect fields, there exist closed curves $C\subset \A^2$ geometrically isomorphic to $\A^1$, but
not isomorphic to $\A^1$ (see \cite{Russell}). Corollary~\ref{Cor:Th1a} shows that every open embedding $\A^2\setminus C\hookrightarrow \A^2$ extends to an automorphism of $\A^2$ for all such curves.
\end{remark}

We can now conclude this section by proving Theorem~\ref{RigidityThm}:

\begin{proof}[Proof of Theorem~$\ref{RigidityThm}$]
We recall the hypotheses of the theorem: we have a geometrically irreducible closed curve  $C\subset \A^2$ and an isomorphism $\varphi\colon \A^2\setminus C\iso \A^2\setminus D$ for some closed curve $C\subset \A^2$. Moreover, $\varphi$ does not extend to an automorphism of $\A^2$. We consider the following three cases:

If $C$ is isomorphic to $\A^1$, then
the implication $(\ref{2contractA2})\Rightarrow (\ref{1Ceqline})$ of Proposition~\ref{Prop:CurvesA1}
shows that $C$ is equivalent to a line and Lemma~\ref{Lemm:Lines}$(\ref{OpenA2minusline})$ implies that the same holds for $D$. In particular, the curves $C$ and $D$ are isomorphic. This achieves the proof of the theorem in this case.

If $C$ is isomorphic to $\A^1\setminus \{0\}$ then so is $D$ by Corollary~\ref{Cor:Th1a}. This also gives the result in this case.

It remains to assume that $C$ is not isomorphic to $\A^1$ or to $\A^1\setminus \{0\}$. Proposition~\ref{Prp:Th1b} shows that the isomorphism $\varphi\colon \A^2\setminus C\iso \A^2\setminus D$ (not extending to an automorphism of $\A^2$) is uniquely determined by $C$, up to left composition by an automorphism of $\A^2$. In particular, there are at most two equivalence classes of curves of~$\A^2$ that have complements isomorphic to $\A^2\setminus C$. Corollary~\ref{Cor:Th1a} gives the existence of isomorphisms $C\simeq \Spec(\k[t,\frac{1}{P}])$ and $D\simeq \Spec(\k[t,\frac{1}{Q}])$ for some square-free polynomials $P,Q\in \k[t]$ that have the same number of roots in $\k$, and also the same number of roots in the algebraic closure of~$\k$. By replacing $\k$ with any field $\k'$ containing $\k$ we obtain the result.
\end{proof}

Corollaries~\ref{Coro:AtMostTwoEq},~\ref{Coro:AtMostOnecounterEx} and~\ref{Coro:Sing} are then direct consequences of Theorem~\ref{RigidityThm}.

\subsection{Automorphisms of complements of curves}   \label{Automorphisms-of-complements-of-curves.subsec}
Another consequence of Theorem~\ref{RigidityThm} is Corollary~\ref{Coro:Index12}, which we now prove:
\begin{proof}[Proof of Corollary~$\ref{Coro:Index12}$]
Recall the hypothesis of the corollary: we start with a geometrically irreducible closed curve $C\subset \A^2$ not isomorphic to $\A^1$ or $\A^1\setminus \{0\}$. We want to show that $\Aut(\A^2,C)$ has index at most $2$ in $\Aut(\A^2 \setminus C)$. If $\varphi_1, \varphi_2$ are automorphisms of $\A^2 \setminus C$ which do not extend to automorphisms of $\A^2$, it is enough to show that $(\varphi_2)^{-1} \varphi_1$ extends to an automorphism of $\A^2$. This follows from Theorem~$\ref{RigidityThm}(\ref{RigidityThmUniqueUp})$.
\end{proof}

\begin{remark}
With the assumptions of Corollary~$\ref{Coro:Index12}$, the group $\Aut(\A^2 \setminus C)$ is a semidirect product of the form $\Aut(\A^2,C) \rtimes \Z / 2\Z$ if and only if there exists an involutive automorphism of $\A^2 \setminus C$ which does not extend to an automorphism of $\A^2$.
\end{remark}

\begin{corollary}\label{FiniteComplement}
If $\k$ is a perfect field and $C \subset \A^2$ is a geometrically irreducible closed curve that is
\begin{enumerate}[$(i)$]
\item\label{not equivalent to a line} not equivalent to a line,
\item \label{not equivalent to a cuspidal curve} not equivalent to a cuspidal curve with equation $x^m-y^n=0$, where $m,n \geq 2$ are coprime integers,
\item \label{not geometrically isomorphic to a punctured line} not geometrically isomorphic to $\A^1 \setminus \{0\}$,
\end{enumerate}
then $\Aut(\A^2 \setminus C)$ is a zero dimensional algebraic group, hence is finite.
\end{corollary}

\begin{proof}
Conditions $(\ref{not equivalent to a line})$-$(\ref{not equivalent to a cuspidal curve})$-$(\ref{not geometrically isomorphic to a punctured line})$ imply that $\Aut(\A^2,C)$ is a zero dimensional algebraic group \cite[Theorem~2]{BS15}. If moreover $C$ is not isomorphic to $\A^1$, then $\Aut(\A^2 \setminus C)$ is also zero dimensional by Corollary~$\ref{Coro:Index12}$. If $C$ is isomorphic to $\A^1$ (but not equivalent to a line by  $(\ref{not equivalent to a line})$), then $\Aut(\A^2 \setminus C) = \Aut(\A^2,C)$ by Proposition~$\ref{Prop:CurvesA1}$.
\end{proof}

\begin{remark}
Let us make a few comments on the group $\Aut(\A^2\setminus C)$ when $C\subset \A^2$ is a
geometrically irreducible closed curve not satisfying the conditions of Corollary~\ref{FiniteComplement}.

\begin{enumerate}
\item[\eqref{not equivalent to a line}]
If $C$ is equivalent to a line, we may assume without loss of generality that $C$ is the line $x=0$. Then, $\Aut(\A^2\setminus C)$ is described in Lemma~$\ref{Lemm:Lines}$.
\item[\eqref{not equivalent to a cuspidal curve}]
 If $C$ does not satisfy \eqref{not equivalent to a cuspidal curve}, we may assume that $C$ has equation $x^m- y^n=0$, where $m,n \geq 2$ are coprime integers. Since the curve $C$ is singular, we have $\Aut(\A^2\setminus C)=\Aut(\A^2,C)$ by Corollary~$\ref{corollary: in the singular case, the embedding extends to an automorphism of the plane}$. Moreover, we have $\Aut(\A^2,C)=\{(x,y) \mapsto (t^n x,t^my)\mid t \in \k^*\}$ by \cite[Theorem~2(ii)]{BS15}.
\item[\eqref{not geometrically isomorphic to a punctured line}(a)]
If $C$ is geometrically isomorphic to $\A^1\setminus \{0\}$, but not isomorphic to $\A^1\setminus \{0\}$, then $\Aut(\A^2,C)$ has index $1$ or $2$ in $\Aut(\A^2\setminus C)$ by Corollary~$\ref{Coro:Index12}$. The group $\Aut(\A^2,C)$ is then an algebraic group of dimension $\le 1$ by \cite[Theorem~2]{BS15}, so the same holds for $\Aut(\A^2\setminus C)$. An example of dimension $1$ is given by the curve of equation $x^2+y^2=1$, in the case where $\k=\R$ (see \cite[Theorem~2(iv)]{BS15}).
\item[\eqref{not geometrically isomorphic to a punctured line}(b)]
If $C$ is isomorphic to $\A^1 \setminus \{0\}$, we do not have a complete description of $\Aut(\A^2\setminus C)$. The simplest cases where $C$ has equation $x^my^n-1$, where $m,n\ge 1$ are coprime, can be completely described. In particular, $\Aut(\A^2\setminus C)$ contains elements of arbitrarily large degree.
\end{enumerate}
\end{remark}

\section{Families of non-equivalent embeddings} \label{SecFamilies}

In this section, we study mainly the curves of $\A^2$ given by an equation of the form
\[a(y)x+b(y)=0\] 
where $a,b\in \k[y]$ are coprime polynomials such that $\deg b < \deg a$. This will lead us to the proof of  Theorem~\ref{NegativityThm}.

These curves already appeared in Lemma~\ref{Lemm:EmbFnminimaldegree}, where we proved in particular that they are isomorphic to $\A^1$ if and only if $a(y)$ is a constant (Lemma~\ref{Lemm:EmbFnminimaldegree}$(\ref{LL1})$-$(\ref{LL3})$). Actually, we have the following obvious and stronger result:

\begin{lemma} \label{lemma: algebra of regular functions of the curve a(y)x+ b(y) = 0}
Let $C \subset \A^2$ be the irreducible curve given by the equation
\[ a(y) x + b(y) =0,\]
where $a,b \in \k [y]$ are coprime polynomials and $a$ is nonzero. Then, the algebra of regular functions on $C$ is isomorphic to
$\k [y, 1/ a(y) ]$.
\end{lemma}

\begin{proof}
The algebra of regular functions on $C$ satisfies
\[ \k [ C ] = \k [x,y] / (a(y) x + b(y)) \simeq \k [ y, -b(y)/a(y)] = \k [ y, 1/ a(y) ] ,\]
where the last equality comes from the fact that there exist $c,d\in \k[y]$ with $ad-bc=1$, which implies that $\frac{1}{a}=\frac{ad-bc}{a}=d-c\cdot \frac{b}{a}\in \k[y,\frac{b}{a}]$.
\end{proof}

\subsection{A construction using elements of $\SL_2(\k[y])$}

\begin{lemma} \label{lemma:SL2}
For each matrix  $\left( \begin{array}{cc} a(y) & b(y) \\  c(y) & d(y) \end{array} \right) \in \SL_2 ( \k [ y] )$, we have an isomorphism
\[\begin{array}{rrcl}
\varphi \colon & \A^2\setminus C& \iso &\A^2\setminus D\\
&(x,y) &\mapsto &\big( \frac{c(y) x + d(y)}{a(y) x +b(y)}, y \big)\end{array}\]
where $C,D\subset \A^2$ are given by $a(y) x + b(y) = 0$ and $a(y) x - c(y) = 0$ respectively.
\end{lemma}

\begin{proof}
Note first that $\varphi$ is a birational transformation of $\A^2$, with inverse $\psi\colon (x,y)\mapsto (\frac{-b(y)x+d(y)}{a(y)x-c(y)},y)$. It remains to prove that the isomorphism $\varphi ^* \colon \k (x,y) \to \k (x,y)$, $x \mapsto \frac{c x + d}{a x +b}$, $y \mapsto y$ induces an isomorphism $\k[ x, y, \frac{1}{ax-c} ] \to \k[ x, y,\frac{1}{ax+b} ].$
This follows from the equalities:
\[ 
\begin{array}{lllrrrr}
\varphi^{*}(x) = \frac{c x + d}{a x +b},  & \varphi^{*}(y) = y, & \varphi^{*} \big( \frac{1}{ax-c} \big) = ax+b & \text{ and} \\
\\
\psi^{*}(x) = \frac{-b x + d}{a x - c},       & \psi^{*}(y) = y,       & \psi^* \big( \frac{1}{ax+b} \big) = a x -c.
\end{array}
\]
\par 
\vspace{-6mm}
\end{proof}
\par

The curves $C$ and $D$ of Lemma~\ref{lemma:SL2} are always isomorphic thanks to Lemma~\ref{lemma: algebra of regular functions of the curve a(y)x+ b(y) = 0}. We now prove that they are in general not equivalent.

\begin{lemma}  \label{Lemm:IsoCDEqui}
Let $C_1,C_2 \subset \A^2$ be two geometrically irreducible closed curves given by 
\[a_1(y)x+b_1(y)=0\mbox{ and } a_2(y)x+b_2 (y)=0\]
respectively, for some polynomials $a_1,a_2,b_1,b_2 \in \k [y]$ such that $\deg a_1 > \deg b_1 \ge 0$ and $\deg a_2 >\deg b_2 \ge 0$. Then, the curves $C_1$ and $C_2$ are equivalent if and only if there exist constants $\alpha,\lambda,\mu \in \k^*$ and $\beta\in \k$ such that 
\[ a_2 (y)=\lambda \cdot  a_1 (\alpha y+\beta), \quad b_2 (y)=\mu \cdot b_1(\alpha y+\beta).\]
\end{lemma}
\begin{proof}
We first observe that if $a_2 (y)=\lambda\cdot  a_1 (\alpha y+\beta)$ and $b_2 (y)=\mu \cdot b_1 (\alpha y+\beta)$ for some $\alpha,\lambda,\mu \in \k^*$, $\beta\in \k$, then the automorphism $(x,y)\mapsto (\frac{\lambda}{\mu} x,\alpha y+\beta)$ of $\A^2$ sends $C_2$ onto~$C_1$.

Conversely, we assume the existence of $\varphi\in \Aut(\A^2)$ that sends $C_2$ onto $C_1$ and want to find $\alpha,\lambda,\mu \in \k^*$, $\beta\in \k$ as above. Writing $\varphi$ as $(x,y)\mapsto (f(x,y),g(x,y))$ for some polynomials $f,g\in \k[x,y]$,
we get
\begin{equation}\label{eqmuafg} \mu \big( a_1(g) f+b_1(g) \big) =  a_2(y) x+b_2(y)  \end{equation}
for some $\mu\in \k^*$.

$(i)$
If $g\in \k[y]$, the fact that $\k[f,g]=\k[x,y]$ implies that 
$g=\alpha y+\beta, f=\gamma x + s(y)$ for some $\alpha,\gamma\in \k^{*},\beta\in \k$ and $s(y) \in \k[y]$. This yields $a_1 (g) f+b_1 (g)=a_1(g) (\gamma x+s(y) )+b_1 (g)$, so that equation~(\ref{eqmuafg}) gives:
\[ a_2 = \mu \gamma \cdot a_1 (g), \quad b_2 = \mu \cdot \big( a_1 (g) s(y) +b_1(g) \big). \]
 This shows in particular that $\deg a_1 =\deg a_2$, whence $\deg b_2  < \deg a_1 (g)$. Since $\deg b_1(g) < \deg a_1(g)$, we find that $s=0$, and thus that $b_2 =\mu \cdot  b_1(g)$, as desired. This concludes the proof, by choosing $\lambda=\mu\gamma$.

$(ii)$ It remains to consider the case where  $g\notin \k[y]$, which corresponds to $\deg_x(g)\ge 1$. We have $\deg_x a_1(g) =\deg a_1 \cdot \deg_x(g)>\deg b_1\cdot \deg_x(g)=\deg_x b_1(g)$, which implies that $\deg_x  \big( a_1 (g)f+b_1(g) \big) = \deg(a_1)\cdot \deg_x(g)+\deg_x(f)$. Equation~(\ref{eqmuafg}) shows that this degree is~$1$, and since $\deg a_1 \ge 1$, we find $\deg a_1 =1$. Similarly, the automorphism sending $C_1$ onto $C_2$ satisfies the same condition, so $\deg a_2 =1$. This implies that $b_1, b_2 \in \k^*$. There thus exist some $\alpha,\lambda,\mu \in \k^*$, $\beta\in \k$ such that $a_2 (y)=\lambda \cdot  a_1 (\alpha y+\beta)$ and $b_2 (y)=\mu \cdot b_1(\alpha y+\beta)$.
\end{proof}

\begin{proposition}  \label{Prop:Negativity3}
For each polynomial $f\in \k[t]$ of degree $\ge 1$, there exist two closed curves $C,D\subset \A^2$, both isomorphic to $\Spec(\k[t,\frac{1}{f}])$, that are non-equivalent and have isomorphic complements. Moreover, the set of equivalence classes of the curves $C$ appearing in such pairs $(C,D)$ is infinite.
\end{proposition}

\begin{proof}
We choose an irreducible polynomial $b\in \k[t]$ which does not divide $f$. For each $n\ge 1$ such that $\deg(f^n)>2\deg (b)$, we then choose two polynomials $c,d\in \k[t]$ such that $f^nd-bc=1$ (this is possible since $\gcd(f^n,b)=1$). Replacing $c,d$ by $c+\alpha f^n,d+\alpha b$, we may moreover assume that $\deg c  < \deg f^n$.
The curves $C_n,D_n\subset \A^2$ given by $f(y)^nx+b(y)=0$ and $f(y)^nx-c(y)=0$ are both isomorphic to $\Spec(\k[t,\frac{1}{f^n}])=\Spec(\k[t,\frac{1}{f}])$ by Lemma~\ref{lemma: algebra of regular functions of the curve a(y)x+ b(y) = 0} and have isomorphic complements by Lemma~\ref{lemma:SL2}. Moreover, as $\deg bc=\deg(f^nd-1)\ge \deg(f^n)>2\deg(b)$, we find that $\deg c>\deg b$, which implies by Lemma~\ref{Lemm:IsoCDEqui} that $C_n$ and $D_n$ are not equivalent. Moreover, the curves $C_n$ are all non-equivalent, again by Lemma~\ref{Lemm:IsoCDEqui}.
\end{proof}

\subsection{Curves isomorphic to $\A^1\setminus \{0\}$}\label{SubSec:A1minus0}

We consider now families of curves in $\A^2$ of the form $xy^{d}+b(y)=0$, for some $d\ge 1$ and some polynomial $b(y) \in \k[y]$ satisfying $b(0) \neq 0$. Note that all these curves are isomorphic to $\Spec(\k[y, \frac{1}{y^d}])=\Spec(\k[y,\frac{1}{y}])\simeq \A^1\setminus \{0\}$ by Lemma~\ref{lemma: algebra of regular functions of the curve a(y)x+ b(y) = 0}.

\begin{lemma}  \label{Lem: b and c}
Let $d\ge 1$ be an integer and $b(y) \in \k[y]$ be a polynomial satisfying $b(0) \neq 0$. We define $D_b \subset \A^2$ to be the curve given by the equation
\[xy^{d}+b(y) =0\]
and $\varphi_b$ to be the birational endomorphism of $\A^2$ given by
\[ \varphi_b (x,y) =(xy^{d}+b(y), y). \]
Denote by $L_x$, resp.~$L_y$, the line in $\A^2$ given by the equation $x=0$, resp.~$y=0$.

\begin{enumerate}[$(1)$]
\item \label{The birational map varphi_b}
The transformation $\varphi_b$ induces an automorphism of $\A^2 \setminus L_y$ and an isomorphism  \[\A^2 \setminus (L_y \cup D_b)\iso\A^2 \setminus ( L_y \cup L_x).\]

\item  \label{definition of c}
Assume now that $b$ has degree $\le d-1$ and fix an integer $m\ge 1$.
Then, there exists a unique polynomial $c \in \k[y]$ of degree $\le d-1$  satisfying
\begin{equation} \label{Equation relating b and c}  b(y)  \equiv c(yb(y)^m)  \pmod {y^{d}}.\end{equation}
Furthermore, we have $c(0) \neq 0$.

\item \label{the complements of D_b and D_c are isomorphic}
Define the birational transformations $\tau$ and $\psi_{b,m}$ of $\A^2$ by
\[ \tau (x,y) = (x,xy) \text{ and } \psi_{b,m} = (\varphi_{c})^{-1}\tau^m \varphi_b.\]
Then, $\psi_{b,m}$ induces an isomorphism  $\A^2 \setminus D_b\iso\A^2 \setminus D_c$ whose expression is
\[ \psi_{b,m} (x,y)
= \left( {\dis \frac{x+\lambda +y f(x,y)} { \left(\vphantom{\Big )}  xy^{d}+b(y) \right)^{md}} }, \, y \left(\vphantom{\Big )}xy^{d}+b(y)\right)^m\right), \]
for some constant $\lambda \in \k$ and some polynomial $f \in \k[x,y]$ $($depending on $b$ and $m)$.

\item\label{psi_m and psi_n are not equivalent when m and n are different}
Fixing the polynomial $b$, all open embeddings $\A^2 \setminus D_b \hookrightarrow \A^2$ given by $\psi_{b,m}$, $m \geq 1$, are non-equivalent.
\end{enumerate}
\end{lemma}

\begin{proof}

$(\ref{The birational map varphi_b})$: 
The automorphism $(\varphi_b)^*$ of $\k(x,y)$ satisfies
\[ (\varphi_b)^*(x) = xy^d + b(y) \text{ and } (\varphi_b)^*(y) =y.  \]
The result follows from the following two equalities:
\[ \begin{array}{lll}
(\varphi_b)^* ( \k[x,y,\frac{1}{y}] ) &=& \k [ xy^d + b(y), y, \frac{1}{y}] = \k [x,y,\frac{1}{y}] \quad \text{and}\\
\vphantom{\Big)} (\varphi_b)^* ( \k[x,y,\frac{1}{x}, \frac{1}{y}] )& =& \k [ xy^d + b(y),\frac{1}{xy^d + b(y)},  y, \frac{1}{y}] = \k [x,y, \frac{1}{y}, \frac{1}{xy^{d}+b(y) }] .\end{array} \]

$(\ref{definition of c})$:
Since $b(0) \neq 0$, the endomorphism of the algebra $\k[y]/(y^{d})$ defined by  $y \mapsto y b(y)^m$ is an automorphism. If the inverse automorphism is given by $y \mapsto u(y)$, note that \eqref{Equation relating b and c} is equivalent to $c(y) \equiv b ( u(y)) \pmod {y^{d}}$. This determines uniquely the polynomial $c$. Finally, replacing $x$ by zero in \eqref{Equation relating b and c}, we get $c(0) = b(0) \neq 0$.

$(\ref{the complements of D_b and D_c are isomorphic})$: Since $\tau$ induces an automorphism of $\A^2 \setminus ( L_y \cup L_x)$, assertion~$(\ref{The birational map varphi_b})$ implies that $\psi$ induces an isomorphism $\A^2 \setminus (L_y \cup D_b) \iso \A^2\setminus (L_y \cup D_c)$ (this would be true for any choice of $c$). It remains to see that the choice of $c$ which we have made implies that $\psi$ extends to an isomorphism $\A^2\setminus D_b \iso \A^2\setminus D_c$ of the desired form.

Since $( \varphi_c)^{-1} (x,y) = \left( \frac{x-c(y)}{y^d}, y \right)$, $\tau^m (x,y) = (x,x^my)$, and $\psi_{b,m}=(\varphi_{c})^{-1}\tau^m \varphi_b$, we get:
\begin{equation} \label{General expression of psi} \begin{array}{rcl}
\psi_{b,m} (x,y) &=& (\varphi_{c})^{-1}\tau^m(xy^{d}+b(y), y)\\
&=&\left( \frac{xy^d + b(y) - c( y  \Delta)}{ y^d  \Delta^d } , y \Delta  \right), \text{ with }\Delta=( xy^d + b(y) )^m.
\end{array}
\end{equation}
To show that $\psi_{b,m}$ has the desired form, we use $b(y)  \equiv c(yb(y)^m)  \pmod {y^{d}}$ (equation \eqref{Equation relating b and c}), which yields $\lambda\in \k$ such that $b(y)\equiv  c(yb(y)^m)+ \lambda y^d \pmod{y^{d+1}}$. Since $y\Delta\equiv yb(y)^m\pmod{y^{d+1}}$, we get $b(y)\equiv c(y\Delta)+\lambda y^d\pmod{y^{d+1}}$. There is thus $f\in \k[x,y]$ such that
\[xy^d + b(y) - c( y  \Delta)=y^d(x+\lambda +yf(x,y)).\]
This yields the desired form for $\psi_{b,m}$ and shows that $\psi_{b,m}$ restricts to the automorphism $x\mapsto x+\lambda$ on $L_y$ and then restricts to an isomorphism $\A^2\setminus D_b \iso \A^2\setminus D_c$.

$(\ref{psi_m and psi_n are not equivalent when m and n are different})$:
It suffices to check that for $m > n \geq 1$ the birational transformation $\theta=\psi_{b,n} \circ (\psi_{b,m})^{-1}$ of $\A^2$ does not correspond to an automorphism of $\A^2$. Setting $l= m- n  \geq 1$ and denoting by $c_m$ and $c_n$ the elements of $\k[y]$ associated to $b$ and to the integers $m$ and $n$ respectively, we get
\[ \theta= \left((\varphi_{c_n})^{-1} \tau^n \varphi_b\right) \circ  \Big( (\varphi_{c_m} )^{-1} \tau^m \varphi_b \Big) ^{-1} = (\varphi_{c_n })^{-1} \tau^{-l} \varphi_{c_m }. \]
The second component of $\theta(x,y)$ is thus equal to the second component of $\tau^{-l}\varphi_{c_{m}}(x,y)$ which is $\frac{y}{(xy^d+c_m(y))^l}\in \k(x,y)\setminus \k[x,y]$. This shows that $\theta$ is not an automorphism of $\A^2$ (and not even an endomorphism) and completes the proof.
\end{proof}

\begin{remark} \label{remark: easy counterexamples to the complement problem in the reducible case}
Note that Lemma~\ref{Lem: b and c}(\ref{The birational map varphi_b}) provides an isomorphism $\A^2 \setminus (L_y \cup D_b)\iso\A^2 \setminus ( L_y \cup L_x)$ where the reducible curves $(L_y \cup D_b)$ and $(L_y \cup L_x)$ are not isomorphic. Indeed, the reducible curve $(L_y \cup D_b)$ has two connected components (since $L_y \cap D_b= \emptyset$), while the reducible curve $ (L_y \cup L_x)$ is connected (since $ L_y \cap L_x \neq \emptyset$). As noted in \cite{Kraft96}, this kind of easy example explains why the complement problem in $\A^n$ has only been formulated for irreducible hypersurfaces.
\end{remark}

\begin{remark}
Geometrically, the construction of Lemma~\ref{Lem: b and c}(\ref{the complements of D_b and D_c are isomorphic}) can be interpreted as follows: the birational morphism $\varphi_b \colon (x,y)\mapsto (xy^{d}+b(y) ,y)$ contracts the line $y=0$ to the point $(b(0),0)$. If $d=1$ then $\varphi_b$ just sends the line onto the exceptional divisor of $(b(0),0)$. If $d \ge 2$, it sends the line onto the exceptional divisor of a point in the $(d-1)$-st neighbourhood of $(b(0),0)$. The coordinates of these points are determined by the polynomial~$b$. The fact that $\tau^m \colon (x,y)\mapsto (x,x^my)$ contracts the line $x=0$ implies that $\psi_{b,m}$ contracts the curve $D_b$ given by $xy^{d}+b(y)=0$. Moreover, $\tau^m$ fixes the point $(b(0),0)$ and induces a local isomorphism around it, hence acts on the set of infinitely near points. This action changes the polynomial $b$ and replaces it by another one, which is the polynomial $c=c_{b,m}$ provided by Lemma~\ref{Lem: b and c}(\ref{definition of c}).
\end{remark}

\begin{proposition} \label{Prop: infinitely many cuves whose complements admit infinitely many open embeddings in the plane}
 There exists an infinite sequence of curves $C_i\subset \A^2$, $i\in \N$, all pairwise non-equivalent, all isomorphic to $\A^1\setminus \{0\}$ and such that for each $i$ there are infinitely many open embeddings $\A^2\setminus C_i\hookrightarrow \A^2$, up to automorphisms of $\A^2$.
\end{proposition}

\begin{proof}
It suffices to choose the curve $C_i$ given by $xy^{i+2}+y+1$, for each $i\ge 2$. These curves are all isomorphic to $\A^1\setminus \{0\}$ by Lemma~\ref{lemma: algebra of regular functions of the curve a(y)x+ b(y) = 0} and are pairwise non-equivalent by Lemma~\ref{Lemm:IsoCDEqui}. The existence of infinitely many open embeddings $\A^2\setminus C_i \hookrightarrow \A^2$, up to automorphisms of $\A^2$, is then ensured by Lemma~\ref{Lem: b and c}(\ref{psi_m and psi_n are not equivalent when m and n are different}).
\end{proof}

One can compute the polynomial $c= c_{b,m}$ provided by Lemma~\ref{Lem: b and c}(\ref{definition of c}), in terms of $b$ and $m$, and find explicit formulas. We obtain in particular the following result:

\begin{lemma} \label{lemma: infinitely many non-equivalent curves having the same complement in characteristic 0}
For each $\mu\in \k$ define the curve $C_\mu\subset \A^2$ by \[xy^3+\mu y^2+y+1=0.\] 
Then, there exists an isomorphism $\A^2\setminus C_\mu \iso \A^2\setminus C_{\mu-1}$.
In particular, if $\mathrm{char}(\k)=0$, we obtain infinitely many closed curves of $\A^2$, pairwise non-equivalent, which have isomorphic complements.
\end{lemma}

\begin{proof}
The isomorphism between $\A^2\setminus C_\mu$ and $\A^2\setminus C_{\mu-1}$ follows from Lemma~\ref{Lem: b and c} with $d=3$, $m=1$,  $b= \mu y^2 + y +1$ and $c= (\mu -1) y^2 + y +1$.

To get the last statement,
we assume
that $\mathrm{char}(\k)=0$ and
observe
that the affine surfaces $\A^2\setminus C_n$ are all isomorphic for each $n\in \Z$. To show that the curves $C_n$, $n\in \Z$ are pairwise non-equivalent, we apply Lemma~\ref{Lemm:IsoCDEqui}: for $m,n\in \Z$, the curves $C_m$ and $C_{n}$ are equivalent only if there exist $\alpha,\lambda,\mu \in \k^*$, $\beta\in \k$ such that 
\[ y^3=\lambda\cdot  (\alpha y+\beta)^3,\ m y^2+y+1=\mu \cdot \big( n(\alpha y+\beta)^2+(\alpha y+\beta)+1 \big) .\]
The first equality gives $\beta=0$, so that the second one becomes $m y^2+y+1=\mu \cdot (n\alpha^2 y^2+\alpha y+1)$. We finally obtain  $\mu=1$, $\alpha=1$ and thus $m=n$, as we wanted.
\end{proof}

If $\mathrm{char}(\k)=p>0$, Lemma~\ref{lemma: infinitely many non-equivalent curves having the same complement in characteristic 0} only gives $p$ non-equivalent curves that have isomorphic complements. We can get more curves by applying Lemma~\ref{Lemm:IsoCDEqui} to polynomials of higher degree:

\begin{lemma} \label{lemma: unboundeness of non-equivalent curves having the same complements in positive characteristic}
 For each integer $n\ge 1$ there exist curves $C_1,\dots,C_n\subset \A^2$, all isomorphic to $\A^1\setminus \{0\}$, pairwise non-equivalent, such that all surfaces $\A^2\setminus C_1$, $\dots$, $\A^2\setminus C_n$ are isomorphic.
\end{lemma}
\begin{proof}
The case where $\mathrm{char}(\k)=0$ is settled by Lemma~\ref{lemma: infinitely many non-equivalent curves having the same complement in characteristic 0} so we may assume that  $\mathrm{char}(\k)=p \geq 2$. Set $b(y) = 1+ y$ and $d = p ^n +2$. For each integer $i$ with $1 \leq i \leq n$, we apply Lemma~\ref{Lem: b and c}(\ref{definition of c}) with $m =p^i$. Hence, there exists a unique polynomial $c_i \in \k[y]$ of degree $\le d-1$  satisfying
\begin{equation} \label{Definition of c_i}  b(y)  \equiv c_i(yb(y)^{p^i})  \pmod {y^{d}}.\end{equation}
Let $C_i\subset \A^2$ be the curve given by the equation
\[xy^d+c_i(y)=0.\]
By Lemma~\ref{Lem: b and c}(\ref{the complements of D_b and D_c are isomorphic}), all surfaces $\A^2\setminus C_1$, $\dots$, $\A^2\setminus C_n$ are isomorphic to $\A^2 \setminus D$, where $D \subset \A^2$ is given by
\[xy^d+b(y) =0.\]
It remains to check that $C_1, \ldots, C_n$ are pairwise non-equivalent. Assume therefore that $C_i$ and $C_j$ are equivalent. By Lemma~\ref{Lemm:IsoCDEqui}, there exist $\alpha,\lambda,\mu \in \k^*$, $\beta\in \k$ such that 
\[ y^d = \lambda \cdot  (\alpha y+\beta)^d,\quad c_j(y) = \mu \cdot   c_i(\alpha y+\beta). \]
The first equality gives $\beta=0$, so that we get:
\begin{equation} \label{c_i and c_j} c_j(y) = \mu \cdot c_i(\alpha y).\end{equation}
However, by equation \eqref{Definition of c_i} we have
\[ 1+ y \equiv c_i ( y + y^{p^i + 1} )  \pmod {y^{p^i + 2}} \]
and this equation admits the unique solution
\[ c_i  =  1 + y -  y^{p^i + 1}  + \text{(terms of higher order)}. \]
(Unicity follows for example again from Lemma~\ref{Lem: b and c}(\ref{definition of c})). Hence, looking at equation \eqref{c_i and c_j} modulo $y^2$,  we obtain $1+y= \mu (1 + \alpha y)$, so that $\alpha = \mu =1$. Equation \eqref{c_i and c_j} finally yields $c_i = c_j$, so that the above (partial) computation of $c_i$ gives us $i=j$.
\end{proof}

The proof of Theorem~\ref{NegativityThm} is now complete:
\begin{proof}[Proof of Theorem~$\ref{NegativityThm}$]
Part~$(\ref{Negativity1})$ corresponds to Proposition~\ref{Prop: infinitely many cuves whose complements admit infinitely many open embeddings in the plane}. Part~$(\ref{Negativity2})$ is given by 
Lemma~\ref{lemma: infinitely many non-equivalent curves having the same complement in characteristic 0} 
($\mathrm{char}(\k)=0$)  and Lemma~\ref{lemma: unboundeness of non-equivalent curves having the same complements in positive characteristic} ($\mathrm{char}(\k)>0$).
Part~$(\ref{Negativity3})$ corresponds to Proposition~\ref{Prop:Negativity3}.
\end{proof}

\section{Non-isomorphic curves with isomorphic complements}\label{Sec:NonIsoCurves}
\subsection{A geometric construction}

We begin with the following fundamental construction:

\begin{proposition}   \label{proposition: existence of two curves with isomorphic complements and prescribed rings of functions}
For each polynomial $P\in \k[t]$ of degree $d \geq 3$ and each $\lambda\in \k$ with $P( \lambda ) \not=0$, there exist two closed curves $C,D\subset \A^2$ of degree $d^2-d+1$ such that $\A^2\setminus C$ and $\A^2\setminus D$ are isomorphic and such that the following isomorphisms hold:
\[ C\simeq \Spec \Big( \k[t,\frac{1}{P}] \Big)  \text{ and } D \simeq \Spec \Big( \k[t,\frac{1}{Q}] \Big), \text{ where } Q(t)=P \Big( \lambda +\frac{1}{t} \Big) \cdot t^d.  \]
\end{proposition}

\begin{proof}
The polynomial $P_d(x,y):=P(\frac{x}{y})y^d\in \k[x,y]$ is a homogeneous polynomial of degree $d$ such that $P_d(x,1)=P(x)$. Let then $\Gamma,\Delta,L,R\subset \p^2$ be the curves given by the equations
\[\Gamma: y^{d-1}z=P_d(x,y),\ \ \Delta: z=0,\ \  L: x=\lambda y,\ \ R:y=0.\]
By construction, $P_d$ is not divisible by $y$. Moreover, the two lines $L$ and $\Delta$ satisfy $L\cap \Gamma=\{p_1,q_1 \}$ where $p_1=[\lambda:1:P(\lambda)]$, $q_1=[0:0:1]$ 
and $\Delta$ does not pass through $p_1$ or $q_1$.

Note that $\Gamma\subset \p^2$ is a cuspidal rational curve, that the point $q_1= [0:0:1] \in \p^2(\k)$ has multiplicity $d-1$ on
$\Gamma$, and is therefore the unique singular point of this curve $($this follows for example from the genus formula of a plane curve$)$.
The situation is then as follows.
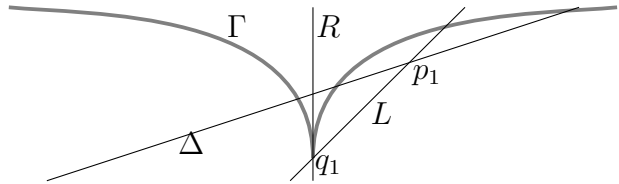
\[
\begin{tikzpicture}
\draw [line width=0.5mm, gray ] (0,0) ..controls +(0,2) and +( -1,-0.1).. (4,2);
\draw [line width=0.5mm, gray ] (0,0) ..controls +(0,2) and +( 1,-0.1).. (-4,2);
\draw (-0.3,-0.3) -- ( 2,2);
\draw (0,-0.3) -- ( 0,2);
\draw (-3.5,-0.3) -- (3.5,2);
\draw ( -1,1.75) node{$\Gamma$};
\draw ( 0.2,1.75) node{$R$};
\draw ( 0.9,0.6) node{$L$};
\draw ( -1.6,0.2) node{$\Delta$};
\draw ( 0.2,-0.1) node{$q_1$};
\draw (1.5,1.1 ) node{$p_1$};
\end{tikzpicture}
\]

Denote by $\pi\colon X\to \p^2$ the birational morphism given by the blow-up of~$p_1$, $q_1$, followed by the blow-up of the points $p_2$,\dots,$p_{d-1}$ and $q_2$,\dots,$q_{d}$ infinitely near $p_1$ and $q_1$ respectively and all belonging to the strict transform of~$\Gamma$. Denote by $\tilde{\Gamma}$, $\tilde{\Delta}$, $\tilde{L}$, $\tilde{R}$, $\mathcal{E}_1$,\dots, $ \mathcal{E}_{d-1}$, $\mathcal{F}_1$,\dots,$\mathcal{F}_{d} \subset X$ the strict transforms of~$\Gamma$, $\Delta,L, R$ and of the exceptional divisors above~$p_1$,\dots,$p_{d-1}$, $q_1$, \dots,$q_{d}$. Consider the tree $($which is in fact a chain$)$
\[ B=\tilde{L}\cup \bigcup\limits_{i=1}^{d-2}\mathcal{E}_i\cup \bigcup\limits_{i=1}^{d}\mathcal{F}_i.\]

We now prove that the situation on $X$ is as in 
the symmetric diagram \eqref{SymmetricD},
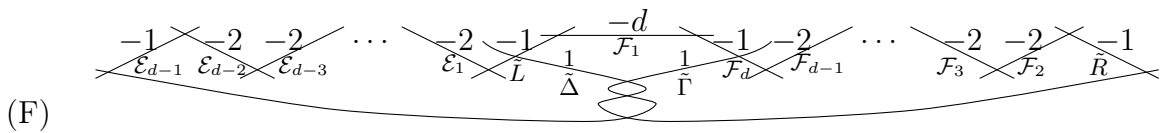
\begin{equation}  \label{SymmetricD}
\begin{tikzpicture}[scale=0.76]
\draw (-1.5,0) -- (1.5,0);
\draw (0, 0.25) node{$-d$};
\draw (-0.9,0.15) -- (-2.7,-0.75);
\draw (-1.95, -0.1) node{$-1$};
\draw (-3.9,0.15) -- (-2.1,-0.75);
\draw (-3, -0.1) node{$-2$};
\draw (0.9,0.15) -- (2.7,-0.75);
\draw (1.8, -0.1) node{$-1$};
\draw (3.9,0.15) -- (2.1,-0.75);
\draw (2.85, -0.1) node{$-2$};
\draw (4.9,0.15) -- (6.7,-0.75);
\draw (5.8, -0.1) node{$-2$};
\draw (7.9,0.15) -- (6.1,-0.75);
\draw (6.85, -0.1) node{$-2$};
\draw (-4.9,0.15) -- (-6.7,-0.75);
\draw (-5.95, -0.1) node{$-2$};
\draw (-7.9,0.15) -- (-6.1,-0.75);
\draw (-7, -0.1) node{$-2$};
\draw (-4.4, -0.1) node{$\dots$};
\draw (4.4, -0.1) node{$\dots$};
\draw (-7.4,0.15) -- (-9.2,-0.75);
\draw (-8.45, -0.1) node{$-1$};
\draw (7.4,0.15) -- (9.2,-0.75);
\draw (8.45, -0.1) node{$-1$};
\draw  plot [smooth] coordinates {(-9.2,-0.6)(-5,-1.3) (-0.6,-1.5) (0.5,-1.2) (-0.3,-0.9) (2,-0.35) (2.5,-0.1)};
\draw  plot [smooth] coordinates {( 9.2,-0.6)(5,-1.3)(0.6,-1.5) (-0.5,-1.2) (0.3,-0.9) (-2,-0.35) (-2.5,-0.1)};
\draw (-1, -0.85) node{\scriptsize$\tilde\Delta$};
\draw (1, -0.85) node{\scriptsize$\tilde\Gamma$};
\draw (-1, -0.4) node{\scriptsize$1$};
\draw (1, -0.4) node{\scriptsize$1$};
\draw (-3, -0.5) node{\scriptsize$\mathcal{E}_1$};
\draw (-5.6, -0.55) node{\scriptsize$\mathcal{E}_{d-3}$};
\draw (-8.1, -0.55) node{\scriptsize$\mathcal{E}_{d-1}$};
\draw (-7, -0.55) node{\scriptsize$\mathcal{E}_{d-2}$};
\draw (-1.9, -0.6) node{\scriptsize$\tilde{L}$};
\draw (0, -0.2) node{\scriptsize$\mathcal{F}_1$};
\draw (1.9, -0.6) node{\scriptsize$\mathcal{F}_d$};
\draw (3.3, -0.5) node{\scriptsize$\mathcal{F}_{d-1}$};
\draw (5.6, -0.55) node{\scriptsize$\mathcal{F}_{3}$};
\draw (8.15, -0.5) node{\scriptsize$\tilde{R}$};
\draw (7, -0.55) node{\scriptsize$\mathcal{F}_{2}$};
\end{tikzpicture}\end{equation}
where all curves are isomorphic to $\p^1$, all intersections indicated are transversal and consist in exactly one $\k$-point, except for $\tilde{\Gamma}\cap \tilde{\Delta}$, which can be more complicated $($the picture shows only the case where we get $3$ points with transversal intersection$)$.

Blowing up once the singular point $q_1$ of $\Gamma$, the strict transform of $\Gamma$ becomes a smooth rational curve having $(d-1)$-th order contact with the exceptional divisor. The unique point of intersection between the strict transform and the exceptional divisor corresponds to the direction of the tangent line $R$. Hence, all points $q_2$, \dots, $q_{d}$ belong to the strict transform of the exceptional divisor of $q_1$. This gives the self-intersections of $\mathcal{F}_1$, \dots, $\mathcal{F}_d$ and their configurations, as shown
in diagram \eqref{SymmetricD}. As $p_1$ is a smooth point of $\Gamma$, the curves $\mathcal{E}_1$, \dots, $\mathcal{E}_{d-1}$ form a chain of curves, as shown
in diagram \eqref{SymmetricD}. The rest of the diagram is checked by looking at the definitions of the curves $\Gamma$, $\Delta$, $L$, $R$.

We now show the existence of isomorphisms 
\[\psi_1\colon X\setminus (B\cup \tilde\Delta)\iso \A^2\mbox{ and }\psi_2\colon X\setminus (B\cup \tilde{\Gamma})\iso\A^2\]
such that $C=\psi_1(\tilde\Gamma \setminus (B\cup \tilde\Delta))$ and $D=\psi_2(\tilde\Delta\setminus (B\cup \tilde\Gamma))$ are of degree $d^2-d+1$.

We first show that $\psi_1$ exists (the case of $\psi_2$ is similar, as diagram \eqref{SymmetricD} is symmetric). We observe that since $\pi$ is the blow-up of $2d-1$ points defined over $\k$, the Picard group of $X$ is of rank $2d$, over $\k$ and over its algebraic closure $\kk$. We contract the curves $\mathcal{F}_d$, \dots, $\mathcal{F}_1$ and obtain a smooth projective surface $Y$ of Picard rank $d$ (again over $\k$ and $\kk$). The configuration of the image of the curves $\mathcal{E}_1$, \dots, $\mathcal{E}_{d-1}$, $\tilde{L}$, $\tilde{\Gamma}$ is then depicted in diagram \eqref{Situation after the contraction of F_1,...,F_d} (we omit the curve $\tilde{R}$ as we will not need it):

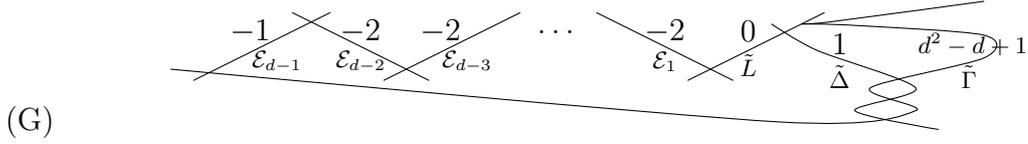
\begin{equation}  \label{Situation after the contraction of F_1,...,F_d}
\begin{tikzpicture}[scale=1]
\draw (-0.9,0.15) -- (-2.7,-0.75);
\draw (-1.9, -0.1) node{$0$};
\draw (-1.9, -0.55) node{\scriptsize$\tilde{L}$};
\draw (-3.9,0.15) -- (-2.1,-0.75);
\draw (-3, -0.1) node{$-2$};
\draw (-3, -0.5) node{\scriptsize$\mathcal{E}_1$};
\draw (-4.9,0.15) -- (-6.7,-0.75);
\draw (-5.95, -0.1) node{$-2$};
\draw (-5.6, -0.5) node{\scriptsize$\mathcal{E}_{d-3}$};
\draw (-7.4,0.15) -- (-9.2,-0.75);
\draw (-8.45, -0.1) node{$-1$};
\draw (-8.1, -0.5) node{\scriptsize$\mathcal{E}_{d-1}$};
\draw (-7.9,0.15) -- (-6.1,-0.75);
\draw (-7, -0.1) node{$-2$};
\draw (-7, -0.5) node{\scriptsize$\mathcal{E}_{d-2}$};
\draw (-4.4, -0.1) node{$\dots$};
\draw  plot [smooth] coordinates {(-9.5,-0.6) (-1.5,-1.3) (0.3,-1.15) (-0.3,-0.85) (1.2,-0.45)(1.2,-0.15) (-0.5,-0.02)(-1.2,0)};
\draw  plot [smooth] coordinates {(1.2,0.3) (-0.8,0.03)(-1.2,0)};
\draw  plot [smooth] coordinates {(0.6,-1.4) (-0.5,-1.15) (0.3,-0.85) (-1,-0.35) (-1.6,0)};
\draw (-0.7, -0.27) node{$1$};
\draw (-0.7, -0.7) node{\scriptsize$\tilde\Delta$};
\draw (1.05, -0.27) node{\scriptsize$d^2-d+1$};
\draw (1, -0.7) node{\scriptsize$\tilde\Gamma$};
\end{tikzpicture}
\end{equation}
In fact, $Y$ is just the blow-up of the points $p_1,\dots, p_{d-1}$ starting from $\p^2$.

In order to show that $X\setminus (B\cup \tilde{\Delta})\simeq Y\setminus(\tilde{\Delta}\cup \tilde{L}\cup\mathcal{E}_1\cup \dots \cup \mathcal{E}_{d-2})$ is isomorphic to $\A^2$, we will construct a birational map $\hat\psi_1\colon Y\dasharrow \p^2$ which restricts to an isomorphism $Y\setminus(\tilde{\Delta}\cup \tilde{L}\cup\mathcal{E}_1\cup \dots \cup \mathcal{E}_{d-2})\iso \p^2\setminus \mathcal{L}$ for some line $\mathcal{L}$. Let us now describe this map. Denote by $r_1$ the unique point of $Y$ such that $\{ r_1 \} =\tilde\Delta\cap \tilde{L}$ in $Y$. We blow up $r_1$ and then the point $r_2$ lying on the intersection of the exceptional curve of $r_1$ and of the strict transform of $\tilde \Delta$. For $i=3, \dots, d$, denoting by $r_i$ the point lying on the intersection of the exceptional curve of $r_{i-1}$ and on the strict transform of the exceptional curve of $r_1$, we successively blow up $r_i$. We thus obtain a birational morphism $\theta\colon Z\to Y$. The configuration of curves on $Z$ is depicted in diagram~\eqref{Situation after the blow-up of r_1,...,r_d} (we again use the same name for a curve on $Y$ and its strict transform on $Z$; we also denote by $\mathcal{G}_i \subset Z$ the strict transform of the exceptional divisor of $r_i$):
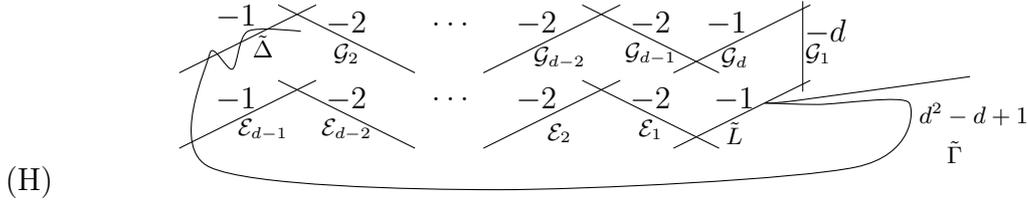
\begin{equation}  \label{Situation after the blow-up of r_1,...,r_d}
\begin{tikzpicture}[scale=1]
\draw (-0.9,1.15) -- (-2.7,0.25);
\draw (-2, 0.9) node{$-1$};
\draw (-1.9, 0.45) node{\scriptsize$\mathcal{G}_d$};
\draw (-3.4,1.15) -- (-5.2,0.25);
\draw (-4.5, 0.9) node{$-2$};
\draw (-4.2, 0.45) node{\scriptsize$\mathcal{G}_{d-2}$};
\draw (-3.9,1.15) -- (-2.1,0.25);
\draw (-3, 0.9) node{$-2$};
\draw (-3, 0.5) node{\scriptsize$\mathcal{G}_{d-1}$};
\draw (-7.4,1.15) -- (-9.2,0.25);
\draw (-8.45, 1) node{$-1$};
\draw (-8.1, 0.6) node{\scriptsize$\tilde\Delta$};
\draw (-7.9,1.15) -- (-6.1,0.25);
\draw (-7, 0.9) node{$-2$};
\draw (-7, 0.5) node{\scriptsize$\mathcal{G}_{2}$};
\draw (-5.6, 0.9) node{$\dots$};
\draw (-0.9,0.15) -- (-2.7,-0.75);
\draw (-1.9, -0.1) node{$-1$};
\draw (-1.9, -0.55) node{\scriptsize$\tilde{L}$};
\draw (-1,0) -- (-1,1.2);
\draw (-0.7, 0.8) node{$-d$};
\draw (-0.8, 0.5) node{\scriptsize$\mathcal{G}_1$};
\draw (-3.4,0.15) -- (-5.2,-0.75);
\draw (-4.5, -0.1) node{$-2$};
\draw (-4.2, -0.55) node{\scriptsize$\mathcal{E}_{2}$};
\draw (-3.9,0.15) -- (-2.1,-0.75);
\draw (-3, -0.1) node{$-2$};
\draw (-3, -0.5) node{\scriptsize$\mathcal{E}_1$};
\draw (-7.4,0.15) -- (-9.2,-0.75);
\draw (-8.45, -0.1) node{$-1$};
\draw (-8.1, -0.5) node{\scriptsize$\mathcal{E}_{d-1}$};
\draw (-7.9,0.15) -- (-6.1,-0.75);
\draw (-7, -0.1) node{$-2$};
\draw (-7, -0.5) node{\scriptsize$\mathcal{E}_{d-2}$};
\draw (-5.6, -0.1) node{$\dots$};
\draw  plot [smooth] coordinates {(-7.6,0.8)(-8.3,0.8)(-8.5,0.3)(-8.8,0.5)(-8.8,-1)(-5,-1.3)  (-0.3,-1) (0.4,-0.15) (-0.8,-0.17)(-1.5,-0.15)};
\draw  plot [smooth] coordinates {(1.2,0.2) (-1.1,-0.12)(-1.5,-0.15)};
\draw (1.25, -0.3) node{\scriptsize$d^2-d+1$};
\draw (1, -0.8) node{\scriptsize$\tilde\Gamma$};
\end{tikzpicture}
\end{equation}
We can then contract the curves $\tilde\Delta,\mathcal{G}_2,\dots,\mathcal{G}_{d-1},\tilde{L},\mathcal{E}_1,\dots,\mathcal{E}_{d-2},\mathcal{G}_1$ and obtain a birational morphism $\rho\colon Z\to \p^2$. The image of the target is $\p^2$, because
it has Picard rank~$1$; note also that the image $\mathcal{L}$ of $\mathcal{G}_d$ is actually a line of $\p^2$ since it has self-intersection $1$.
 The birational map $\hat\psi_1\colon Y\dasharrow \p^2$ given by $\hat\psi_1=\rho\theta^{-1}$ is the desired birational map. 
The closure $\overline{C}$ of~$C\subset \A^2$ in $\p^2$ is then equal to the image of~$\tilde \Gamma$ by $\rho$.

For each contracted curve above, the multiplicity (on $\overline{C}$) at the point where it is contracted,
is equal  to $d$ for $\tilde\Delta,\mathcal{G}_2,\dots,\mathcal{G}_{d-1}$, to $d-1$ for $\tilde{L},\mathcal{E}_1,\dots,\mathcal{E}_{d-2}$, and is equal to $(d-1)^2$ for $\mathcal{G}_1$. Adding the singular point of multiplicity $d-1$ of $\tilde{\Gamma}$, we obtain the two sequences of multiplicities $(\underbrace{d, \ldots, d}_{d-1})$ and $( (d-1)^2, \underbrace{d-1, \ldots, d-1}_{d})$. The self-intersection of $\overline{C}$ is then
\[ (d^2-d+1)+(d-1)\cdot d^2+(d-1)\cdot (d-1)^2+((d-1)^2)^2=(d^2-d+1)^2,\]
which implies that the curve has degree $d^2-d+1$.

The case of $\psi_2$ is similar, since the diagram \eqref{SymmetricD} is symmetric.

In particular, this construction provides an isomorphism $\A^2\setminus C\simeq \A^2\setminus D$, where $C,D\subset \A^2$ are closed curves isomorphic to $\tilde\Gamma \setminus (B\cup \tilde\Delta) \simeq \Gamma \setminus (\Delta\cup \{ q_1 \} )$ and $\tilde\Delta\setminus (B\cup \tilde\Gamma) \simeq \Delta\setminus (\Gamma\cup L)$ respectively, both of degree $d^2-d+1$. 

Since $\Gamma\setminus \{q _1 \}$ is isomorphic to $\A^1$ via 
$t\mapsto [t:1:P_d(t,1)]=[t:1:P(t)]$, we obtain that $C\simeq \Gamma \setminus (\Delta\cup \{ q_1 \})$ is isomorphic to $\Spec(\k[t,\frac{1}{P}])$.

We then take the isomorphism $\A^1\iso \Delta\setminus L=\Delta\setminus \{[\lambda:1:0]\}$ given by $t\mapsto [\lambda t+1:t:0]$. The pull-back of~$\Delta\cap \Gamma$ corresponds to the zeros of~$P_d(\lambda t+1,t)=t^dP_d(\lambda+\frac{1}{t},1)=Q(t)$. Hence, $D$ is isomorphic to  $\Spec(\k[t,\frac{1}{Q}])$ as desired.
\end{proof}

\begin{corollary} \label{Coro:Points}
For each $d\ge 0$ and every choice of distinct points $a_1,\dots,a_d,b_1, b_2 \in \p^1(\k)$, there are two closed curves $C,D\subset \A^2$ such that $\A^2\setminus C$ and $\A^2\setminus D$ are isomorphic and such that $C\simeq \p^1 \setminus \{a_1,\dots,a_d, b_1\}$ and $D \simeq \p^1 \setminus \{a_1,\dots,a_d, b_2 \}$.
\end{corollary}

\begin{proof}
The case where $d \leq 2$ is obvious: Since $\PGL_2( \k)$ acts $3$-transitively on $\p^1(\k)$,
we may take $C=D$ given by the equation $x=0$, resp.~$xy= 1$, resp.~$x(x-1) y=1$, if $d=0$, resp.~$d=1$, resp.~$d=2$. Let us now assume that $d \geq 3$.
Since $\PGL_2( \k)$ acts transitively on $\p^1(\k)$, we may assume without restriction that $b_1$ is the point at infinity $[1:0]$. Therefore, there exist distinct constants $\mu_1, \ldots,\mu_d, \lambda \in \k$ such that $a_1 = [ \mu_1 : 1], \dots, a_d = [ \mu_d : 1 ]$ and $b_2 = [ \lambda : 1]$. We now apply Proposition~\ref{proposition: existence of two curves with isomorphic complements and prescribed rings of functions} with $P=\prod_{i=1}^d (t-\mu_i)$. We get two closed curves $C,D\subset \A^2$ such that $\A^2\setminus C$ and $\A^2\setminus D$ are isomorphic and such that $C\simeq \Spec(\k[t,\frac{1}{P}])\simeq \A^1 \setminus \{\mu_1,\dots,\mu_d\} \simeq \p^1 \setminus \{ a_1, \dots, a_d, b_1 \}$ and $D\simeq \Spec(\k[t,\frac{1}{Q}]) \simeq \A^1 \setminus \{ \frac{1}{\mu_1 - \lambda} , \dots, \frac{1}{\mu_d- \lambda} \}
$, where $Q(t)=P(\lambda +\frac{1}{t})\cdot t^d$. It remains to observe that  $D$ is isomorphic to $\p^1\setminus \{[\mu_1:1],\dots,[\mu_d:1],[\lambda :1]\}$ via $t\mapsto [\lambda t+1:t]$.
\end{proof}

\begin{corollary} \label{Cor:AllcurvesInfinitefield}
If $\k$ is infinite and $P\in \k[t]$ is a polynomial with at least $3$ roots in $\overline{\k}$, we can find two curves $C,D\subset \A^2$ that have isomorphic complements, such that $C$ is isomorphic to
$\Spec(\k[t,\frac{1}{P}])$, but $D$ is not. 
\end{corollary}

\begin{proof}
By Lemma~\ref{lemma: P and Q are not isomorphic for a general lambda} below, there exists a constant $\lambda$ in $\k$ such that $P( \lambda) \neq 0$ and such that the curves $\Spec(\k[t,\frac{1}{P}])$ and $\Spec(\k[t,\frac{1}{Q}])$ are not isomorphic.  The result now follows from Proposition~\ref{proposition: existence of two curves with isomorphic complements and prescribed rings of functions}.
\end{proof}

\begin{lemma}   \label{lemma: P and Q are not isomorphic for a general lambda}
If $\k$ is infinite and $P\in \k[t]$ is a polynomial with at least $3$ roots in $\overline{\k}$, then for a general $\lambda \in \k$, the polynomial $Q(t)=P(\lambda +\frac{1}{t})\cdot t^{\deg(P)}$ has the property that the curves $\Spec(\k[t,\frac{1}{P}])$ and $\Spec(\k[t,\frac{1}{Q}])$ are not isomorphic.
\end{lemma}

\begin{proof}
Let $\lambda_1, \ldots, \lambda_d\in \kk$ be the single roots of $P$. It suffices to check that for a general $\lambda$ there is no automorphism of $\p^1$ that sends $\{ \lambda_1, \ldots, \lambda_d, \infty \}$ to $\{ \frac{1}{\lambda_1- \lambda}, \ldots, \frac{1}{\lambda_d- \lambda}, \infty \}$, or equivalently that there is no automorphism that sends $\{ \lambda_1, \ldots, \lambda_d, \infty \}$ to  $\{ \lambda_1, \ldots, \lambda_d, \lambda \}$. But if an automorphism sends $\{ \lambda_1, \ldots, \lambda_d, \infty \}$ to  $\{ \lambda_1, \ldots, \lambda_d, \lambda \}$, it necessarily belongs to the set  ${\mathcal A}$ of automorphisms $\varphi$ such that $\varphi^{-1} ( \{ \lambda_1, \lambda_2, \lambda_3 \} ) \subset \{ \lambda_1, \ldots, \lambda_d, \infty \}$. Since an automorphism of $\p^1$ is determined by the image of 3 points, the set ${\mathcal A}$ has at most $6 {d+1 \choose 3} = (d+1)d(d-1)$ elements. In conclusion, if $\lambda$ is not of the form $\varphi (\mu) $ for some $\varphi \in {\mathcal A}$ and some $\mu \in \{ \lambda_1, \ldots, \lambda_d, \infty \}$, then no automorphism of $\p^1$ sends  $\{ \lambda_1, \ldots, \lambda_d, \infty \}$ to  $\{ \lambda_1, \ldots, \lambda_d, \lambda \}$.
\end{proof}

\begin{remark}
If $\k$ is a finite field (with at least $3$ elements), then the conclusion of Corollary~\ref{Cor:AllcurvesInfinitefield} is false for the polynomial  $P=\prod\limits_{\alpha\in \k}(x-\alpha)$. Indeed, if $C,D\subset \A^2$ are two curves such that $C$ is isomorphic to $\Spec(\k[t,\frac{1}{P}])$ and $\A^2\setminus C$ is isomorphic to $\A^2\setminus D$, then $D$ is isomorphic to $\Spec(\k[t,\frac{1}{Q}])$ for some polynomial $Q$ that has no square factors and the same number of roots in $\k$ and in $\kk$ as $P$ (Theorem~\ref{RigidityThm}$(\ref{RigidityThmPQ})$). This implies that $Q$ is equal to $\mu P$ for some $\mu\in\k^*$ and thus that $C$ and $D$ are isomorphic. 

A similar argument holds for $P=\prod\limits_{\alpha\in \k^*}(x-\alpha)$ and $P=\prod\limits_{\alpha\in \k\setminus \{0,1\}}(x-\alpha)$ (when the field has at least $4$, respectively $5$ elements) since $\PGL_2(\k)$ acts $3$-transitively on $\p^1 (\k)$.
\end{remark}

\begin{corollary}  \label{Coro:Field}
For each ground field $\k$ with more than $27$ elements, there exist two geometrically irreducible closed curves $C,D\subset \A^2$ of degree $7$ which are not
isomorphic, but such that $\A^2\setminus C$ and $\A^2\setminus D$ are isomorphic.
\end{corollary}

\begin{proof}
We fix some element $\zeta \in \k \setminus \{0,1\}$. For each $\lambda\in \k\setminus \{0,1,\zeta\}$,
we apply Corollary~\ref{Coro:Points} with $d=3$, $a_1= [ 0:1 ]$, $a_2 = [ 1 : 1]$, $a_3= [ \zeta : 1 ]$, $b_1 = [1:0 ]$, $b_2 = [ \lambda : 1 ]$ and get two closed curves $C,D\subset \A^2$ such that $\A^2\setminus C$ and $\A^2\setminus D$ are isomorphic and such that $C\simeq \A^1 \setminus \{0,1,\zeta\}=\p^1\setminus \{[0:1],[1:1],[\zeta:1],[1:0]\}$ and $D\simeq \p^1\setminus \{[0:1],[1:1],[\zeta:1],[\lambda:1]\}$. It remains to see that we
can find at least one $\lambda$ such that $C$ and $D$ are not isomorphic. Note that $C$ and $D$ are isomorphic if and only if there is an element of~$\Aut(\p^1)=\PGL_2(\k)$ that sends $\{[0:1],[1:1],[\zeta:1],[\lambda:1]\}$ onto $\{[0:1],[1:1],[\zeta:1],[1:0]\}$. The image of this element is determined by the image of~$[0:1],[1:1],[\zeta:1]$, so
we have at most $24$ automorphisms to avoid, hence at most $24$ elements of~$\k\setminus \{0,1,\zeta\}$ to avoid. Since the field $\k$ has at least $28$ elements, we find at least one $\lambda$ with the desired property.
\end{proof}

We can now prove Theorem~$\ref{Theorem:ComplementProblemAllFields}$.
\begin{proof}[Proof of Theorem~$\ref{Theorem:ComplementProblemAllFields}$]
If the field is infinite (or simply has more than $27$ elements), the theorem follows from Corollary~\ref{Coro:Field}. Let us therefore assume that $\k$ is a finite field. We again apply Proposition~\ref{proposition: existence of two curves with isomorphic complements and prescribed rings of functions} (with $\lambda = 0$). Therefore, if $\lvert \k \rvert > 2$ (resp.~$\lvert \k \rvert = 2$), it suffices to give a polynomial $P \in \k [ t]$ of degree $3$ (resp.~$4$) such that $P(0) \neq 0$ and such that if we set $Q:= P(\frac{1}{t} ) t^{\deg P}$, then the $\k$-algebras $\k [ t, \frac{1}{P} ]$ and $\k [ t, \frac{1}{Q} ]$ are not isomorphic.

We begin with the case where the characteristic of $\k$ is odd. Then, the kernel of the morphism of groups $\k^* \to \k^*$, $x \mapsto x^2$ is equal to $\{ -1, 1 \}$, so that this map is not surjective. Let us pick an element $\alpha \in \k^* \setminus (\k^*)^2$. Let us check that we can take $P= (t-1) ((t-1)^2 - \alpha)$. Indeed, up to a multiplicative constant, we have $Q= (t-1) ((t-1)^2 - \alpha t^2)$. Let us assume by contradiction that the algebras $\k [ t, \frac{1}{P} ]$ and $\k [ t, \frac{1}{Q} ]$ are isomorphic. Then, these algebras would still be isomorphic if we replaced $P$ and $Q$ by
\[ \tilde{P} = P(t+1) = t \big( t^2 - \alpha \big) \; \text{and} \; \tilde{Q} = Q(t+1) = t \big( t^2 - \alpha(t+1)^2\big).\] 
This would produce an automorphism of $\p^1$, via the embedding $t\mapsto [t:1]$, which sends the polynomial $uv(u^2-\alpha v^2)$ onto a multiple of $uv(u^2-\alpha (u+v)^2)$. This automorphism preserves the set of $\k$-roots: $\{ [0:1], [1:0] \}$, and is of the form either $[u:v] \mapsto [\mu u:v]$ or $[u:v] \mapsto [\mu v:u]$ where $\mu \in \k^*$.  The polynomial $u^2-\alpha v^2$ must be sent to a multiple of $u^2-\alpha (u+v)^2$, which is not
possible, because of the term $uv$.

We now treat the case where $\k$ has characteristic $2$. We divide it into three cases, depending on whether the cube homomorphism of groups $\k^* \to \k^*$, $x \mapsto x^3$ is surjective or not (which corresponds to asking that $\lvert \k\rvert$ not be a power of $4$) and setting aside the field with two elements.

If the cube homomorphism is not surjective, we can pick an element $\alpha \in \k^* \setminus (\k^*)^3$. We may take the irreducible polynomial $P= t^3 - \alpha\in \k[t]$. Indeed, up to a multiplicative constant, we have $Q= t^3 - \alpha^{-1}$. Assume by contradiction that the algebras $\k [ t, \frac{1}{P} ]$ and $\k [ t, \frac{1}{Q} ]$ are isomorphic. Then, there should exist constants $\lambda, \mu, c \in \k$ with $\lambda c \neq 0$ such that
\[ c (t^3 - \alpha^{-1}) = (\lambda t + \mu)^3 - \alpha.\]
This gives us $\mu = 0$ and $\lambda^3 = c = \alpha^2$.  Since the square homomorphism of groups $\k^* \to \k^*$, $x \mapsto x^2$ is bijective, there is a unique square root for each element of $ \k^*$. Taking the square root of the equality $\alpha^2 = \lambda^3$, we obtain $\alpha = ( \nu)^3$, where $\nu$ is the square root of $\lambda$. This is impossible since $\alpha$ was chosen not to be a cube.

If the cube homomorphism is surjective, then $1$ is the only root of $t^3-1=(t-1)(t^2+t+1)$, so $t^2+t+1\in \k[t]$ is irrreducible. If moreover $\k$ has more than $2$ elements, we can choose $\alpha\in \k\setminus \{0,1\}$ and take $P=(t - \alpha) (t^2+t+1)$. Up to a multiplicative constant, we have $Q= (t - \alpha^{-1}) (t^2+t+1)$. Let us assume by contradiction that the algebras $\k [ t, \frac{1}{P} ]$ and $\k [ t, \frac{1}{Q} ]$ are isomorphic.  Then, these algebras would still be isomorphic if we replaced $P$ and $Q$ by
\[ \tilde{P} = P(t+\alpha) = t ( t^2+t+  \alpha^2 + \alpha +1 ) \; \text{and} \; \tilde{Q} = Q(t+ \alpha^{-1}) = t ( t^2 +t+  \alpha^{-2} + \alpha^{-1} +1 ).\] 
This would yield an automorphism of $\p^1$, via the embedding $t\mapsto [t:1]$, which sends the polynomial $uv(u^2+uv + (\alpha^2 + \alpha +1) v^2)$ onto a multiple of $uv(u^2+uv + (\alpha^{-2} + \alpha^{-1} +1) v^2)$. The same argument as before gives $\alpha^2 + \alpha +1 = \alpha^{-2} + \alpha^{-1} +1$, i.e.~$\alpha^2 + \alpha +1  = \alpha^{-2}(\alpha^2 + \alpha +1)$. This is impossible since $\alpha^2 + \alpha +1 \neq 0$ and $\alpha^2 \neq 1$.

The last case is that in which $\k=\{0,1\}$ is the field with two elements. Here the construction does not work with polynomials of degree $3$: the only ones which are not symmetric and do not vanish at $0$ are $t^3+t^2+1$ and $t^3+t+1$, and they are equivalent via $t\mapsto t+1$. We then choose for $P$ the irreducible polynomial $P=t^4+t+1$ (it has no root and is not equal to $(t^2+t+1)^2=t^4+t^2+1$). This gives $Q=t^4+t^3+1$. Let us assume by contradiction that the algebras $\k [ t, \frac{1}{P} ]$ and $\k [ t, \frac{1}{Q} ]$ are isomorphic. Then, there would exist constants $\lambda, \mu, c \in \k$  such that  $\lambda c \neq 0$ and
\[c(t^4+t^3+1)= (\lambda t+\mu)^4+(\lambda t+\mu)+1.\]
This is impossible since $(\lambda t+\mu)^4+(\lambda t+\mu)+1=\lambda^4t^4+\lambda t+(\mu^4+\mu+1)$.
\end{proof}

\subsection{Finding explicit formulas}
To obtain the equations of the curves $C,D$ and the isomorphism $\A^2\setminus C\iso \A^2\setminus D$ given by Proposition~\ref{proposition: existence of two curves with isomorphic complements and prescribed rings of functions}, we could follow the construction and explicitly compute the birational maps described: The proposition establishes the existence of isomorphisms 
\[\psi_1\colon X\setminus (B\cup \tilde\Delta)\iso \A^2\mbox{ and }\psi_2\colon X\setminus (B\cup \tilde{\Gamma})\iso\A^2\]
such that $C=\psi_1(\tilde\Gamma \setminus (B\cup \tilde\Delta))$ and $D=\psi_2(\tilde\Delta\setminus (B\cup \tilde\Gamma))$ are of degree $d^2-d+1$, where $B=\tilde{L}\cup \bigcup\limits_{i=1}^{d-2}\mathcal{E}_i\cup \bigcup\limits_{i=1}^{d}\mathcal{F}_i$, and $\psi_1,\psi_2$ are given by blow-ups and blow-downs, so it is possible to compute $\psi_i \pi^{-1}\colon \p^2\dasharrow \p^2$ with formulas (looking at the linear systems), and then to get the isomorphism $\psi_2\pi^{-1}\circ (\psi_1 \pi^{-1})^{-1}\colon \A^2\setminus C\iso \A^2\setminus D$. However, the formulas for $\psi_1 \pi^{-1},\psi_2 \pi^{-1}$ are complicated.

Another possibility is the following: we choose a birational morphism $X\to W$ that contracts $\tilde{L},\mathcal{E}_1,\dots,\mathcal{E}_{d-2}$ and $\mathcal{F}_d,\dots,\mathcal{F}_2$ to two smooth points of $W$, passing through the image of $\mathcal{F}_1$ (this is possible, see diagram~(\ref{SymmetricD})). The situation of the image of the curves $\tilde{R},\mathcal{E}_{d-1},\mathcal{F}_1,\tilde\Gamma,\tilde\Delta$ (which we again denote by the same name) in $W$ is as follows:

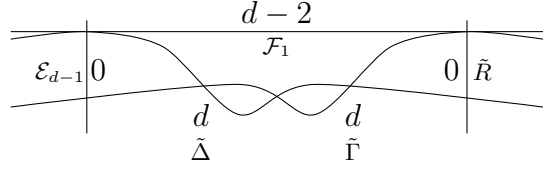
\[
\begin{tikzpicture}[scale=1]
\draw (-3.5,0) -- (3.5,0);
\draw (0, 0.25) node{$d-2$};
\draw (-2.5,0.15) -- (-2.5,-1.35);
\draw (-2.35, -0.5) node{$0$};
\draw (2.5,0.15) -- (2.5,-1.35);
\draw (2.3, -0.5) node{$0$};
\draw  plot [smooth] coordinates {(-3.5,-1)  (-0.5,-0.7) (0.5,-1.1)(1.5,-0.2)  (2.5,0)(3.5,-0.1)};
\draw  plot [smooth] coordinates {(3.5,-1) (0.5,-0.7) (-0.5,-1.1)(-1.5,-0.2)  (-2.5,0)(-3.5,-0.1)};
\draw (-1, -1.55) node{\scriptsize$\tilde\Delta$};
\draw (1, -1.55) node{\scriptsize$\tilde\Gamma$};
\draw (-1, -1.1) node{$d$};
\draw (1, -1.1) node{$d$};
\draw (-2.85, -0.55) node{\scriptsize$\mathcal{E}_{d-1}$};
\draw (0, -0.2) node{\scriptsize$\mathcal{F}_1$};
\draw (2.7, -0.5) node{\scriptsize$\tilde{R}$};
\end{tikzpicture}\]
Computing the dimension of the Picard group, we find that $W$ is a Hirzebruch surface. Hence, the curves $\mathcal{E}_{d-1},\tilde{R}$ are fibres of a $\p^1$-bundle $W\to \p^1$ and  $\mathcal{F}_1,\tilde\Delta,\tilde\Gamma$ are sections of self-intersection $d-2,d,d$. We can then find many examples in $\mathbb{F}_1$ and $\mathbb{F}_0$
(depending on the parity of $d$), but also in $\mathbb{F}_m$ for $m\ge 2$ if the polynomial chosen at the outset is special enough.

The case where $d=3$ corresponds to curves of degree $7$ in $\A^2$ (Proposition~\ref{proposition: existence of two curves with isomorphic complements and prescribed rings of functions}), which is the first interesting case, as it gives non-isomorphic curves for almost every field (Theorem~\ref{Theorem:ComplementProblemAllFields}). When $d=3$, we find that $\mathcal{F}_1$ is a section of self-intersection $1$ in $W=\mathbb{F}_1$, so $\mathbb{F}_1\setminus \mathcal{F}_1$ is isomorphic to the blow-up of $\A^2$ at one point, and $\tilde\Gamma,\tilde\Delta$ are sections of self-intersection~$3$ and are thus strict transforms of parabolas passing through the point blown up. This explains how the following result is derived from Proposition~\ref{proposition: existence of two curves with isomorphic complements and prescribed rings of functions}. 
However, the statement and the proof that we give are independent of the
latter proposition:

\begin{proposition}\label{Prop:ExplicitDegree7}
Let us fix some constants $a_0,a_1,a_2,a_3 \in \k$ with $a_0a_3 \neq 0$ and consider the two irreducible polynomials $P,Q\in \k[x,y]$ of degree $2$ given by
\[P=x^2-a_2x-a_3y \quad \text{and} \quad Q=y^2+a_0x+a_1y. \]

\begin{enumerate}[$(1)$]
\item\label{IsoPhiPQ}
Denoting by $\eta\colon \hat\A^2\to \A^2$  the blow-up of the origin and by  $\tilde\Gamma,\tilde\Delta \subset \hat\A^2$  the strict transforms of the curves $\Gamma,\Delta\subset \A^2$ given by $P=0$ and $Q=0$ respectively, the rational maps
\[\begin{array}[t]{rccc} \varphi_P\colon &\A^2 & \dasharrow & \A^2\\
&(x,y)&\mapsto &\left( - {\dis \frac{x}{P(x,y)} },P(x,y)\right)\end{array}
 \text{and } \;
\begin{array}[t]{rccc} \varphi_Q\colon &\A^2 & \dasharrow & \A^2\\
&(x,y)&\mapsto &\left(  {\dis \frac{y}{Q(x,y)} } ,Q(x,y) \right)\end{array}\]
are birational maps that induce isomorphisms
\[\begin{array}{rccc}
\psi_P=(\varphi_P\eta)|_{\hat{\A}^2\setminus \tilde{\Gamma}} \colon & \hat{\A}^2\setminus \tilde{\Gamma} & \iso & \A^2\end{array}
\quad \text{and} \quad
\begin{array}{rccc} \psi_Q=(\varphi_Q\eta)|_{\hat{\A}^2\setminus \tilde{\Delta}}\colon & \hat{\A}^2\setminus \tilde{\Delta} & \iso & \A^2.\end{array}\]

\item\label{IsoA2CA2Dexplicit}
Define the curves $C,D \subset \A^2$ by $C=\psi_Q(\tilde{\Gamma}\setminus \tilde{\Delta})$, $D=\psi_P(\tilde{\Delta}\setminus \tilde{\Gamma})$ and denote by $\psi\colon \A^2\setminus C\iso \A^2\setminus D$ the isomorphism induced by the birational transformation $\psi_P (\psi_Q)^{-1} \colon \A^2 \dasharrow \A^2$. Then, the curves $C,D \subset \A^2$ are given by $f=0$ and $g=0$ respectively, where the polynomials $f,g\in \k[x,y]$ are defined by:
\[\begin{array}{rcl}
f&=&  \left(\vphantom{\Big)} 1-x(xy+a_1) \right) \left(y \left(\vphantom{\Big)} 1-x(xy+a_1) \right)-a_0a_2\right)-x(a_0)^2a_3,\vspace{0.2cm}\\
g&=& \left(\vphantom{\Big)} 1-x(xy+a_2) \right) \left(y \left(\vphantom{\Big)} 1-x(xy+a_2) \right)-a_1a_3\right)-xa_0(a_3)^2.\end{array}\]
The following isomorphisms hold:
\[ C \simeq \Spec\left(\k[t,\frac{1}{\sum_{i=0}^3 a_i t^i}]\right) \quad \text{and} \quad D \simeq \Spec\left(\k[t,\frac{1}{\sum_{i=0}^3 a_{3-i} t^i}]\right).\] 
Moreover, $\psi$ and $\psi^{-1}$ are given by  
\[\begin{array}{rccc}
\psi\colon & (x,y)&\mapsto & {\dis \left(\frac{a_0 \left(\vphantom{\Big)} x(xy+a_1)-1 \right) }{f(x,y)},\frac{y \, f(x,y)}{(a_0)^2}\right) } \\
& {\dis \left(\frac{a_3 \left(\vphantom{\Big)} x(xy+a_2) -1 \right) }{g(x,y)},\frac{y \, g(x,y)}{(a_3)^2}\right) } & \mapsfrom & (x,y).
\end{array}
\]
\end{enumerate}
\end{proposition}
\begin{proof}
$(\ref{IsoPhiPQ})$: Let us first prove that $\varphi_P$ is birational and that $\varphi_P\eta$ induces an isomorphism $\hat{\A}^2\setminus \tilde{\Gamma}\iso \A^2$. We observe that $\kappa\colon (x,y)\mapsto (x,x^2-a_2x-a_3y)$ is an automorphism of $\A^2$ that sends $\Gamma$ onto the line $L_y \subset \A^2$ of equation $y=0$. Moreover $\tilde\varphi_P=   \varphi_P \kappa^{-1} 
\colon (x,y)\mapsto (-\frac{x}{y},y)$  is birational, so $\varphi_P$ is birational. Since $\kappa$ fixes the origin, $\eta^{-1}\kappa\eta$ is an automorphism of $\hat{\A}^2$ that sends $\tilde{\Gamma}$ onto the strict transform $\tilde{L}_y \subset \hat{\A}^2$ of $L_y$.
The fact that $\tilde{\varphi}_P\eta$ induces an isomorphism $\hat{\A}^2\setminus \tilde{L}_y \iso  \A^2$ is straightforward using the classical description of the blow-up $\hat{\A}^2$ in which
\[ \hat{\A}^2  = \{ ((x,y), [u:v]) \;|\; xv=yu \} \subset \A^2 \times \p^1 \]
and $\eta \colon \hat{\A}^2 \to \A^2$ is the first projection. Actually, with this description $\tilde{L}_y = L_y \times [1:0]$ is given by the equation $v=0$ and the following morphisms are inverses of each other:
\[  \begin{array}{ll} \hat{\A}^2\setminus \tilde{L}_y \to  \A^2,  & ( (x,y), [u:v])  \mapsto  (-\frac{u}{v}, y) \\
\A^2 \to \hat{\A}^2 \setminus \tilde{L}_y,   &  (x,y) \mapsto ( (-xy,y), [-x:1]). \end{array} \]
It follows that $(\tilde{\varphi}_P \eta) (\eta^{-1}\kappa \eta) = \varphi_P \eta$ induces an isomorphism $\hat{\A}^2\setminus \tilde{\Gamma}\iso \A^2$. The case of $\varphi_Q$ and $\varphi_Q\eta$ would be treated similarly, using the automorphism of $\A^2$ given by $(x,y)\mapsto (y^2+a_0x+a_1y,y)$. This proves~$(\ref{IsoPhiPQ})$.

$(\ref{IsoA2CA2Dexplicit})$: Now that $(\ref{IsoPhiPQ})$ is proven, we get two isomorphisms
\[\psi_P|_U\colon U\iso \A^2 \setminus D,  \quad \psi_Q|_U\colon U\iso \A^2 \setminus C,\]
where $U=\hat{\A}^2 \setminus (\tilde{\Gamma}\cup \tilde{\Delta})$.
Remembering that $\Gamma\subset \A^2$ is given by $x(x-a_2)=a_3y$, we have an isomorphism
\[\begin{array}{rccc}
\rho\colon & \A^1 & \iso &\Gamma\\
&t & \mapsto &\left(t a_3+a_2, t(t a_3+a_2)\right)\\
&\frac{1}{a_3}(x-a_2) & \mapsfrom& (x,y). \end{array}\]
Replacing $\rho(t)$ in the polynomial $Q(x,y)=x a_0+ya_1+y^2$ used to define $\Delta$, we find 
\[Q(t a_3+a_2,t(t a_3+a_2))=(ta_3+a_2)(t^3a_3+t^2a_2+ta_1+a_0).\]
The root of $ta_3+a_2$ is sent by $\rho$ to the origin, which is itself blown up by $\eta$. Hence, the map $\eta^{-1} \rho$ induces an isomorphism from $V=\Spec(\k[t,\frac{1}{\sum_{i=0}^3 t^ia_i}])\subset \A^1$ to $\tilde\Gamma\setminus \tilde\Delta$.
Applying $\psi_Q=(\varphi_Q\eta)|_{\hat{\A}^2\setminus \tilde{\Delta}}$, we get an isomorphism $\theta=(\varphi_Q\rho)|_V\colon V \iso C$. Since $(\varphi_Q)^{-1}$ is given by
\[(\varphi_Q)^{-1}\colon (x,y) \mapsto \left( \frac{y \left(\vphantom{\Big)} 1-x(xy+a_1) \right)  }{a_0},xy \right) , \]
we can explicitly give $\theta$ and its inverse:
\[\begin{array}{rccc}
\theta\colon & \Spec(\k[t,\frac{1}{\sum_{i=0}^3 t^ia_i}]) & \iso &C\\
&t & \mapsto &\left( {\dis \frac{t}{ \sum_{i=0}^3 t^ia_i} },(ta_3+a_2) ({ \sum_{i=0}^3 t^ia_i})  \right)\\
&{\dis \frac{1}{a_3} \left( \frac{ y \left(\vphantom{\Big)} 1-x(xy-a_1) \right) }  {a_0}-a_2 \right)  }  & \mapsfrom& (x,y). \end{array}\]
Computing the extension of $\theta$ to a morphism $\p^1\to \p^2$, we see that the curve $C\subset \A^2$ has degree $7$. To find its equation, we can compute $((\varphi_Q)^{-1})^*(P)$: since $(a_0)^2P(x,y)=(a_0x)(a_0x-a_0a_2)-(a_0)^2a_3y$, we get 
\[\begin{array}{rcl}
(a_0)^2((\varphi_Q)^{-1})^*(P)&=&(a_0)^2 P \left(\frac{ y \left( 1-x(xy+a_1) \right) }{a_0}, xy\right)\\
&=&y \left( 1-x(xy+a_1) \right) (y \left( 1-x(xy+a_1) \right) -a_0a_2)-xy(a_0)^2a_3\\
&=&yf(x,y),\end{array}\]
where  
\[f= \left( 1-x(xy+a_1) \right) (y \left( 1-x(xy+a_1) \right) -a_0a_2)-x(a_0)^2a_3\in \k[x,y]\]
is the equation of $C$ (note that the polynomial $y=0$ appears here, because
it corresponds to the line contracted by $(\psi_Q)^{-1}$, corresponding to the exceptional divisor of $\hat{\A}^2\to \A^2$ via the isomorphism $\A^2\to \hat{\A}^2\setminus \hat\Delta$). The linear involution of $\A^2$ given by $(x,y)\mapsto (-y,-x)$ exchanges the polynomials $P$ and $Q$ and the maps $\varphi_P$ and $\varphi_Q$, by replacing $a_0,a_1,a_2,a_3$ by $a_3,a_2,a_1,a_0$ respectively. This shows that $D\subset \A^2$ has equation $g=0$, where $g$ is obtained from $f$ on replacing $a_0,a_1,a_2,a_3$ by $a_3,a_2,a_1,a_0$, i.e.
\[g= \left( 1-x(xy+a_2) \right) (y \left( 1-x(xy+a_2) \right) -a_1a_3)-xa_0(a_3)^2 \in \k[x,y]. \]
Therefore, $D$ is isomorphic to $\Spec(\k[t,\frac{1}{\sum_{i=0}^3 \alpha_{3-i} t^i}])$. It remains to compute the isomorphism $\psi\colon \A^2\setminus C\to \A^2\setminus D$, which is by construction equal to the birational maps $\psi_P(\psi_Q)^{-1}=\varphi_P(\varphi_Q)^{-1}$. Using the equation $(a_0)^2P \left( \frac{y \left( 1-x(xy+a_1) \right)  }{a_0},xy \right) =yf(x,y)$, we get:
\[\begin{array}{rll}
 \psi (x,y) &=& {\dis \varphi_P \left( \frac{ y \left( 1-x(xy+a_1) \right) }{a_0},xy \right) } \vspace{0.2cm} \\
& =& {\dis \left(-\frac{ y \left( 1-x(xy+a_1) \right) }{a_0 P \left( \frac{y \left( 1-x(xy+a_1) \right) }{a_0},xy \right) },P \left( \frac{ y \left( 1-x(xy+a_1) \right) } {a_0},xy \right)\right) } \vspace{0.2cm}\\
& = & {\dis  \left(\frac{a_0 \left( x(xy+a_1)-1\right) }{f(x,y)}, \frac{y \, f(x,y)}{(a_0)^2}\right). } \end{array}\]
By symmetry, the expression of $\psi^{-1}$ is obtained from that of $\psi$ by  replacing $a_0,a_1,a_2,a_3$ by $a_3,a_2,a_1,a_0$, i.e.~it is given by ${\dis \psi^{-1}(x,y) = \left(\frac{a_3 \left( x(xy+a_2)-1\right) }{g(x,y)},\frac{y \, g(x,y)}{(a_3)^2}\right) }$.\end{proof}

\begin{remark}
Proposition~\ref{Prop:ExplicitDegree7} yields an isomorphism $\psi^{*}\colon \k[x,y,\frac{1}{g}]\iso \k[x,y,\frac{1}{f}]$ which sends the invertible elements onto the invertible elements and thus sends $g$ onto $\lambda f^{\pm 1}$ for some $\lambda\in \k^*$  (see Lemma~\ref{Lemm:RingInvertible}). This corresponds to saying that $\psi$ induces an isomorphism between the two fibrations
\[\A^2\setminus C\stackrel{f}{\to} \A^1\setminus \{0\} \quad \text{and} \quad \A^2\setminus D\stackrel{g}{\to} \A^1\setminus \{0\},\]
possibly exchanging the fibres. To study these fibrations, we use the equalities
\begin{equation}\label{phiQP}
(\varphi_Q)^*(f)=\frac{(a_0)^2P}{Q},\quad (\varphi_P)^{*}(g)=\frac{(a_3)^2Q}{P},\end{equation}
which can either be checked directly, or deduced as follows: the first equality follows from $((\varphi_Q)^{-1})^*(P)=\frac{yf(x,y)}{(a_0)^2}$, applying $(\varphi_Q)^*$, and the second is obtained by symmetry.

Note that equation~(\ref{phiQP}) provides $\psi^{*}(g)=\frac{(a_0a_3)^2}{f}$, since  $\psi=\varphi_P(\varphi_Q)^{-1}$.

For each $\mu\in \k$, the fibre $C_\mu\subset \A^2$ given by $f(x,y)=\mu$ is an algebraic curve isomorphic to its preimage by the isomorphism $\psi_Q=(\varphi_Q\eta)|_{\hat{\A}^2\setminus \tilde{\Delta}}\colon  \hat{\A}^2\setminus \tilde{\Delta}  \iso  \A^2$ of Proposition~\ref{Prop:ExplicitDegree7}$(\ref{IsoPhiPQ})$. By construction, $(\psi_Q)^{-1}(C_\mu)$ is equal to $\tilde{\Gamma}_\mu\setminus \tilde{\Delta}$, where $\tilde{\Gamma}_\mu\subset \hat{\A}^2$ is the strict transform of the curve $\Gamma_\mu\subset \A^2$ given by $(a_0)^2P-\mu Q=0$ (follows from equation~(\ref{phiQP})). The closure of $\Gamma_\mu$ in $\p^2$ is the conic given by
\[(a_0)^2x^2-\mu y^2-z \left(\vphantom{\Big)}a_0 ( \mu+a_0a_2 ) x- ( \mu a_1+(a_0)^2a_3 ) y \right)=0, \]
which passes through $[0:0:1]$ and is irreducible for a general $\mu$. Projecting from the point $[0:0:1]$ we obtain an isomorphism with $\p^1$ (still for a general $\mu$). The curve $\tilde{\Gamma}_\mu\setminus \tilde{\Delta}$ is then isomorphic to $\p^1$ minus three $\kk$-points of $\tilde{\Delta}$, which are fixed and do not depend on $\mu$, and minus the two points at infinity, which correspond to $(a_0)^2x^2-\mu y^2=0$.

When the field is algebraically closed, we thus find that the general fibres of $f$ are isomorphic to $\p^1$ minus $5$  points, whereas the zero fibre is isomorphic to $\p^1$ minus $4$  points (if $\sum_{i=0}^3 a_i t^i$ is chosen to have three distinct roots). Moreover, the two points of intersection with the line at infinity say that this curve is a \emph{horizontal curve of degree~$2$}, or a \emph{horizontal curve which is not a section} (in the usual notation of polynomials and components on boundary, see \cite{Neuman,AC96,CD17}), so the polynomials $f$ and $g$ are
rational, but
not of simple type (see \cite{Neuman,CD17}). When $\k=\C$, this implies that the polynomial has non-trivial monodromy  \cite[Corollary 2, page 320]{ACD98}.
\end{remark}
\section{Related questions}\label{Related}
\subsection{Higher dimensional counterexamples}\label{SubSec:HigherDim}
The negative answer to the Complement Problem for $n=2$
also furnishes a negative answer for any $n\ge 3$.
This relies mainly on the cancellation property for curves, as explained in the following result:

\begin{proposition}   \label{Prop:Products}
Let $C,D\subset \A^2$ be two closed geometrically irreducible curves that have isomorphic complements. Then for each $m\ge 1$, the varieties $H_C=C\times \A^m$ and $H_D=D\times \A^m$ are closed hypersurfaces of $\A^2\times \A^m=\A^{m+2}$ that have isomorphic complements. Moreover, $C$ and $D$ are isomorphic if and only if  $C\times \A^m$ and $D\times \A^m$ are.
\end{proposition}
\begin{proof}
Denoting by $f,g\in \k[x,y]$ the geometrically irreducible polynomials that define the curves $C,D$, the varieties $H_C,H_D\subset \A^2\times \A^m=\A^{m+2}$ are given by the same polynomials and are thus again geometrically irreducible closed
hypersurfaces. The isomorphism $\A^2\setminus C\iso \A^2\setminus D$  then  extends naturally to an isomorphism $\A^{m+2}\setminus H_C\iso \A^{m+2}\setminus H_D$.

The last equivalence is the well-known cancellation property for curves, proven in \cite[Corollary (3.4)]{AHE72}.
\end{proof}
\begin{corollary}\label{Coro:HighDimExp}
For each ground field $\k$ and each integer $n\ge 3$, there exist two geometrically irreducible smooth closed hypersurfaces $E,F \subset \A^n$ which are not
isomorphic, but whose complements $\A^n\setminus E$ and $\A^n\setminus F$ are isomorphic. Furthermore, the hypersurfaces can be given by polynomials $f,g\in \k[x_1,x_2]\subset \k[x_1,\dots,x_n]$ of degree $7$ if the field admits more than $2$ elements and of degree $13$ if the field has $2$ elements. The hypersurfaces $E,F$ are
isomorphic to $C\times \A^{n-2}$ and $D\times \A^{n-2}$ for some smooth closed curves $C,D\subset \A^2$ of the same degree.
\end{corollary}
\begin{proof}
It suffices to choose for $f,g$ the equations of the curves $C,D\subset \A^2$ given by Theorem~\ref{Theorem:ComplementProblemAllFields}. The result then follows from Proposition~\ref{Prop:Products}.
\end{proof}

\subsection{The holomorphic case}

\begin{proposition} \label{Prop:Curvesholomorphic}
For every choice of $d+1$ distinct points $a_1,\dots,a_d,a_{d+1}\in \C$, with $d\ge 3$, there exist two closed algebraic curves $C,D\subset \C^2$ of degree $d^2-d+1$ such that $C$ and $D$ are algebraically isomorphic to $\C\setminus \{a_1,\dots,a_{d-1},a_d\}$ and $\C\setminus \{a_1,\dots,a_{d-1},a_{d+1}\}$ respectively, and such that $\C^2\setminus C$ and $ \C^2\setminus D$ are algebraically isomorphic.

In particular, if we choose the points in general position,
the curves $C$ and $D$ are not biholomorphic, but their complements are. 
\end{proposition}

\begin{proof}
The existence of $C,D$ follows directly from Proposition~\ref{proposition: existence of two curves with isomorphic complements and prescribed rings of functions}. It remains to observe that $C$ and $D$ are not biholomorphic if the points are in general position.
If $f\colon C\to D$ is a biholomorphism, then $f$ extends to a holomorphic map $\C\p^1 \to \C\p^1$, as it cannot have essential singularities. The same holds for $f^{-1}$, so $f$ is just an element of $\PGL_2(\C)$, hence an algebraic automorphism of the projective complex line. Removing at least $4$ points of $\C\p^1$ (this is the case since $d\ge 3$) and moving one of them produces infinitely many curves with isomorphic complements, up to biholomorphism.
\end{proof}

\begin{corollary}\label{Coro:Holhighdimension}
For each $n\ge 2$, there exist algebraic hypersurfaces $E,F\subset \C^n$ which are complex manifolds that are not biholomorphic, but
have biholomorphic complements.
\end{corollary}

\begin{proof}
It suffices to take polynomials $f,g\in \C[x_1,x_2]$ provided by Proposition~\ref{Prop:Curvesholomorphic}, whose zero sets are smooth algebraic curves $C,D\subset \C^2$ that are not
biholomorphic, but
have holomorphic complements. We then use the same polynomials to define $E,F\subset \C^n$, which are smooth complex manifolds that have biholomorphic complements and are biholomorphic to $C\times \C^{n-2}$ and $D\times \C^{n-2}$ respectively. It remains to observe that $C\times \C^{n-2}$ and $D\times \C^{n-2}$ are not biholomorphic. Denote by $p_C\colon C\times \C^{n-2}\to C$ and $p_D\colon D\times \C^{n-2}\to D$ the projections on the first factor. If $\psi\colon \C^{n-2}\times C\to\C^{n-2}\times D $ is a biholomorphism, then $p_D\circ \psi\colon \C^{n-2}\times C\to D$ induces,  for each $c \in C$, a holomorphic map $\C^{n-2} \to D$ which must be constant by Picard's theorem (since it avoids at least two values of $\C$). Therefore, the map $p_D \circ \psi$ factors through a holomorphic map $\chi \colon C \to D$: we have $p_D \circ \varphi = \chi \circ p_C$. We analogously get a holomorphic map $\theta \colon D \to C$, which is by construction the inverse of $\chi$, so $C$ and $D$ are biholomorphic, a contradiction.
\end{proof}

\section*{Appendix: The case of $\p^2$}

\addtocounter{section}{1} 
\renewcommand{\thesection}{A}
\setcounter{lemma}{0}

In this appendix, we describe some results on the question of complements of curves in $\p^2$ explained in the introduction. These are not directly related to the rest of the text and serve only as comparison with the affine case.

We recall the following simple argument, known to specialists, for lack of reference:

\begin{proposition} \label{Prop:InP2Counterexamplessingular}
Let $C,D\subset \p^2$ be two geometrically irreducible closed curves such that $\p^2\setminus C$ and $\p^2\setminus D$ are isomorphic. If $C$ and $D$ are not equivalent, up to automorphism of $\p^2$, then $C$ and $D$ are singular rational curves.
\end{proposition}

\begin{proof}
Denote by $\varphi\colon \p^2\dasharrow \p^2$ a birational map which restricts to an isomorphism $\p^2\setminus C\iso\p^2\setminus D$. If $\varphi$ is an automorphism of $\p^2$, then $C$ and $D$ are equivalent. Otherwise, the same argument as in Proposition~\ref{Prop:Threecasescontraction} shows that both $C$ and $D$ are rational (this also follows from \cite[Lemma 2.2]{Bla09}). If $C$ and $D$ are singular, we are done, so we may assume that one of them is smooth, and then has degree $1$ or $2$. Since the Picard group of $\p^2 \setminus C$ is $\Z/\deg(C)\Z$,
we find
that $C$ and $D$ have the same degree. This implies that $C$ and $D$ are equivalent under automorphisms of $\p^2$. The case of lines is obvious. For conics, it is enough to check that a rational conic over any field is necessarily equivalent to the conic of equation $xy+ z^2 =0$. Actually,
we may always assume that the rational conic contains the point $[1:0:0]$, since it contains a rational point.
We may furthermore assume that the tangent at this point has equation $y=0$. This means that the equation of the conic is of the form $xy + u (y,z)$, where $u$ is a homogenous polynomial of degree $2$. Using a change of variables of the form $(x,y,z) \mapsto (x + a y + bz, y,z)$, where $a,b\in\k$,
we may assume that the equation is of the form $xy + c z^2 = 0$, where $c\in \k^*$. Then, using the change of variables $(x,y,z) \mapsto (cx, y,z)$,
we finally get, as announced, the equation $xy + z^2 = 0$.
\end{proof}

In order to get families of (singular) curves in $\p^2$ that have isomorphic complements, we here give explicit equations from the construction of Paolo Costa~\cite{Cos12}. We thus obtain unicuspidal curves in $\p^2$ which have isomorphic
complements, but which are non-equivalent under the action of $\Aut(\p^2)$.
We give the details of the proof for self-containedness, and also because the results below are not explicitly stated in~\cite{Cos12}.

\begin{lemma}\label{Lemm:Costa}
Let $\k$ be a field. 
Let $d\ge 1$ be an integer and $P \in \k [x,y]$ a homogenous polynomial of degree $d$, not a multiple of $y$. We define the homogeneous polynomial $f_P\in \k[x,y,z]$ of degree $4d+1$ by the following formula, where $w:=xz-y^2$:
\[f_P=zw^{2d}+2y w^d P(x^2,w)+x P^2 (x^2,w). \]
Denote by $C_P, \mathcal{L}, \mathcal{Q} \subset \p^2$ the curves of equations $f_P= 0$, resp.~$z=0$, resp.~$w=0$, and by
$V_P, V_\mathcal{L}, V_\mathcal{Q} \subset \A^3$ their corresponding cones $($given by the same equations$)$. Then:

\begin{enumerate}[$(1)$]
\item\label{f_P is irreducible}
The polynomial $f_P$ is geometrically irreducible $($i.e.~irreducible in $\kk[x,y,z])$.
\item\label{psiPautoQ}
The rational map $\psi_P \colon \A^3 \dasharrow \A^3$ which sends $(x,y,z)$ to
\[ \Big(x,y+xP \left( x^2 w^{-1},1 \right) ,z+2yP \left( x^2w^{-1},1 \right) +x P^2 \left( x^2w^{-1},1\right) \Big) \]
is a birational map of $\A^3$ that restricts to isomorphisms 
\[\A^3\setminus V_{\mathcal{Q}}\iso\A^3\setminus V_{\mathcal{Q}},\ V_P\setminus V_\mathcal{Q}\iso V_\mathcal{L}\setminus V_\mathcal{Q}\text{ and }\A^3\setminus (V_\mathcal{Q}\cup V_P)\iso \A^3\setminus (V_\mathcal{Q}\cup V_\mathcal{L}).\]

Since $\psi_P$ is homogeneous, the same formula induces a birational map of $\p^2$ that restricts to isomorphisms 
\[\p^2\setminus\mathcal{Q}\iso \p^2\setminus\mathcal{Q},\  C_P\setminus \mathcal{Q}\iso \mathcal{L}\setminus \mathcal{Q}\text{ and }
\p^2\setminus (\mathcal{Q}\cup C_P)\iso \p^2  \setminus (\mathcal{Q}\cup \mathcal{L}).\]
Since the point $[0:0:1]$ is the unique intersection point between $C_P$ and $\mathcal{Q}$, it is also the unique singular point of $C_P$.

\item\label{varphilambdatheta}
Let $\lambda$ be a nonzero element of $\k$. Then, the rational map
\[\varphi_{\lambda}\colon (x,y,z)\mapsto \left(x+(\lambda-1) wz^{-1},  y, z\right) = (\lambda x - (\lambda-1) y^2 z^{-1}, y,z) \]
is a birational map of $\A^3$ that restricts to automorphisms of $\A^3\setminus V_\mathcal{L}$, $V_{\mathcal{Q}}\setminus V_\mathcal{L}$ and $\A^3\setminus (V_\mathcal{L}\cup V_\mathcal{Q})$. The same formula then gives automorphisms of $\p^2\setminus \mathcal{L}$, $\mathcal{Q}\setminus \mathcal{L}$ and $\p^2\setminus (\mathcal{L}\cup \mathcal{Q})$.

\item\label{PQphi}
Set $\tilde{P}(x,y)=P(\lambda x,y)$ and $\kappa=(\psi_{\tilde{P}})^{-1}\varphi_\lambda \psi_P$. Then, the rational map $\kappa$ restricts to an isomorphism $\A^3 \setminus V_{P}\iso \A^3\setminus V_{\tilde{P}}$.
In particular, $\kappa$ also induces an isomorphism $\p^2\setminus C_P\iso \p^2\setminus C_{\tilde{P}}$.
\item\label{LinPPprime}
For each homogeneous polynomial $\tilde{P} \in \k [x,y]$ of degree $d$ which is not divisible by $y$,
the curves $C_P$ and $C_{\tilde{P}}$ are equivalent up to automorphisms of $\p^2$, if and only if there exist some constants $\rho \in \k^*,\mu \in \k$ such that
\[ \tilde{P}(x,y)=\rho P( \rho^2 x,y)+\mu y^d.\]
\end{enumerate}
\end{lemma}

\begin{proof}
$(\ref{f_P is irreducible})$-$(\ref{psiPautoQ})$: As does each rational map $\A^3\dasharrow \A^3$, the rational map $\psi_P$ supplies a morphism of $\k$-algebras $(\psi_P)^*\colon \k[x,y,z]\to \k(x,y,z)$. This sends $x,y,z$ onto $x,y+x P(x^2w^{-1},1), z+2y P(x^2 w^{-1},1)+ x P^2(x^2 w^{-1},1) $. Note that $(\psi_P)^{*}$ fixes $x$ and $w$. This implies that $(\psi_P)^{*}$ extends to an endomorphism of $\k[x,y,z,w^{-1}]$, which is moreover an automorphism since $(\psi_P)^{*}\circ (\psi_{-P})^{*}=\mathrm{id}$. Extending to the quotient field $\k(x,y,z)$,
we get an automorphism of $\k(x,y,z)$, that we again denote by $(\psi_P)^*$, so $\psi_P$ is a birational map of $\A^3$ and induces moreover an isomorphism of $\A^3\setminus V_\mathcal{Q}$, because $(\psi_P)^{*}(\k[x,y,z,w^{-1} ])=\k[x,y,z,w^{-1} ]$.
We then observe
that $(\psi_P)^{*}(z)=f_Pw^{-2 d}$  where $f_P$ and $w=xz-y^2$ are coprime since $f_P(1,0,0) =P^2 (1,0) \not=0$. Let us also
notice
that $V_P \cap V_\mathcal{Q} = \{ (x,y,z) \in \A^3\;|\; x=y=0\}$ and that $V_\mathcal{L} \cap V_\mathcal{Q} =  \{ (x,y,z) \in \A^3\;|\; y=z=0\}$. Hence
$\psi_P$ restricts to an isomorphism of surfaces $V_P \setminus V_\mathcal{Q} \iso V_\mathcal{L} \setminus V_\mathcal{Q}$. This implies that $V_P$ and $C_P$ are rational, and that $f_P$ is geometrically irreducible, which proves $(\ref{f_P is irreducible})$. This also implies that $\psi_P$ restricts to an isomorphism $\A^3\setminus (V_\mathcal{Q}\cup V_P)\iso \A^3\setminus (V_\mathcal{Q}\cup V_\mathcal{L} )$. As $\psi_P$ is homogeneous, we get the analogous results by replacing $\A^3$, $V_P$, $V_\mathcal{L}$, $V_\mathcal{Q}$ by $\p^2$, $C_P$, $\mathcal{L}$, $\mathcal{Q}$ respectively.

$(\ref{varphilambdatheta})$:
We check that $\varphi_{\lambda}\circ \varphi_{\lambda^{-1}}=\mathrm{id}$, so $\varphi_{\lambda}$ is a birational map of $\A^3$, which  restricts to an automorphism of $\A^3\setminus V_\mathcal{L}$, since the denominators only involve $z$. Moreover, $(\varphi_{\lambda})^{*}(w)=\lambda w$ (where $(\varphi_{\lambda})^{*}$ is the automorphism of $\k(x,y,z)$ corresponding to $\varphi_{\lambda}$), so the surface $V_\mathcal{Q}\setminus V_\mathcal{L}$ is preserved, hence $\varphi_{\lambda}$ restricts to automorphisms of $\A^3\setminus V_\mathcal{L}$, $V_{\mathcal{Q}}\setminus V_\mathcal{L}$ and $\A^3\setminus (V_\mathcal{L}\cup V_\mathcal{Q})$. Since $\varphi_{\lambda}$ is homogeneous, the same formula then gives automorphisms of $\p^2\setminus \mathcal{L}$, $\mathcal{Q}\setminus \mathcal{L}$ and $\p^2\setminus (\mathcal{L}\cup \mathcal{Q})$.

$(\ref{PQphi})$:
By $(\ref{psiPautoQ})$-$(\ref{varphilambdatheta})$, the transformation $\kappa=(\psi_{\tilde{P}})^{-1}\varphi_\lambda \psi_P$  restricts to an isomorphism $\A^3\setminus (V_\mathcal{Q}\cup V_P) \iso \A^3\setminus(V_\mathcal{Q} \cup V_{\tilde{P}})$. Let us prove that with the special choice of $\tilde{P}$ that we have made,  $\kappa$ then restricts to an isomorphism $\A^3\setminus V_P\iso\A^3\setminus V_{\tilde{P}}$. For this, we prove that the restriction of $\kappa$ is the identity automorphism on $V_{\mathcal{Q}}\setminus V_P=V_{\mathcal{Q}}\setminus V_{\tilde{P}}=V_{\mathcal{Q}} \setminus \{ (x,y,z) \in \A^3\;|\; x=y=0 \}$. We compute 
\[\varphi_{\lambda}\psi_P(x,y,z)=\left(x+ (\lambda-1) w^{2d+1} f_P^{-1}, y+x P(x^2,w)w^{-d}, f_P w^{-2d} \right)\]
which satisfies $(\varphi_{\lambda}\psi_P)^*(w)=(\varphi_{\lambda})^{*}(w)=\lambda w$. To simplify the notation, we write $\delta=(\lambda-1)  w^{2d+1} f_P^{-1}$ and  get that $\kappa(x,y,z)=(\psi_{\tilde{P}})^{-1}\varphi_\lambda\psi_P(x,y,z)$ is equal to 
\[\left(x+\delta,y+x P(x^2,w)w^{-d}-(x+\delta )\tilde{P} \left( \lambda^{-1} (x+\delta)^2  w^{-1},1 \right), z+\zeta \right)\]
for some $\zeta \in \k(x,y,z)$. Since $\tilde{P}(x,y)=P(\lambda x,y)$, the second component is \[\kappa^{*}(y)=y+\frac{xP(x^2,w)-P((x+\delta)^2,w)(x+\delta)}{w^d}.\]
As $w^{d+1}$ divides the numerator of $\delta$,
we can write $\kappa^*(y)$ as $y+w (f_P)^{-n}R$, for some $R\in \k[x,y,z]$ and $n\ge 0$. Similarly, $\kappa^*(x)=x+w f_P^{-1} S$, where $S\in \k[x,y,z]$. Since $\kappa^*(w)=\lambda w$, we get
\[ \lambda w =(x+w f_P^{-1} S)(z+\zeta)-(y+w f_P^{-n} R)^2,\]
which shows that $\zeta (x+wf_P^{-1} S)= w f_P^{- \tilde{m}} \tilde{T}$ for some $\tilde{T}\in \k[x,y,z]$, $\tilde{m}\ge 0$. Hence
we can write $\kappa^*(z)=z+\zeta=z+w f_P^{-m} T$ for some $T\in \k[x,y,z]$ and $m\ge 0$. This shows that $\kappa$ is well defined on $V_{\mathcal{Q}}\setminus V_P=V_{\mathcal{Q}}\setminus V_{\tilde{P}}=V_{\mathcal{Q}}\setminus \{ (x,y,z) \in \A^3\;|\; x=y=0 \}$ and restricts to the identity on this surface.

Since $\kappa$ is homogeneous, the isomorphism $\A^3 \setminus V_{P}\iso \A^3\setminus V_{\tilde{P}}$ also induces an isomorphism $\p^2\setminus C_P\iso \p^2\setminus C_{\tilde{P}}$, which fixes pointwise the curve $\mathcal{Q}\setminus C_P=\mathcal{Q}\setminus C_{\tilde{P}}$.

$(\ref{LinPPprime})$:
Suppose first that $\tilde{P}(x,y)=\rho P( \rho^2 x, y)+\mu y^d$ for some $\rho\in \k^*,\mu \in \k$.
Define the transformation $\alpha \in \GL_3(\k)$ by
\[ \alpha (x,y,z) = (x,\rho y-\mu x,\rho ^2z-2\rho \mu y+\mu^2x) \]
and the birational transformation $s \in \mathrm{Bir}(\A^3)$ by $s=\psi_{\tilde{P}} \alpha (\psi_P)^{-1}$. Let us note that $s^* =  (\psi_{P}^*)^{-1} \alpha^{*}  \psi_{\tilde{P}} ^*$.
We check that $\alpha^* (w) = \rho^2 w$, from which we get $s^*(w) = \rho^2 w$.
The equality
\[\begin{array}{rcl}
\alpha^{*}  ( \psi_{\tilde{P}} ^*(y) )&=&\alpha^{*}(y+x\tilde{P}(x^2 w^{-1},1))=\rho y-\mu x+x \tilde{P}(\rho^{-2} x^2 w^{-1}, 1)\\
&=&\rho y+\rho xP(x^2w^{-1},1)=\rho \,  \psi_{P}^*(y) \end{array}\]
gives us $s^* (y) = \rho y$. The relation $z = x^{-1} (w-y^2)$
combined with the  equality $s^*(x) = x$ now proves that $s^* (z) = \rho^2 z$. But we have $(\psi_P)^{*}(z)=f_Pw^{-2 d}$ and $(\psi_{\tilde{P}} )^{*}(z)=f_{\tilde{P}} w^{-2 d}$, so that  we get $\alpha^* (f_{\tilde{P}} w^{-2 d}) = \rho^2  f_Pw^{-2 d}$. In turn, this latter equality  yields
\[ \alpha^* (f_{\tilde{P}} ) = \rho^{4d+2} f_P.\]
This shows that $\alpha$ induces an automorphism of $\p^2$ sending $C_P$ onto $C_{\tilde{P}}$.

Conversely, suppose that there exists $\tau\in \Aut(\p^2)$ sending $C_P$ onto $C_{\tilde{P}}$.

We begin by proving that $\tau$ preserves the conic $\mathcal{Q}$.
Since $C_{P}\setminus \mathcal{Q}\simeq C_{\tilde{P}}\setminus \mathcal{Q}\simeq L\setminus \mathcal{Q}\simeq \A^1$, the irreducible conic $\mathcal{Q}\subset \p^2$ intersects $C_P$ (respectively $C_{\tilde P}$) in exactly one $\kk$-point, the unique singular point $[0:0:1]$ of $C_P$ (resp.~$C_{\tilde P}$). The irreducible conic $\tau(\mathcal{Q})$ thus also intersects  $C_{\tilde{P}}$ in one $\kk$-point, namely $[0:0:1]$. Observe that this implies that $\tau(\mathcal{Q})=\mathcal{Q}$.
We first notice that $C_{\tilde{P}}\setminus \{ [0:0:1] \} \simeq \A^1$, so there is one $\k$-point at each step of the resolution of $C_{\tilde{P}}$. We can then write $q_1=[0:0:1]$ and define a sequence of points $(q_i)_{i\ge 1}$ such that $q_i$ is the point infinitely near $q_{i-1}$ belonging to the strict transform of $C_{\tilde{P}}$, for each $i\ge 2$. Denote by $r$ the biggest integer such that $q_r$ belongs to the strict transform of $\mathcal{Q}$ and by $r'$ the biggest integer such that $q_{r'}$ belongs to the strict transform of $\tau ( \mathcal{Q} )$.  By B\'ezout's Theorem (since $\mathcal{Q}$ and $\tau(\mathcal{Q})$ are smooth), we have
\[\sum_{i=1}^r m_{q_i}(C_{\tilde{P}})=\deg(\mathcal{Q})\deg(C_{\tilde{P}})=\deg(\tau(\mathcal{Q}))\deg(C_{\tilde{P}})=\sum_{i=1}^{r'} m_{q_i}(C_{\tilde{P}}),\]
which yields $r=r'$. On the blow-up $X\to \p^2$ of $q_1,\dots,q_r$, the strict transform of the curve $C_{\tilde{P}}$ is then disjoint from those of $\mathcal{Q}$ and $\tau(\mathcal{Q})$, which are linearly equivalent.
Assume by contradiction that we have $\tau(\mathcal{Q})\not=\mathcal{Q}$.
Then, we claim that the strict transform of any irreducible conic $\mathcal{Q}'$ in the pencil generated by $\mathcal{Q}$ and $\tau(\mathcal{Q})$ is also disjoint from the strict transform of $C_{\tilde{P}}$. Indeed, we first note that  $C_{\tilde{P}}$ and $\mathcal{Q}'$ have no common irreducible component since $C_{\tilde{P}}$ is an irreducible curve whose degree satisfies
\[ \deg C_{\tilde{P}} \geq 5 > 2= \deg \mathcal{Q}'.\]
Finally, since the (infinitely near) points $q_1, \ldots, q_r$ belong to both $\mathcal{Q}'$ and $C_{\tilde{P}}$ and since $\sum_{i=1}^r m_{q_i}(C_{\tilde{P}})=\deg(\mathcal{Q}')\deg(C_{\tilde{P}})$, the curves $\mathcal{Q}'$ and $C_{\tilde{P}}$ do not have any other common (infinitely near) point.

Choose now a general point $q$ of $\p^2$ which belongs to $C_{\tilde{P}} \setminus \{ q_1 \}\simeq \A^1$ and choose the conic $\mathcal{Q}'$ in the pencil generated by $\mathcal{Q}$ and $\tau(\mathcal{Q})$ which passes through $q$. Then, the strict transforms of $\mathcal{Q}'$ and $C_{\tilde{P}}$ intersect in $X$ (at the point $q$).
This contradiction shows that $\mathcal{Q}$ is preserved by $\tau$.

Since $\tau \in \Aut(\p^2)=\PGL_3(\k)$ fixes the point $[0:0:1]$ (which is the unique singular point of both $C_P$ and $C_{\tilde{P}}$) and preserves the line $x=0$ (which is the tangent line of both  $C_P$ and $C_{\tilde{P}}$ at the point $[0:0:1]$), it admits a (unique) lift $\alpha\in \GL_3(\k)$ which is triangular and satisfies $\alpha^* (x) = x$. This means that $\alpha$ is of the form:
\[ \alpha \colon (x,y,z) \mapsto (x, \rho y -\mu x, \gamma z + \delta y + \varepsilon x), \]
for some constants $\rho, \mu, \gamma, \delta, \varepsilon  \in \k$ (satisfying $\rho \gamma \neq 0$). Since $\alpha ^* (w)$ is proportional to $w$, we get $\gamma = \rho^2$, $\delta = -2 \rho \mu$ and $\varepsilon = \mu^2$, i.e.~$\alpha$ is of the form 
\[ \alpha \colon (x,y,z) \mapsto (x, \rho y -\mu x,  \rho ^2  z - 2 \rho \mu y + \mu^2 x). \]

Set $s:=\psi_{\tilde{P}} \alpha (\psi_P)^{-1}\in \mathrm{Bir}(\A^3)$. Since $\alpha^* (w) = \rho^2 w$, we also get $s^*(w) = \rho^2 w$. Since $(\psi_P)^{*}(z)=f_Pw^{-2 d}$, $(\psi_{\tilde{P}} )^{*}(z)=f_{\tilde{P}} w^{-2 d}$ and since $\alpha^*( f_{\tilde {P} })$ and $f_P$ are proportional, the fractions $s^*(z)$ and $z$ are also proportional. Therefore, there exists a nonzero constant $\xi \in \k$ such that 
\begin{equation} \label{sstar}s^{*}(x)=x, \quad s^{*}(w)=\rho^2 w, \quad s^{*}(z)=\xi z. \end{equation}

 Moreover, $s$ induces a birational map $\hat{s}$ of $\p^2$ which is an automorphism of $\p^2\setminus \mathcal{Q}$, because the same holds for $\alpha$, $\psi_P$ and $\psi_{\tilde{P}}$. Let us observe that $\hat{s}$ is in fact an automorphism of $\p^2$. Indeed, otherwise $\hat{s}$ would contract $\mathcal{Q}$ to one point.
This is impossible: Since $\hat{s}$ preserves the two pencils of conics given by $[x:y:z]\mapsto [w :x^2]$ and $[x:y:z]\mapsto [w:z^2]$, which have distinct base-points $[0:0:1]$ and $[1:0:0]$, these base-points are fixed by $\hat{s}$.
Hence, there exist some constants $\zeta, \eta, \theta \in \k$ such that $s^* (y) = \zeta x + \eta y + \theta z$. Hence (\ref{sstar}) gives us $\zeta= \theta =0$, i.e.~$s^* (y) = \eta y$. But the equality $s=\psi_{\tilde{P}} \alpha (\psi_P)^{-1}$ is equivalent to $\psi_{\tilde{P}} \alpha= s\psi_P$ and by taking the second coordinate we get
\[( \rho y-\mu x)+x\tilde{P} \left(  \rho^{-2} x^2 w^{-1},1 \right) =(\psi_{\tilde{P}} \alpha)^*(y)=(s\psi_P)^{*}(y)=\eta \left( y+xP\left(x^2 w^{-1},1 \right) \right) \]
which yields $\rho=\eta$ and $\tilde{P}( \rho^{-2} x^2w^{-1}, 1)=\rho P(x^2 w^{-1},1)+\mu$.
By substituting $\rho^{-2}y + x^{-1} y^2$ for $z$ and by noting that $w(x,y, \rho^{-2}y + x^{-1} y^2) = \rho^{-2} xy$, we obtain $\tilde{P}(xy^{-1},1)=\rho P(\rho^2 xy^{-1},1)+\mu$, which is equivalent to $\tilde{P}(x,y)=\rho P(\rho^2 x,y)+\mu y^d$, as we required. 
\end{proof}

The construction of Lemma~\ref{Lemm:Costa} yields, for each $d\ge 1$, families of curves of degree $4d+1$  having isomorphic complements. These are equivalent for $d=1$, at least when $\k$ is algebraically closed (Lemma~\ref{Lemm:Costa}$(\ref{LinPPprime})$), but not for $d\ge 2$.
We can now easily provide explicit examples:

\begin{proposition}\label{Prop:CostaFamilies}
Let $d\ge 2$ be an integer. Set $P=x^d+x^{d-1}y$ and $w= xz-y^2 \in \k[x,y]$. All curves of $\p^2$ given by
\[z w^{2d}+2yw^d P(\lambda x^2,w)+x P^2( \lambda x^2,w)  =0 \]
for $\lambda\in \k^*$, have isomorphic complements and are pairwise not equivalent up to automorphisms of $\p^2$.
\end{proposition}

\begin{proof}
The curves correspond to the curves $C_{P(\lambda x,y)}$ of Lemma~\ref{Lemm:Costa} and thus have isomorphic complements by Lemma~\ref{Lemm:Costa}$(\ref{PQphi})$. It remains to show that if $C_{P(\lambda x,y)}$ is equivalent to $C_{P(\tilde\lambda x,y)}$, then $\lambda=\tilde\lambda$. Lemma~\ref{Lemm:Costa}$(\ref{PQphi})$ yields the existence of $\rho\in \k^*,\mu \in \k$ such that $P(\tilde\lambda x,y)=\rho P(\rho^2\lambda x,y)+\mu y^d$. Since $d\ge 2$, both $P(\tilde\lambda x,y)$ and $\rho P( \rho^2\lambda x,y)$ do not have component with $y^d$, so $\mu=0$. We then compare the coefficients of $x^d$ and $x^{d-1}y$ and get
\[\tilde\lambda^d=\rho (\rho^2\lambda)^d,\ \ \tilde\lambda^{d-1}=\rho (\rho^2\lambda)^{d-1},\]
which yields $\tilde\lambda=\rho^2 \lambda$,  whence $\rho=1$ and $\tilde\lambda=\lambda$ as  desired.
\end{proof}

\end{document}